\documentclass[11pt]{amsart}

\usepackage{a4wide}
\usepackage{amssymb}
\usepackage{mathrsfs}
\usepackage{enumerate}
\usepackage{esint}
\usepackage{titletoc}
\usepackage[colorlinks=true, urlcolor=red, linkcolor=red, citecolor=blue]{hyperref}
\usepackage[nameinlink]{cleveref}

\theoremstyle{plain}
\newtheorem{theorem}{Theorem}[section]
\newtheorem{lemma}[theorem]{Lemma}
\newtheorem{corollary}[theorem]{Corollary}

\theoremstyle{definition}
\newtheorem{definition}[theorem]{Definition}
\newtheorem{remark}[theorem]{Remark}

\numberwithin{equation}{section}

\DeclareMathOperator*{\osc}{osc}
\DeclareMathOperator*{\essosc}{ess\,osc}
\DeclareMathOperator*{\esssup}{ess\,sup}
\DeclareMathOperator*{\essinf}{ess\,inf}
\DeclareMathOperator*{\essliminf}{ess\,liminf}

\makeatletter
\@namedef{subjclassname@2020}{\textup{2020} Mathematics Subject Classification}
\makeatother

\title[Supersolutions and superharmonic functions for nonlocal operators]{Supersolutions and superharmonic functions for nonlocal operators with Orlicz growth}

\author{Minhyun Kim}
\address{Department of Mathematics \& Research Institute for Natural Sciences, Hanyang University, 04763 Seoul, Republic of Korea}
\email{minhyun@hanyang.ac.kr}

\author{Se-Chan Lee}
\address{Department of Mathematical Sciences, Seoul National University, Seoul 08826, Republic of Korea}
\email{dltpcks1@snu.ac.kr}

\thanks{M. Kim is supported by the National Research Foundation of Korea (NRF) grant funded by the Korean government (MSIT) (RS-2023-00252297). S.-C. Lee is supported by Basic Science Research Program through the National Research Foundation of Korea (NRF) funded by the Ministry of Education (2022R1A6A3A01086546).}

\begin{document}

\begin{abstract}
We study supersolutions and superharmonic functions related to problems involving nonlocal operators with Orlicz growth, which are crucial tools for the development of nonlocal nonlinear potential theory. We provide several fine properties of supersolutions and superharmonic functions, and reveal the relation between them. Along the way we prove some results for nonlocal obstacle problems such as the well-posedness and (both interior and boundary) regularity estimates, which are of independent interest.
\end{abstract}

\maketitle

%%%%%%%%%%%%%%%%%%%%%%%%%%%%%%%%%%%%%%%
\section{Introduction}
%%%%%%%%%%%%%%%%%%%%%%%%%%%%%%%%%%%%%%%

Nonlocal nonlinear potential theory (NNPT) is the study of harmonic functions associated with nonlocal nonlinear operators. It was initiated by Kuusi--Mingione--Sire~\cite{KMS15} and its basis has been established by Korvenp\"a\"a--Kuusi--Palatucci~\cite{KKP17}, Lindgren–Lindqvist~\cite{LL17}, and Korvenp\"a\"a--Kuusi--Lindgren~\cite{KKL19}. As in the classical nonlinear potential theory (NPT), it turns out that it is crucial to study supersolutions and superharmonic functions in NNPT. Indeed, the Perron method was developed in \cite{KKP17,LL17} by using superharmonic functions. Moreover, the Wiener criterion was recently obtained by Kim--Lee--Lee~\cite{KLL23} based on these results. 

The nonlinearity governing the nonlocal problems in the aforementioned papers is given by a polynomial of degree $p>1$. However, as in the classical NPT, various nonlinearities can come into play in NNPT. For instance, one can consider nonlocal problems with Orlicz growth, variable exponent, double phase structure, and so on. In this work, we focus on the nonlinearity described by an Orlicz function. To the best of our knowledge, NNPT with Orlicz growth has not been developed yet. The aim of this paper is to present a comprehensive study on supersolutions and superharmonic functions associated with nonlocal operators with Orlicz growth as a first step toward the development of NNPT with Orlicz growth.

The nonlocal operator we are interested in is defined by
\begin{equation*}
	\mathcal{L}u(x) = 2\,\mathrm{p.v.} \int_{\mathbb{R}^{n}} g(|D^su|) \frac{D^su}{|D^su|} \frac{k(x, y)}{|x-y|^{n+s}} \,\mathrm{d}y,
\end{equation*}
where $g$ is given by the derivative of a Young function $G$ (see \Cref{sec-preliminaries} for the definition), 
\begin{equation*}
D^su = D^su(x, y) = \frac{u(x)-u(y)}{|x-y|^s}
\end{equation*}
denotes a \emph{fractional gradient} of $u$ of order $s \in (0, 1)$, and $k: \mathbb{R}^{n} \times \mathbb{R}^{n} \to [0, \infty]$ is a measurable function satisfying $\Lambda^{-1} \leq k(x, y) = k(y, x) \leq \Lambda$ for a.e.\ $x, y \in \mathbb{R}^{n}$ for some ellipticity constant $\Lambda \geq 1$. Note that if $G$ has the standard growth, i.e., $G(t)=t^p$ for some $p>1$, then our study falls into the rather standard case of fractional $p$-Laplacian-type operators. Let us first provide the definition of a supersolution with respect to $\mathcal{L}$. We refer the reader to \Cref{sec-preliminaries} for the definitions of function spaces.

\begin{definition}\label{def-supersolution}
A function $u \in W_{\text{loc}}^{s, G}(\Omega)$ with $u_- \in L^g_s(\mathbb{R}^n)$ is a \emph{(weak) supersolution} of $\mathcal{L}u=0$ in $\Omega$ if
\begin{equation}\label{eq-supersolution}
\int_{\mathbb{R}^n} \int_{\mathbb{R}^n} g(|D^su|) \frac{D^su}{|D^su|} D^s\varphi \frac{k(x, y)}{|x-y|^n} \,\mathrm{d}y \,\mathrm{d}x \geq 0
\end{equation}
for all nonnegative functions $\varphi \in C_c^\infty(\Omega)$. A function $u$ is a \textit{(weak) subsolution} if $-u$ is a supersolution, and $u$ is a \textit{(weak) solution} if it is both a subsolution and a supersolution.
\end{definition}

Let us next define a superharmonic function with respect to $\mathcal{L}$. While a supersolution is defined by using the energy corresponding to the operator $\mathcal{L}$, a superharmonic function is defined by means of the comparison principle.

\begin{definition}\label{def-superharmonic}
A measurable function $u: \mathbb{R}^n \to [-\infty, \infty]$ is a ($\mathcal{L}$-)\emph{superharmonic function} in $\Omega$ if it satisfies the following:
\begin{enumerate}[(i)]
\item $u < \infty$ a.e.\ in $\mathbb{R}^n$ and $u>-\infty$ everywhere in $\Omega$,
\item $u$ is lower semicontinuous in $\Omega$,
\item if $D \Subset \Omega$ is an open set and $v \in C(\overline{D})$ is a solution of $\mathcal{L}v=0$ in $D$ with $v_+ \in L^{\infty}(\mathbb{R}^n)$ such that $u \geq v$ on $\partial D$ and almost everywhere on $\mathbb{R}^n \setminus D$, then $u \geq v$ in $D$,
\item $u_- \in L_{s}^{g}(\mathbb{R}^n)$.
\end{enumerate}
\end{definition}

The first part of this paper is devoted to local regularity estimates of solutions up to the boundary. More precisely, in \Cref{sec-loc-est}, we prove the local boundedness of subsolutions (\Cref{thm-loc-bdd-int} and \Cref{thm-loc-bdd-bdry}) and the weak Harnack inequality for supersolutions (\Cref{thm-WHI-int} and \Cref{thm-WHI-bdry}) both in the interior and up to the boundary of the domain. For contributions on such results for the fractional $p$-Laplacian-type operators, we refer the reader to Di Castro--Kuusi--Palatucci~\cite{DCKP14,DCKP16}, Korvenp\"a\"a--Kuusi--Palatucci~\cite{KKP16}, Kim--Lee--Lee~\cite{KLL23}, and references therein. In the the general framework with Orlicz growth, the theory seems rather incomplete; only interior estimates were recently obtained in Chaker--Kim--Weidner~\cite{CKW22,CKW23} and Byun--Kim--Ok~\cite{BKO23} by the approach of De Giorgi. Our contribution in this section is two-fold: on the one hand, we provide the local estimates up to the boundary. On the other hand, we develop Moser's iteration method. It allows us to obtain the weak Harnack inequalities (\Cref{thm-WHI-int} and \Cref{thm-WHI-bdry}) with an improved exponent compared to the known result in \cite{CKW23}. We emphasize that our exponent is sharp in the sense that it recovers the optimal exponent for the weak Harnark inequality \cite{DCKP14,KLL23} for the fractional $p$-Laplacian when $G(t)=t^p$.

We then move our attention to the nonlocal obstacle problems in \Cref{sec-obstacle}. The analysis on the obstacle problem plays a crucial role in capturing behavior of superharmonic functions, see Heinonen--Kilpel\"ainen--Martio~\cite{HKM06} and references therein for the $p$-Laplacian-type operators and Korvenp\"a\"a--Kuusi--Palatucci~\cite{KKP16} for the fractional $p$-Laplacian-type operators. We prove the existence and uniqueness of the solution to the nonlocal obstacle problem (\Cref{thm-obstacle}) in a certain admissible class of functions. Moreover, we derive the interior and boundary regularity of solutions to the nonlocal obstacle problems by using the techniques developed in \Cref{sec-loc-est}.

These results can be used for studying the Perron method, capacity potential, and balayage, which are important components in NNPT. Furthermore, since the Dirichlet problem is a special case of the obstacle problem, our results can be immediately applied to solutions of the Dirichlet problem. We also mention that the nonlocal obstacle problems are of independent interest as they have significant applications in various contexts. Just to name a few, the nonlocal obstacle problems appear in the interacting particle systems \cite{BCLR13,CDM16}, the thin obstacle problem \cite{ACS08,Sig59}, and the pricing model for American options \cite{CT04,Mer76}; we also refer to \cite{FR22,FRRO23} for an exhaustive discussion. In the special case when $G(t)=t^2$, the corresponding (linear) fractional obstacle problem was treated in the celebrated papers \cite{CF13,CSS08,Sil07}.

The last part of this paper consists of the properties of supersolutions and superharmonic functions. In \Cref{sec-supersolution}, we prove the lower semicontinuity, comparison principle, and convergence results for supersolutions. Based on such results we then prove \Cref{thm-relation}, which explains the relation between supersolutions and superharmonic functions. It is noteworthy that this relation between two function classes is an important building block in the development of NPT. See, for instance, Kilpel\"ainen--Mal\'y~\cite{KM94} and Kim--Lee--Lee~\cite{KLL23} for the Wolff potential estimates and the Wiener criterion. Finally, we provide the pointwise behavior of superharmonic functions and prove some integrability results in the rest of \Cref{sec-superharmonic}.

%%%%%%%%%%%%%%%%%%%%%%%%%%%%%%%%%%%%%%%
\section{Preliminaries}\label{sec-preliminaries}
%%%%%%%%%%%%%%%%%%%%%%%%%%%%%%%%%%%%%%%

%%%%%%%%%%%%%%%%%%%%%%%%%%%%%%%%%%%%%%%
\subsection{Growth function}
%%%%%%%%%%%%%%%%%%%%%%%%%%%%%%%%%%%%%%%

A function $G: [0, \infty) \to [0, \infty)$ is called a {\it Young function} if it has the form
\begin{equation*}
G(t) = \int_0^t g(\tau) \,\mathrm{d}\tau
\end{equation*}
for some non-decreasing, left-continuous function $g: [0, \infty) \to [0, \infty)$ which is not identically equal to 0. Clearly, a Young function is convex. Throughout the paper, we assume that $G$ is a differentiable Young function, with a strictly increasing function $g=G'$, and there exist $1<p\leq q$ such that
\begin{equation}\label{eq-pq}
	pG(t) \leq tg(t) \leq qG(t).
\end{equation}
Without loss of generality, we may assume that $G(1)=1$.

It is sometimes more convenient to use a function
\begin{equation}\label{eq-bar-g}
\bar{g}(t) = G(t)/t
\end{equation}
instead of $g$ because $\bar{g}$ satisfies a growth condition
\begin{equation}\label{eq-bar-g-pq}
(p-1)\bar{g}(t) \leq t\bar{g}'(t) \leq (q-1)\bar{g}(t).
\end{equation}
We emphasize that the condition \eqref{eq-bar-g-pq} with $\bar{g}$ replaced by $g$ does not follow from the assumption \eqref{eq-pq}. Note that $\bar{g}$ satisfies
\begin{equation}\label{eq-bar-g-comp}
p\bar{g}(t) \leq g(t) \leq q\bar{g}(t) \quad\text{and}\quad p^{1/(q-1)} g^{-1}(t) \leq \bar{g}^{-1}(t) \leq q^{1/(p-1)} g^{-1}(t).
\end{equation}
Moreover, it is easy to see that $G$, $g$ and $\bar{g}$ satisfy the following properties; see Chaker--Kim--Weidner~\cite{CKW22,CKW23} for instance.

\begin{lemma}\label{lem-G}
Let $t, t' \geq 0$ and $c=p^{1/(q-1)}/q^{1/(p-1)}$.
\begin{enumerate}[(i)]
\item
For all $\lambda \geq 1$, it holds that
\begin{alignat*}{2}
&\lambda^p G(t) \leq G(\lambda t) \leq \lambda^q G(t),
&\quad&\lambda^{1/q} G^{-1}(t) \leq G^{-1}(\lambda t) \leq \lambda^{1/p}G^{-1}(t), \\
&\lambda^{p-1} \bar{g}(t) \leq \bar{g}(\lambda t) \leq \lambda^{q-1} \bar{g}(t),
&&\lambda^{1/(q-1)} \bar{g}^{-1}(t) \leq \bar{g}^{-1}(\lambda t) \leq \lambda^{1/(p-1)} \bar{g}^{-1}(t), \\
&\frac{p}{q} \lambda^{p-1} g(t) \leq g(\lambda t) \leq \frac{q}{p} \lambda^{q-1} g(t),
&&c \lambda^{1/(q-1)} g^{-1}(t) \leq g^{-1}(\lambda t) \leq c^{-1} \lambda^{1/(p-1)} g^{-1}(t).
\end{alignat*}
\item
For all $\lambda \leq 1$, it holds that
\begin{alignat*}{2}
&\lambda^q G(t) \leq G(\lambda t) \leq \lambda^p G(t),
&\quad&\lambda^{1/p} G^{-1}(t) \leq G^{-1}(\lambda t) \leq \lambda^{1/q}G^{-1}(t), \\
&\lambda^{q-1} \bar{g}(t) \leq \bar{g}(\lambda t) \leq \lambda^{p-1} \bar{g}(t),
&&\lambda^{1/(p-1)} \bar{g}^{-1}(t) \leq \bar{g}^{-1}(\lambda t) \leq \lambda^{1/(q-1)} \bar{g}^{-1}(t), \\
&\frac{p}{q} \lambda^{q-1} g(t) \leq g(\lambda t) \leq \frac{q}{p} \lambda^{p-1} g(t),
&&c \lambda^{1/(p-1)} g^{-1}(t) \leq g^{-1}(\lambda t) \leq c^{-1} \lambda^{1/(q-1)} g^{-1}(t).
\end{alignat*}
\item
It holds that
\begin{alignat*}{2}
&G(t+t') \leq 2^q(G(t)+G(t')),
&\quad&G^{-1}(t+t') \leq 2^{1/p}(G^{-1}(t)+G^{-1}(t')), \\
&\bar{g}(t+t') \leq 2^{q-1}(\bar{g}(t)+\bar{g}(t')),
&&\bar{g}^{-1}(t+t') \leq 2^{1/(p-1)} (\bar{g}^{-1}(t) + \bar{g}^{-1}(t')), \\
&g(t+t') \leq \frac{q}{p}2^{q-1}(g(t)+g(t')),
&&g^{-1}(t+t') \leq c^{-1} 2^{1/(p-1)} (g^{-1}(t)+g^{-1}(t')).
\end{alignat*}
\end{enumerate}
\end{lemma}

Another useful property of the growth functions is the following: for any $a, b \geq 0$ and $\varepsilon > 0$,
\begin{equation}\label{eq-alg}
g(a)b \leq \varepsilon g(a)a + g(b/\varepsilon)b
\end{equation}
Notice that \eqref{eq-alg} follows by splitting the cases $b \leq \varepsilon a$ and $b > \varepsilon a$.

The \emph{conjugate function} of a Young function $H$ is defined by $H^\ast(\tau) = \sup_{t \geq 0} (t\tau-H(t))$ for all $\tau \geq 0$. We obtain by definition
\begin{equation}\label{eq-Young}
t\tau \leq H(t) + H^\ast(\tau) \quad\text{for all}~ t, \tau \geq 0,
\end{equation}
which is called \emph{Young's inequality}. The following lemma, together with Young's inequality, is frequently used in the sequel.

\begin{lemma}\label{lem-H}
Let $H(t) = G(t^{1/a})$ for $a \in (0, p]$. Then $H^\ast(G(t)/t^a) \sim G(t)$. In particular, $G^\ast(g(t)) \sim G(t)$.
\end{lemma}

\begin{proof}
Note that we have $H^{-1}(t) = G^{-1}(t)^a$. It follows from Theorem~2.4 in Harjulehto--H\"ast\"o~\cite{HH19} that $H^{-1}(t) (H^{\ast})^{-1}(t) \sim t$, and hence $(H^{\ast})^{-1}(t) \sim t/G^{-1}(t)^a$. This concludes $H^\ast(G(t)/t^a) \sim G(t)$.
\end{proof}

We refer the reader to Diening--Harjulehto--H\"ast\"o--R{\r u}{\v z}i{\v c}ka~\cite{DHHR11} and Harjulehto--H\"ast\"o~\cite{HH19} for an exhaustive description on Orlicz functions.

%%%%%%%%%%%%%%%%%%%%%%%%%%%%%%%%%%%%%%%
\subsection{Function spaces}
%%%%%%%%%%%%%%%%%%%%%%%%%%%%%%%%%%%%%%%

Let $\Omega$ be an open subset of $\mathbb{R}^n$. The \emph{Orlicz} and \emph{(fractional) Orlicz--Sobolev spaces} are defined by
\begin{align*}
L^G(\Omega) &= \left\{ u: \Omega \to \mathbb{R} ~\text{measurable}: \varrho_{L^G(\Omega)}(u) < \infty \right\}, \\
W^{s, G}(\Omega) &= \left\{ u \in L^G(\Omega): \varrho_{W^{s, G}(\Omega)}(u) < \infty \right\}, \\
V^{s, G}(\Omega) &= \left\{ u: \mathbb{R}^n \to \mathbb{R} ~\text{measurable}: u|_\Omega \in L^G(\Omega), \varrho_{V^{s, G}(\Omega)}(u) < \infty \right\},
\end{align*}
where
\begin{align*}
\varrho_{L^G(\Omega)}(u) &= \int_\Omega G(|u|) \,\mathrm{d}x, \\
\varrho_{W^{s, G}(\Omega)}(u) &= \int_\Omega \int_\Omega G(|D^s u|) \frac{\mathrm{d}y \,\mathrm{d}x}{|x-y|^n}, \\
\varrho_{V^{s, G}(\Omega)}(u) &= \int_\Omega \int_{\mathbb{R}^n} G(|D^s u|) \frac{\mathrm{d}y \,\mathrm{d}x}{|x-y|^n}.
\end{align*}
These spaces are Banach spaces with the norms
\begin{align*}
\|u\|_{L^G(\Omega)} &= \inf \left\{ \lambda>0: \varrho_{L^G(\Omega)}(u/\lambda) \leq 1 \right\}, \\
\|u\|_{W^{s, G}(\Omega)} &= \|u\|_{L^G(\Omega)} + \inf \left\{ \lambda>0: \varrho_{W^{s, G}(\Omega)}(u/\lambda) \leq 1 \right\}, \\
\|u\|_{V^{s, G}(\Omega)} &= \|u\|_{L^G(\Omega)} + \inf \left\{ \lambda>0: \varrho_{V^{s, G}(\Omega)}(u/\lambda) \leq 1 \right\},
\end{align*}
respectively. By $W^{s, G}_{\text{loc}}(\Omega)$ we denote the space of functions that belong to $W^{s, G}(D)$ for each open set $D \Subset \Omega$. We also define the space $V^{s, G}_0(\Omega) = \overline{C_c^{\infty}(\Omega)}^{V^{s, G}(\Omega)}$. We refer the reader to \cite[Remark~2.2]{KLL23} for remarks on the spaces $V^{s, G}(\Omega)$ and $V^{s, G}_0(\Omega)$. The following theorem will be useful in the sequel. See Proposition~3.2 in Salort~\cite{Sal20} for the proof of \Cref{thm-Poincare}.

\begin{theorem}[Fractional Poincar\'e inequality]\label{thm-Poincare}
Let $\Omega \subset \mathbb{R}^n$ be open and bounded. There exists a constant $C = C(n, p, q, s) > 0$ such that
\begin{equation*}
\int_\Omega G(|u|) \,\mathrm{d}x \leq C \int_{\Omega} \int_{\mathbb{R}^n} G(\mathrm{diam}^s(\Omega) |D^su|) \frac{\mathrm{d}y \,\mathrm{d}x}{|x-y|^n}
\end{equation*}
for all $u \in W_0^{s, G}(\Omega, \Omega')$.
\end{theorem}

Since we work with nonlocal equations, we need a function space that captures integrability of functions in the whole of $\mathbb{R}^n$. We define the \emph{tail space} $L^g_s(\mathbb{R}^n)$ by
\begin{equation*}
L^g_s(\mathbb{R}^n) = \left\{ u \in L^g_{\text{loc}}(\mathbb{R}^n): \int_{\mathbb{R}^n} g\left(\frac{|u(x)|}{(1+|x|)^s}\right) \frac{\mathrm{d}x}{ (1+|x|)^{n+s} } <\infty \right\},
\end{equation*}
or equivalently,
\begin{equation*}
L^g_s(\mathbb{R}^n) = \left\{ u \in L^g_{\text{loc}} (\mathbb{R}^n): \mathrm{Tail}_g(u; x_0, R) <\infty ~\text{for some $x_0 \in \mathbb{R}^n$ and $R>0$} \right\},
\end{equation*}
where $\mathrm{Tail}_g(u; x_0, R)$ is a \emph{(nonlocal) tail} defined by
\begin{equation*}
\mathrm{Tail}_g(u; x_0, R) = R^s g^{-1} \left( R^s \int_{\mathbb{R}^n \setminus B_R(x_0)} g \left( \frac{|u(x)|}{|x-x_0|^s} \right) \frac{\mathrm{d}x}{|x-x_0|^{n+s}} \right).
\end{equation*}
Note that $\mathrm{Tail}_g(u; x_0, R)$ is finite for any $x_0 \in \mathbb{R}^n$ and $R>0$ if $u \in L^g_s(\mathbb{R}^n)$. Since we often investigate the integral in the tail, we introduce
\begin{equation*}
T(u; x_0, R) := g\left( \frac{\mathrm{Tail}_g(u; x_0, R)}{R^s} \right) = R^s \int_{\mathbb{R}^n \setminus B_R(x_0)} g \left( \frac{|u(x)|}{|x-x_0|^s} \right) \frac{\mathrm{d}x}{|x-x_0|^{n+s}}
\end{equation*}
for notational convenience.

Having the spaces $W^{s, G}_\mathrm{loc}(\Omega)$ and $L^g_s(\mathbb{R}^n)$ at hand, we can define supersolution of $\mathcal{L}u=0$ in $\Omega$ as in \Cref{def-supersolution}. Note that $W^{s, G}_\mathrm{loc}(\Omega)$ is consistent with the local counterpart $W^{1, G}_\mathrm{loc}(\Omega)$ and serves as a natural function space for the interior regularity theory. However, a delicate issue arises when we choose an appropriate function space for the Dirichlet problem and the boundary regularity theory. Indeed, there have been several approaches to find suitable spaces for functions regular across the boundary of the domain. For instance, Byun--Kim--Ok~\cite{BKO23} and Kim--Lee--Lee~\cite{KLL23} used $V^{s, G}(\Omega)$ and Korvenp\"a\"a--Kuusi--Palatucci~\cite{KKP16} used $W^{1, G}_{\mathrm{loc}}(\Omega')$. However, these spaces are not appropriate for our purpose due to the following reasons:

\begin{enumerate}[(i)]
	\item Functions in $V^{s, G}(\Omega)$ belong to $L^G_s(\mathbb{R}^n)$, which is smaller than $L^g_s(\mathbb{R}^n)$. Supersolutions and superharmonic functions are not necessarily in $L^g_s(\mathbb{R}^n)$.	
	\item $W^{s, G}(\Omega')$ converges to $W^{1, G}(\Omega')$ as $s \nearrow 1$, which relies on the particular choice of $\Omega' (\Supset \Omega)$. 
\end{enumerate}

To complement aforementioned problems, we introduce a new space. Let $\Omega, \Omega'$ be open subsets of $\mathbb{R}^n$ such that $\Omega \Subset \Omega'$. We define
\begin{equation*}
W^{s, G}(\Omega, \Omega') = \left\{ u: \Omega' \to \mathbb{R} ~\text{measurable}: u|_\Omega \in L^G(\Omega), \varrho_{W^{s, G}(\Omega, \Omega')}(u) < \infty \right\},
\end{equation*}
where
\begin{equation*}
\varrho_{W^{s, G}(\Omega, \Omega')}(u) = \int_\Omega \int_{\Omega'} G(|D^s u|) \frac{\mathrm{d}y \,\mathrm{d}x}{|x-y|^n}.
\end{equation*}
It is a Banach space with the norm 
\begin{equation*}
\|u\|_{W^{s, G}(\Omega, \Omega')} = \|u\|_{L^G(\Omega)} + \inf \left\{ \lambda>0: \varrho_{W^{s, G}(\Omega, \Omega')}(u/\lambda) \leq 1 \right\}.
\end{equation*}
Obviously, we have $V^{s, G}(\Omega) \subset W^{s, G}(\Omega, \Omega') \subset W^{s, G}(\Omega) \subset L^G(\Omega)$. We also define a space $W^{s, G}_0(\Omega, \Omega') = \overline{C_c^{\infty}(\Omega)}^{W^{s, G}(\Omega, \Omega')}$. A couple of remarks are in order. First, functions in $W^{s, G}(\Omega, \Omega')$ do not belong to $L^G_s(\mathbb{R}^n)$ in general unless $\Omega'=\mathbb{R}^n$. Second, $W^{s, G}(\Omega, \Omega')$ converges to $W^{1, G}(\Omega)$ as $s \nearrow 1$, whereas $W^{s, G}(\Omega')$ converges to $W^{1, G}(\Omega')$, see Fern\'andez Bonder--Salort~\cite{FBS19}.

%%%%%%%%%%%%%%%%%%%%%%%%%%%%%%%%%%%%%%%
\subsection{Supersolutions and superharmonic functions}
%%%%%%%%%%%%%%%%%%%%%%%%%%%%%%%%%%%%%%%

In the introduction, we defined supersolutions and superharmonic functions. We provide several properties and make some remarks on these functions. Let us begin with supersolutions. For measurable functions $u, v: \mathbb{R}^n \to \mathbb{R}$, we consider a quantity
\begin{equation*}
\mathcal{E}(u, v) = \int_{\mathbb{R}^n} \int_{\mathbb{R}^n} g(|D^su|) \frac{D^su}{|D^su|} D^sv \frac{k(x, y)}{|x-y|^n} \,\mathrm{d}y \,\mathrm{d}x.
\end{equation*}
A sufficient condition for $\mathcal{E}$ to be well defined is given in the next lemma.

\begin{lemma}\label{lem-finiteness}
For $u \in W^{s, G}_{\mathrm{loc}}(\Omega) \cap L^g_s(\mathbb{R}^n)$ and for $v \in W^{s, G}_{\mathrm{loc}}(\Omega)$ with compact support in $\Omega$, $\mathcal{E}(u, v)$ is finite.
\end{lemma}

\begin{proof}
We first prove
\begin{equation}\label{eq-Holder1}
\int_\Omega \int_E g(|D^su|) |D^s v| \frac{\mathrm{d}y\,\mathrm{d}x}{|x-y|^n} \leq C \left( [u]^{q-1} \land [u]^{p-1} \right) [v],
\end{equation}
where $E$ is either $\Omega$, $\Omega'$, or $\mathbb{R}^n$, and $[\cdot]$ is the corresponding norm, namely, $[\cdot]=[\cdot]_{W^{s, G}(\Omega)}$, $[\cdot]=[\cdot]_{W^{s, G}(\Omega, \Omega')}$, or $[\cdot]=[\cdot]_{V^{s, G}(\Omega)}$, respectively. Indeed, we may assume that $[u] \neq 0$ and $[v] \neq 0$, then by \eqref{eq-Young}, \Cref{lem-G}, \Cref{lem-H}, and the unit ball property, we obtain
\begin{align*}
&\int_\Omega \int_E \frac{q}{p} \frac{g(|D^s u|)}{[u]^{q-1} \land [u]^{p-1}} \frac{|D^s v|}{[v]} \frac{\mathrm{d}y \,\mathrm{d}x}{|x-y|^n} \\
&\leq \int_\Omega \int_E \left( G^\ast \left( \frac{p}{q} \frac{g(|D^s u|)}{[u]^{p-1} \land [u]^{q-1}} \right) + G\left( \frac{|D^sv|}{[v]} \right) \right) \frac{\mathrm{d}y \,\mathrm{d}x}{|x-y|^n} \\
&\leq \int_\Omega \int_E \left( G^\ast \left( g\left( \frac{|D^su|}{[u]} \right) \right) + G\left( \frac{|D^sv|}{[v]} \right) \right) \frac{\mathrm{d}y \,\mathrm{d}x}{|x-y|^n} \\
&\leq C+1,
\end{align*}
as desired. Note that the same argument also shows that
\begin{equation}\label{eq-Holder2}
\int_\Omega g(|u|)|v| \,\mathrm{d}x \leq C \left( \|u\|_{L^G(\Omega)}^{q-1} \land \|u\|_{L^G(\Omega)}^{p-1} \right) \|v\|_{L^G(\Omega)}.
\end{equation}
Let $D=\mathrm{supp}\,v$. Then for $D \Subset D' \Subset \Omega$, we have by using \eqref{eq-Holder1} and \eqref{eq-Holder2}
\begin{align}\label{eq-finiteness}
\begin{split}
|\mathcal{E}(u, v)|
&\leq 2\Lambda \int_D \int_{D'} g(|D^su|) |D^s v| \frac{\mathrm{d}y \,\mathrm{d}x}{|x-y|^n} \\
&\quad + 2\Lambda \int_D \int_{\mathbb{R}^n \setminus D'} g\left( \frac{|u(x)|+|u(y)|}{|x-y|^s} \right) \frac{|v(x)|}{|x-y|^s} \frac{\mathrm{d}y \,\mathrm{d}x}{|x-y|^n} \\
&\leq C [u]_{W^{s, G}(D, D')}^{p-1}[v]_{W^{s, G}(D, D')} + C\|u\|_{L^G(D)}^{p-1} \|v\|_{L^G(D)} + C T(u; x_0, d) \|v\|_{L^1(D)} \\
&\leq C \left( \|u\|_{W^{s, G}(D')}^{p-1} + T(u; x_0, d) \right) \|v\|_{W^{s, G}(D')},
\end{split}
\end{align}
where $x_0$ is any point in $D$ and $d=\mathrm{dist}(D, \partial D')$. Thus, $\mathcal{E}(u, v)$ is finite.
\end{proof}

\Cref{lem-finiteness} shows that the left-hand side of \eqref{eq-supersolution} in \Cref{def-supersolution} is always well defined. Moreover, a standard approximation argument shows the following result.

\begin{lemma}
If $u \in W^{s, G}(\Omega, \Omega')$ ($u \in V^{s, G}(\Omega)$, respectively) is a supersolution of $\mathcal{L}u=0$ in $\Omega$, then \eqref{eq-supersolution} holds for all nonnegative functions $\varphi \in W^{s, G}_0(\Omega, \Omega')$ ($\varphi \in V^{s, G}_0(\Omega)$, respectively).
\end{lemma}

The requirement $u_- \in L_s^g(\mathbb{R}^n)$ in \Cref{def-supersolution} can be replaced by $u \in L_s^g(\mathbb{R}^n)$ due to the next lemma.

\begin{lemma}\label{lem-tail}
Let $\sigma \in (0,s)$ satisfy
\begin{equation}\label{eq-sigma}
\sigma \geq (sq-1)/(q-1).
\end{equation}
There exists a constant $C=C(n, p, q, s, \sigma, \Lambda) > 0$ such that if $u$ is a supersolution of $\mathcal{L}u=0$ in $B_R=B_R(x_0)$, then
\begin{equation}\label{eq-tail}
T(u_+; x_0, R) \leq C \left( \fint_{B_R} \int_{B_R} g(R^{\sigma-s}|D^\sigma u|) \frac{\mathrm{d}y\,\mathrm{d}x}{|x-y|^n} + \fint_{B_R} g\left( \frac{|u|}{R^s} \right) \,\mathrm{d}x + T(u_-; x_0, R) \right).
\end{equation}
In particular, if $u$ is a supersolution of $\mathcal{L}u=0$ in $\Omega$, then $u \in L^g_s(\mathbb{R}^n)$.
\end{lemma}

\begin{proof}
Let $\eta \in C_c^{\infty}(B_{R/2})$ be a cut-off function such that $\eta \equiv 1$ on $B_{R/4}$, $0 \leq \eta \leq 1$, and $|\nabla \eta| \leq 8/R$. By testing \eqref{eq-supersolution} with $\eta$, we obtain
\begin{align*}
0
&\leq \int_{B_R} \int_{B_R} g(|D^su|) \frac{D^su}{|D^su|} D^s\eta \frac{k(x, y)}{|x-y|^n} \,\mathrm{d}y \,\mathrm{d}x \\
&\quad + 2\int_{B_{R/2}} \int_{\mathbb{R}^n \setminus B_R} g(|D^su|) \frac{D^su}{|D^su|} \frac{\eta(x)}{|x-y|^s} \frac{k(x, y)}{|x-y|^n} \,\mathrm{d}y \,\mathrm{d}x =: I_1 + 2I_2.
\end{align*}
By using \Cref{lem-G}, $|D^s\eta| \leq 8|x-y|^{1-s}/R$, and the assumption \eqref{eq-sigma}, we estimate
\begin{align*}
I_1
&\leq \frac{C}{R} \int_{B_R} \int_{B_R} g(R^{\sigma-s}|D^\sigma u|) \left( \frac{R}{|x-y|} \right)^{(s-\sigma)(q-1)} \frac{\mathrm{d}y\,\mathrm{d}x}{|x-y|^{n+s-1}} \\
&\leq \frac{C}{R^s} \int_{B_R} \int_{B_R} g(R^{\sigma-s}|D^\sigma u|) \frac{\mathrm{d}y\,\mathrm{d}x}{|x-y|^n}.
\end{align*}
For $I_2$, we use \Cref{lem-ineq1} and \Cref{lem-G} to have
\begin{equation*}
g(|D^su|) \frac{D^su}{|D^su|} \leq 2C_0 g\left( \frac{u_+(x)}{|x-y|^s} \right) + 2C_0 g\left( \frac{u_-(y)}{|x-y|^s} \right) - \frac{1}{C_0} g\left( \frac{u_+(y)}{|x-y|^s} \right),
\end{equation*}
where $C_0 = \frac{q}{p}2^{q-1}$. Thus, we obtain
\begin{equation*}
\frac{I_2}{R^{n-s}} \leq C \fint_{B_{R/2}} g\left( \frac{|u|}{R^s} \right) \,\mathrm{d}x + C T(u_-; x_0, R) - c T(u_+; x_0, R).
\end{equation*}
By combining the estimates for $I_1$ and $I_2$, we arrive at \eqref{eq-tail}.

For the second assertion, we fix $B_R=B_R(x_0) \Subset \Omega$ and let $I_3$, $I_4$ denote the first and the second terms in the right-hand side of \eqref{eq-tail}, respectively. By using Young's inequality and applying \Cref{lem-H}, we have
\begin{align*}
I_3
&\leq C \fint_{B_R} \int_{B_R} g(|D^su|) \left| \frac{x-y}{R} \right|^{(s-\sigma)(p-1)} \frac{\mathrm{d}y\,\mathrm{d}x}{|x-y|^n} \\
&\leq C \fint_{B_R} \int_{B_R} G(|D^su|) \frac{\mathrm{d}y\,\mathrm{d}x}{|x-y|^n} + C \fint_{B_R} \int_{B_R} G\left( \left| \frac{x-y}{R} \right|^{(s-\sigma)(p-1)} \right) \frac{\mathrm{d}y\,\mathrm{d}x}{|x-y|^n} \\
&\leq C \fint_{B_R} \int_{B_R} G(|D^su|) \frac{\mathrm{d}y\,\mathrm{d}x}{|x-y|^n} + C \fint_{B_R} \int_{B_R} \left| \frac{x-y}{R} \right|^{(s-\sigma)(p-1)p} \frac{\mathrm{d}y\,\mathrm{d}x}{|x-y|^n} < \infty
\end{align*}
and
\begin{equation*}
I_4 \leq C \fint_{B_R} G\left( \frac{|u(x)|}{R^s} \right) \,\mathrm{d}x + C < \infty.
\end{equation*}
Thus, $u \in L^g_s(\mathbb{R}^n)$.
\end{proof}

The solvability of the Dirichlet problem for the equation $\mathcal{L}u=0$ with Sobolev boundary data is proved in \Cref{sec-obstacle} in a more general framework of obstacle problems. Moreover, some properties of supersolutions (and subsolutions) are provided in \Cref{sec-loc-est} and \Cref{sec-supersolution}.

Let us next make a couple of remarks on superharmonic functions.

\begin{remark}\label{rmk-superharmonic}
\begin{enumerate}[(i)]
\item
It immediately follows from the definition that the pointwise minimum of two superharmonic functions is superharmonic as well.
\item
A superharmonic function is locally bounded from below in $\Omega$ since the minimum on compact sets cannot be $-\infty$ by definition.
\end{enumerate}
\end{remark}

Some properties of superharmonic functions are studied in~\Cref{sec-superharmonic}.

%%%%%%%%%%%%%%%%%%%%%%%%%%%%%%%%%%%%%%%
\section{Local estimates up to boundary}\label{sec-loc-est}
%%%%%%%%%%%%%%%%%%%%%%%%%%%%%%%%%%%%%%%

In this section, we study the local boundedness of subsolutions and the weak Harnack inequality for supersolutions. Such results for local operators with Orlicz growth are very well known, see Lieberman~\cite{Lie91a}, Moscariello--Nania~\cite{MN91}, and Mascolo--Papi~\cite{MP96} for instance. Let us first state the interior estimates for nonlocal operators.

\begin{theorem}\label{thm-loc-bdd-int}
Let $\varepsilon, p_0 > 0$. If $u$ is a subsolution of $\mathcal{L}u=0$ in $B_R=B_R(x_0)$, then
\begin{equation}\label{eq-loc-bdd-int}
\esssup_{B_{R/2}} u_+ \leq \varepsilon \, \mathrm{Tail}_g(u_+; x_0, R/2) + C \left( \fint_{B_R} u_+^{p_0} \,\mathrm{d}x \right)^{1/p_0},
\end{equation}
where $C=C(n, p, p_0, q, s, \Lambda, \varepsilon)>0$.
\end{theorem}

Some results similar to \Cref{thm-loc-bdd-int} were proved in Byun--Kim--Ok~\cite{BKO23} and Chaker--Kim--Weidner~\cite{CKW23} by using De Giorgi's iteration technique. These results are obtained under some additional regularity assumption on $u$ ($u \in V^{s, G}(\Omega$)), but the proofs indeed do not use this assumption. Thus, \Cref{thm-loc-bdd-int} can be seen as a slight modification of the results in \cite{BKO23,CKW23}.

\begin{theorem}\label{thm-WHI-int}
Let $\tau_1, \tau_2 \in (0,1)$. Let $\delta \in (0, \frac{n}{n-sp})$ if $sp<n$ and $\delta \in (0,\infty)$ if $sp \geq n$. If $u$ is a supersolution of $\mathcal{L}u=0$ in $B_R=B_R(x_0)$ such that $u\geq 0$ in $B_R$, then
\begin{equation}\label{eq-WHI-int}
\fint_{B_{\tau_1R}} g^\delta \left( \frac{u}{R^s} \right) \,\mathrm{d}x \leq C g^\delta\left( \essinf_{B_{\tau_2R}} \frac{u}{R^s} \right) + C g^\delta \left( \frac{\mathrm{Tail}_g(u_-; x_0, R)}{R^s} \right),
\end{equation}
where $C = C(n, p, q, s, \Lambda, \tau_1, \tau_2, \delta)>0$.
\end{theorem}

The weak Harnack inequality with a small exponent $\delta$ is provided in \cite{CKW23}. However, their result is not sufficient for the development of nonlocal nonlinear potential theory because we need an optimal exponent $\delta$ in \Cref{thm-WHI-int} in some occasions. We emphasize that the range of $\delta$ in \Cref{thm-WHI-int} is optimal in the sense that \Cref{thm-WHI-int} recovers Theorem~1.2 in Di Castro--Kuusi--Palatucci~\cite{DCKP14} when $G(t)=t^{p}$.

Let us next state the local regularity estimates up to the boundary. In this case, we need regularity of $u$ across the boundary.

\begin{theorem}\label{thm-loc-bdd-bdry}
Let $\Omega, \Omega'$ be open subsets of $\mathbb{R}^n$ such that $\Omega \Subset \Omega'$ and let $B_R=B_R(x_0)$ satisfy $B_R \cap \Omega \neq \emptyset$. Let $\varepsilon, p_0 > 0$. If $u \in W^{s, G}(\Omega, \Omega')$ is a subsolution of $\mathcal{L}u=0$ in $\Omega$, then
\begin{equation}\label{eq-loc-bdd-bdry}
\esssup_{B_{R/2}} u_M^+ \leq \varepsilon \,\mathrm{Tail}_g(u_M^+; x_0, R/2) + C \left( \fint_{B_R} (u_M^+)^{p_0} \,\mathrm{d}x \right)^{1/p_0},
\end{equation}
where
\begin{equation}\label{eq-uM}
M=\esssup_{B_R \setminus \Omega} u_+, \quad u_M^+(x) = \max\{u_+(x), M\},
\end{equation}
and $C=C(n, p, p_0, q, s, \Lambda, \varepsilon)>0$.
\end{theorem}

\begin{theorem}\label{thm-WHI-bdry}
Let $\tau_1, \tau_2 \in (0,1)$. Let $\delta \in (0, \frac{n}{n-sp})$ if $sp<n$ and $\delta \in (0,\infty)$ if $sp \geq n$. Let $\Omega, \Omega'$ be open subsets of $\mathbb{R}^n$ such that $\Omega \Subset \Omega'$ and let $B_R=B_R(x_0)$ satisfy $B_R \cap \Omega \neq \emptyset$. If $u \in W^{s, G}(\Omega, \Omega')$ is a supersolution of $\mathcal{L}u=0$ in $\Omega$ such that $u\geq0$ in $B_R$, then
\begin{equation}\label{eq-WHI-bdry}
\fint_{B_{\tau_1R}} g^\delta \left( \frac{u_m^-}{R^s} \right) \,\mathrm{d}x \leq C g^\delta\left( \essinf_{B_{\tau_2R}} \frac{u_m^-}{R^s} \right) + C g^\delta \left( \frac{\mathrm{Tail}_g((u_m^-)_-; x_0, R)}{R^s} \right),
\end{equation}
where
\begin{equation}\label{eq-um}
m=\essinf_{B_R \setminus \Omega} u, \quad u_m^-(x) = \min\{u(x), m\},
\end{equation}
and $C = C(n, p, q, s, \Lambda, \tau_1, \tau_2, \delta)>0$.
\end{theorem}

If $B_R(x_0) \subset \Omega$ in \Cref{thm-loc-bdd-bdry} and \Cref{thm-WHI-bdry}, then $M=\esssup_{\emptyset} u_+ = -\infty$, $u_M^+=u_+$, $m=\essinf_\emptyset u=\infty$, and $u_m^-=u$, and hence the estimates \eqref{eq-loc-bdd-bdry} and \eqref{eq-WHI-bdry} reduce to \eqref{eq-loc-bdd-int} and \eqref{eq-WHI-int}, respectively.

%%%%%%%%%%%%%%%%%%%%%%%%%%%%%%%%%%%%%%%
\subsection{Caccioppoli estimates}\label{sec-Caccioppoli}
%%%%%%%%%%%%%%%%%%%%%%%%%%%%%%%%%%%%%%%

To prove \Cref{thm-loc-bdd-int}--\Cref{thm-WHI-bdry}, we use Moser's iteration technique. For this purpose, we establish Caccioppoli-type estimates.

\begin{lemma}\label{lem-Caccio1}
Let $\gamma=\beta+1>1$ and $d>0$.
\begin{enumerate}[(i)]
\item
If $u$ is a subsolution of $\mathcal{L}u=0$ in $B_R=B_R(x_0)$, then
\begin{align}\label{eq-Caccio1}
\begin{split}
&\int_{B_r} \int_{B_r} G(|D^su_+|) \left( \fint_{v(y)}^{v(x)} \frac{\bar{g}^{\beta}(t)}{t} \,\mathrm{d}t \right) \frac{\mathrm{d}y \,\mathrm{d}x}{|x-y|^n} \\
&\leq C\left( \frac{1}{|\beta|} + \frac{1}{|\beta|^{q}} \right) \left( \frac{R}{R-r} \right)^{n+q} \left( 1+ \frac{T(u_++d; x_0, R)}{g(d/R^s)} \right) \int_{B_R} \bar{g}^{\gamma}(v) \,\mathrm{d}x
\end{split}
\end{align}
for any $r \in (0,R)$, where $v=(u_++d)/R^s$ and $C=C(n, p, q, s, \Lambda)>0$.
\item
Let $\Omega, \Omega'$ be open subsets of $\mathbb{R}^n$ such that $\Omega \Subset \Omega'$ and let $B_R=B_R(x_0)$ satisfy $B_R \cap \Omega \neq \emptyset$. If $u \in W^{s, G}(\Omega, \Omega')$ is an essentially bounded subsolution of $\mathcal{L}u=0$ in $\Omega$, then the estimate \eqref{eq-Caccio1} with $u_+$ replaced by $u_M^+$ holds for any $r \in (0,R)$, where $u_M^+$ is defined as in \eqref{eq-uM}.
\end{enumerate}
\end{lemma}

\begin{lemma}\label{lem-Caccio2}
Let $\gamma=1+\beta<1$ and $d>0$.
\begin{enumerate}[(i)]
\item
If $u$ is a supersolution of $\mathcal{L}u=0$ in $B_R=B_R(x_0)$ such that $u\geq 0$ in $B_R$, then
\begin{align}\label{eq-Caccio2}
\begin{split}
&\int_{B_r} \int_{B_r} G(|D^su|) \left( \fint_{v(y)}^{v(x)} \frac{\bar{g}^{\beta}(t)}{t} \,\mathrm{d}t \right) \frac{\mathrm{d}y \,\mathrm{d}x}{|x-y|^n} \\
&\leq C\left( \frac{1}{|\beta|} + \frac{1}{|\beta|^{q}} \right) \left( \frac{R}{R-r} \right)^{n+q} \left( 1+ \frac{T((u+d)_-; x_0, R)}{g(d/R^s)} \right) \int_{B_R} \bar{g}^{\gamma}(v) \,\mathrm{d}x
\end{split}
\end{align}
for any $r \in (0, R)$, where $v=(u+d)/R^s$ and $C=C(n, p, q, s, \Lambda)>0$.
\item
Let $\Omega, \Omega'$ be open subsets of $\mathbb{R}^n$ such that $\Omega \Subset \Omega'$ and let $B_R=B_R(x_0)$ satisfy $B_R \cap \Omega \neq \emptyset$. If $u \in W^{s, G}(\Omega, \Omega')$ is a supersolution of $\mathcal{L}u=0$ in $\Omega$ such that $u \geq 0$ in $B_R$, then the estimate \eqref{eq-Caccio2} with $u$ replaced by $u_m^-$ holds for any $r \in (0, R)$, where $u_m^-$ is defined as in \eqref{eq-um}.
\end{enumerate}
\end{lemma}

\begin{proof}[Proofs of \Cref{lem-Caccio1} and \Cref{lem-Caccio2}]
We set
\begin{equation}\label{eq-bar-u}
\bar{u}=
\begin{cases}
u_+ &\text{in the case of \Cref{lem-Caccio1} (i)}, \\
u_M^+ &\text{in the case of \Cref{lem-Caccio1} (ii)}, \\
u &\text{in the case of \Cref{lem-Caccio2} (i)}, \\
u_m^- &\text{in the case of \Cref{lem-Caccio2} (ii)},
\end{cases}
\end{equation}
then $v=(\bar{u}+d)/R^s$. If $u$ is either a subsolution or a supersolution, then so is $\bar{u}$. Let $\eta \in C^\infty_c(B_{(R+r)/2})$ be a cut-off function such that $0 \leq \eta \leq 1$, $\eta \equiv 1$ on $B_r$, and $|\nabla \eta| \leq C/(R-r)$. We define a nonnegative function
\begin{equation}\label{eq-varphi}
\varphi = (\bar{g}^\beta(v) - \bar{g}^\beta(l))\eta^q,
\end{equation}
where
\begin{equation*}
l=
\begin{cases}
d/R^s &\text{in the case of \Cref{lem-Caccio1} (i)}, \\
(M+d)/R^s &\text{in the case of \Cref{lem-Caccio1} (ii)}, \\
d/R^s &\text{in the case of \Cref{lem-Caccio2} (i)}, \\
(m+d)/R^s &\text{in the case of \Cref{lem-Caccio2} (ii)}.
\end{cases}
\end{equation*}
Note that $\varphi \in W^{s, G}_0(B_R)$ has a compact support in $B_R$ in the case (i) and $\varphi \in W^{s, G}_0(\Omega, \Omega')$ in the case (ii). By using $\varphi$ as a test function, we have
\begin{equation}\label{eq-Caccio-I1I2}
\begin{split}
0
&\geq \frac{R^s}{\beta} \int_{B_R} \int_{B_R} g(|D^s\bar{u}|) \frac{D^s\bar{u}}{|D^s\bar{u}|} D^s\varphi \frac{k(x, y)}{|x-y|^n} \,\mathrm{d}y \,\mathrm{d}x \\
&\quad + \frac{2R^s}{\beta} \int_{B_R} \int_{\mathbb{R}^n \setminus B_R} g(|D^s\bar{u}|) \frac{D^s\bar{u}}{|D^s\bar{u}|} \frac{\varphi(x)}{|x-y|^s} \frac{k(x, y)}{|x-y|^n} \,\mathrm{d}y \,\mathrm{d}x = I_1 + I_2.
\end{split}
\end{equation}
We apply \Cref{lem-ineq2} to estimate $I_1$ as
\begin{align*}
I_1
&\geq \frac{c}{\Lambda} \int_{B_R} \int_{B_R} G(|D^s\bar{u}|) \left( \fint_{v(y)}^{v(x)} \frac{\bar{g}^{\beta}(t)}{t} \,\mathrm{d}t \right) (\eta(x) \land \eta(y))^q \frac{\mathrm{d}y \,\mathrm{d}x}{|x-y|^n} \\
&\quad - C \Lambda \left( \frac{1}{|\beta|} \lor \frac{1}{|\beta|^q} \right) \int_{B_R} \bar{g}^{\gamma}(v(x)) \int_{B_R} \frac{(R^s |D^s\eta|)^p \lor (R^s |D^s\eta|)^q}{|x-y|^n} \,\mathrm{d}y \,\mathrm{d}x.
\end{align*}
Since
\begin{align}\label{eq-eta-q}
\begin{split}
\int_{\mathbb{R}^n} \frac{(R^s|D^s\eta|)^\alpha}{|x-y|^n} \,\mathrm{d}y
&\leq C \frac{R^{s\alpha}}{(R-r)^\alpha} \int_{B_R(x)} \frac{\mathrm{d}y}{|x-y|^{n-(1-s)\alpha}} + R^{s\alpha} \int_{\mathbb{R}^n \setminus B_R(x)} \frac{\mathrm{d}y}{|x-y|^{n+s\alpha}} \\
&\leq C \left( \frac{R}{R-r} \right)^{\alpha} + C \leq C \left( \frac{R}{R-r} \right)^q
\end{split}
\end{align}
for any $\alpha \in (0, q]$, we obtain
\begin{align}\label{eq-Caccio-I1}
\begin{split}
I_1
&\geq c \int_{B_r} \int_{B_r} G(|D^s\bar{u}|) \left( \fint_{v(y)}^{v(x)} \frac{\bar{g}^{\beta}(t)}{t} \,\mathrm{d}t \right) \frac{\mathrm{d}y \,\mathrm{d}x}{|x-y|^n} \\
&\quad - C \left( \frac{1}{|\beta|} + \frac{1}{|\beta|^q} \right) \left( \frac{R}{R-r} \right)^q \int_{B_R} \bar{g}^{\gamma}(v) \,\mathrm{d}x
\end{split}
\end{align}
for some $c, C>0$ depending only on $n$, $p$, $q$, $s$, and $\Lambda$.

Let us next estimate $I_2$. We have for $\beta>0$
\begin{align}\label{eq-Caccio-I2-pos}
\begin{split}
I_2
&\geq - \frac{2\Lambda}{|\beta|}R^s \int_{B_R} \int_{\mathbb{R}^n \setminus B_R} g(-R^sD^sv) \frac{(\bar{g}^{\beta}(v(x)) - \bar{g}^\beta(l))\eta^q(x)}{|x-y|^s} {\bf 1}_{\lbrace \bar{u}(x)\leq \bar{u}(y) \rbrace} \frac{\mathrm{d}y \,\mathrm{d}x}{|x-y|^n} \\
&\geq - \frac{2\Lambda}{|\beta|}R^s \int_{B_R} \int_{\mathbb{R}^n \setminus B_R} g\left( \frac{R^sv(y)}{|x-y|^s} \right) \frac{\bar{g}^{\beta}(v(x)) \eta^q(x)}{|x-y|^s} \frac{\mathrm{d}y \,\mathrm{d}x}{|x-y|^n} \\
&\geq - \frac{2\Lambda}{|\beta|} \left( R^s \sup_{x \in \mathrm{supp}\,\eta} \int_{\mathbb{R}^n \setminus B_R} g\left( \frac{R^sv(y)}{|x-y|^s} \right) \frac{\mathrm{d}y}{|x-y|^{n+s}} \right) \int_{B_R} \frac{\bar{g}^{\gamma}(v)}{\bar{g}(d/R^s)} \,\mathrm{d}x
\end{split}
\end{align}
and for $\beta<0$
\begin{align}\label{eq-Caccio-I2-neg}
\begin{split}
I_2
&\geq - \frac{2\Lambda}{|\beta|}R^s \int_{B_R} \int_{\mathbb{R}^n \setminus B_R} g(D^s\bar{u}) \frac{(\bar{g}^{\beta}(v(x)) - \bar{g}^\beta(l)) \eta^q(x)}{|x-y|^s} {\bf 1}_{\lbrace \bar{u}(x)\geq \bar{u}(y) \rbrace} \frac{\mathrm{d}y \,\mathrm{d}x}{|x-y|^n} \\
&\geq - \frac{2\Lambda}{|\beta|}R^s \int_{B_R} \int_{\mathbb{R}^n \setminus B_R} g\left( R^s\frac{v(x)+v_-(y)}{|x-y|^s} \right) \frac{\bar{g}^{\beta}(v(x)) \eta^q(x)}{|x-y|^s} \frac{\mathrm{d}y \,\mathrm{d}x}{|x-y|^n} \\
&\geq - \frac{C}{|\beta|}R^s \int_{B_R} \int_{\mathbb{R}^n \setminus B_R} g\left( \frac{R^sv(x)}{|x-y|^s} \right) \frac{\bar{g}^{\beta}(v(x)) \eta^q(x)}{|x-y|^s} \frac{\mathrm{d}y \,\mathrm{d}x}{|x-y|^n} \\
&\quad - \frac{C}{|\beta|} \left( R^s \sup_{x \in \mathrm{supp}\,\eta} \int_{\mathbb{R}^n \setminus B_R} g\left( \frac{R^sv_-(y)}{|x-y|^s} \right) \frac{\mathrm{d}y}{|x-y|^{n+s}} \right) \int_{B_R} \frac{\bar{g}^{\gamma}(v)}{\bar{g}(d/R^s)} \,\mathrm{d}x.
\end{split}
\end{align}
Since $\mathrm{supp}\,\eta \subset B_{(R+r)/2}(x_0)$, we have $|x-y| \geq \frac{R-r}{2R}|y-x_0|$ for $x \in \mathrm{supp}\,\eta$ and $y \in \mathbb{R}^n\setminus B_R$. Thus, \Cref{lem-G} yields
\begin{equation}\label{eq-Caccio-tail-pos}
R^s \sup_{x \in \mathrm{supp}\,\eta} \int_{\mathbb{R}^n \setminus B_R} g\left( \frac{R^sv(y)}{|x-y|^s} \right) \frac{\mathrm{d}y}{|x-y|^{n+s}} \leq C \left( \frac{R}{R-r} \right)^{n+sq} T(\bar{u}+d; x_0, R)
\end{equation}
for $\beta>0$ and
\begin{equation}\label{eq-Caccio-tail-neg}
R^s \sup_{x \in \mathrm{supp}\,\eta} \int_{\mathbb{R}^n \setminus B_R} g\left( \frac{R^sv_-(y)}{|x-y|^s} \right) \frac{\mathrm{d}y}{|x-y|^{n+s}} \leq C \left( \frac{R}{R-r} \right)^{n+sq} T((\bar{u}+d)_-; x_0, R)
\end{equation}
for $\beta<0$. Therefore, \eqref{eq-Caccio-I1I2}, \eqref{eq-Caccio-I1}, \eqref{eq-Caccio-I2-pos}, and \eqref{eq-Caccio-tail-pos} conclude \Cref{lem-Caccio1}.

To finish the proof of \Cref{lem-Caccio2}, we need to estimate the second to the last term in \eqref{eq-Caccio-I2-neg}. Note that we have
\begin{align*}
g\left( \frac{R^sv(x)}{|x-y|^s} \right) \frac{R^s}{|x-y|^s} \eta^q(x)
&\leq C g(v(x)) \left( \left( \frac{R^s}{|x-y|^s} \right)^p + \left( \frac{R^s}{|x-y|^s} \right)^q \right) \eta^q(x) \\
&\leq C \bar{g}(v(x)) ((R^s|D^s\eta|)^p + (R^s|D^s\eta|)^q)
\end{align*}
by using \Cref{lem-G} and \eqref{eq-bar-g-comp}. Using \eqref{eq-eta-q} again, we obtain
\begin{equation}\label{eq-Caccio-I2-extra}
R^s \int_{B_R} \int_{\mathbb{R}^n \setminus B_R} g\left( \frac{R^sv(x)}{|x-y|^s} \right) \frac{\bar{g}^{\beta}(v(x)) \eta^q(x)}{|x-y|^s} \frac{\mathrm{d}y \,\mathrm{d}x}{|x-y|^n} \leq C \left( \frac{R}{R-r} \right)^q \int_{B_R} \bar{g}^{\gamma}(v) \,\mathrm{d}x.
\end{equation}
Therefore, \Cref{lem-Caccio2} follows from \eqref{eq-Caccio-I1I2}, \eqref{eq-Caccio-I1}, \eqref{eq-Caccio-I2-neg}, \eqref{eq-Caccio-tail-neg}, and \eqref{eq-Caccio-I2-extra}.
\end{proof}

The following lemma improves \Cref{lem-Caccio2} when $\gamma=1+\beta=0$.

\begin{lemma}\label{lem-Caccio3}
Let $\sigma \in (0,s)$ and $d>0$.
\begin{enumerate}[(i)]
\item
If $u$ is a supersolution of $\mathcal{L}u=0$ in $B_R=B_R(x_0)$ such that $u \geq 0$ in $B_R$, then
\begin{equation}\label{eq-Caccio3}
\int_{B_r} \int_{B_r} |D^\sigma \log(u+d)|^p \frac{\mathrm{d}y\,\mathrm{d}x}{|x-y|^n} \leq CR^{n-\sigma p} \left( \frac{R}{R-r} \right)^{n+q} \left( 1+ \frac{T((u+d)_-; x_0, R)}{g(d/R^s)} \right),
\end{equation}
for any $r \in (0, R)$, where $C=C(n, p, q, s, \sigma, \Lambda) > 0$.
\item
Let $\Omega, \Omega'$ be open subsets of $\mathbb{R}^n$ such that $\Omega \Subset \Omega'$ and let $B_R=B_R(x_0)$ satisfy $B_R \cap \Omega \neq \emptyset$. If $u \in W^{s, G}(\Omega, \Omega')$ is a supersolution of $\mathcal{L}u=0$ in $\Omega$ such that $u\geq 0$ in $B_R$, then the estimate \eqref{eq-Caccio3} with $u$ replaced by $u_m^-$ holds for any $r \in (0,R)$, where $u_m^-$ is defined as in \eqref{eq-um}.
\end{enumerate}
\end{lemma}

\begin{proof}
Let $L$ be the left-hand side of \eqref{eq-Caccio3} (with $u$ replaced by $u_m^-$ in the case (ii)). Then \Cref{lem-ineq3} shows that
\begin{align*}
L
&\leq \frac{C}{R^{sp}} \int_{B_r} \int_{B_r} |x-y|^{(s-\sigma)p} \left( G(|D^s\bar{u}|) \fint_{v(y)}^{v(x)} \frac{\mathrm{d}t}{G(t)} + 1 \right) \frac{\mathrm{d}y\,\mathrm{d}x}{|x-y|^n} \\
&\leq \frac{C}{R^{\sigma p}} \int_{B_r} \int_{B_r} G(|D^s\bar{u}|) \left( \fint_{v(y)}^{v(x)} \frac{\mathrm{d}t}{G(t)} \right) \frac{\mathrm{d}y\,\mathrm{d}x}{|x-y|^n} + CR^{n-\sigma p},
\end{align*}
where
\begin{equation}\label{eq-bar-u-v}
\bar{u}=
\begin{cases}
u &\text{in the case (i)}, \\
u_m^- &\text{in the case (ii)},
\end{cases}
\quad\text{and}\quad v=(\bar{u}+d)/R^s.
\end{equation}
Thus, \Cref{lem-Caccio2} with $\gamma=1+\beta=0$ finishes the proof.
\end{proof}

%%%%%%%%%%%%%%%%%%%%%%%%%%%%%%%%%%%%%%%
\subsection{Local boundedness}\label{sec-loc-bdd}
%%%%%%%%%%%%%%%%%%%%%%%%%%%%%%%%%%%%%%%

Using the Caccioppoli-type estimates proved in \Cref{sec-Caccioppoli} and Moser's iteration technique, we prove the local boundedness of subsolutions (\Cref{thm-loc-bdd-int} and \Cref{thm-loc-bdd-bdry}) in this section and the weak Harnack inequality for supersolutions (\Cref{thm-WHI-int} and \Cref{thm-WHI-bdry}) in the next section. We begin with the reverse H\"older inequalities.

\begin{lemma}\label{lem-RHI-sub}
Let $\gamma=\beta+1>1$, $d>0$, $\sigma \in (0, s)$, and
\begin{equation}\label{eq-chi}
\chi=
\begin{cases}
n/(n-\sigma p) &\text{if}~ p<n/\sigma, \\
\text{any number larger than 1 and $\sigma/(n-\sigma)$} &\text{if}~ p\geq n/\sigma.
\end{cases}
\end{equation}
\begin{enumerate}[(i)]
\item
If $u$ is a subsolution of $\mathcal{L}u=0$ in $B_R=B_R(x_0)$, then
\begin{align}\label{eq-RHI-sub}
\begin{split}
&\left( \fint_{B_r} \bar{g}^{\gamma\chi}(v) \,\mathrm{d}x \right)^{\frac{1}{\gamma\chi}} \\
&\leq (C(1+\gamma^p))^\frac{1}{\gamma} \left( \frac{R}{r} \right)^\frac{n}{\gamma} \left( \frac{R}{R-r} \right)^{\frac{n+q}{\gamma}} \left( 1+ \frac{T(u_++d; x_0, R)}{g(d/R^s)} \right)^\frac{1}{\gamma} \left( \fint_{B_R} \bar{g}^{\gamma}(v) \,\mathrm{d}x \right)^{\frac{1}{\gamma}},
\end{split}
\end{align}
for any $r \in (0,R)$, where $v=(u_++d)/R^s$. The constant $C=C(n, p, q, s, \sigma, \Lambda, \beta, \chi)>0$ is bounded when $\beta$ is bounded away from zero.
\item
Let $\Omega, \Omega'$ be open subsets of $\mathbb{R}^n$ such that $\Omega \Subset \Omega'$ and let $B_R=B_R(x_0)$ satisfy $B_R \cap \Omega \neq \emptyset$. If $u \in W^{s, G}(\Omega, \Omega')$ is an essentially bounded subsolution of $\mathcal{L}u=0$ in $\Omega$, then the estimate \eqref{eq-RHI-sub} with $u_+$ replaced by $u_M^+$ holds for $r \in (0,R)$, where $u_M^+$ is defined as in \eqref{eq-uM}.
\end{enumerate}
\end{lemma}

The following lemma provides the reverse H\"older inequalities for supersolutions, which are used in the next section, but let us include it here because its proof is similar to that of \Cref{lem-RHI-sub}.

\begin{lemma}\label{lem-RHI-super}
Let $\gamma=\beta+1<1$, $\gamma\neq0$, $d>0$, $\sigma \in (0,s)$, and $\chi$ be as in \eqref{eq-chi}.
\begin{enumerate}[(i)]
\item
If $u$ is a supersolution of $\mathcal{L}u=0$ in $B_R=B_R(x_0)$ such that $u \geq 0$ in $B_R$, then for any $r \in (0,R)$,
\begin{align}\label{eq-RHI-super-pos}
\begin{split}
&\left( \fint_{B_r} \bar{g}^{\gamma\chi}(v) \,\mathrm{d}x \right)^{\frac{1}{\gamma\chi}} \\
&\leq (C(1+\gamma^p))^\frac{1}{\gamma} \left( \frac{R}{r} \right)^\frac{n}{\gamma} \left( \frac{R}{R-r} \right)^{\frac{n+q}{\gamma}} \left( 1+ \frac{T((u+d)_-; x_0, R)}{g(d/R^s)} \right)^\frac{1}{\gamma} \left( \fint_{B_R} \bar{g}^{\gamma}(v) \,\mathrm{d}x \right)^{\frac{1}{\gamma}}
\end{split}
\end{align}
when $0<\gamma<1$ and
\begin{align}\label{eq-RHI-super-neg}
\begin{split}
&\left( \fint_{B_R} \bar{g}^{\gamma}(v) \,\mathrm{d}x \right)^{\frac{1}{\gamma}} \\
&\leq (C(1+|\gamma|^p))^\frac{1}{|\gamma|} \left( \frac{R}{r} \right)^\frac{n}{|\gamma|} \left( \frac{R}{R-r} \right)^{\frac{n+q}{|\gamma|}} \left( 1+ \frac{T((u+d)_-; x_0, R)}{g(d/R^s)} \right)^\frac{1}{|\gamma|} \left( \fint_{B_r} \bar{g}^{\gamma \chi}(v) \,\mathrm{d}x \right)^{\frac{1}{\gamma\chi}}
\end{split}
\end{align}
when $\gamma < 0$, where $v=(u+d)/R^s$. The constant $C=C(n, p, q, s, \sigma, \Lambda, \beta, \chi)>0$ is bounded when $\beta$ is bounded away from zero.
\item
Let $\Omega, \Omega'$ be open subsets of $\mathbb{R}^n$ such that $\Omega \Subset \Omega'$ and let $B_R=B_R(x_0)$ satisfy $B_R(x_0) \cap \Omega \neq \emptyset$. If $u \in W^{s, G}(\Omega, \Omega')$ is a supersolution of $\mathcal{L}u=0$ in $\Omega$ such that $u \geq 0$ in $B_R(x_0)$, then the estimates \eqref{eq-RHI-super-pos} and \eqref{eq-RHI-super-neg} with $u$ replaced by $u_m^-$ hold for any $r \in (0,R)$ when $0<\gamma<1$ and $\gamma<0$, respectively, where $u_m^-$ is defined as in \eqref{eq-um}.
\end{enumerate}
\end{lemma}

\begin{proof}
Let $\bar{u}$ be defined as in \eqref{eq-bar-u} and $v=(\bar{u}+d)/R^s$. We define a function $h=\bar{g}^{\gamma/p}(v)$. Let us first assume $p<n/\sigma$, then the fractional Sobolev inequality applied to $h$ yields
\begin{equation*}
L := \left( \fint_{B_r} \bar{g}^{\gamma \chi}(v) \,\mathrm{d}x \right)^{1/\chi} = \left( \fint_{B_r} h^{p^{\ast}_{\sigma}} \,\mathrm{d}x \right)^{1/\chi} \leq C r^{\sigma p} \fint_{B_r} \int_{B_r} |D^\sigma h|^p \frac{\mathrm{d}y \,\mathrm{d}x}{|x-y|^n} + C \fint_{B_r} h^p \,\mathrm{d}x,
\end{equation*}
where $p^{\ast}_{\sigma} := p\chi = np/(n-\sigma p)$ and $C=C(n, \sigma, p) > 0$. By applying \Cref{lem-ineq4}, we have
\begin{equation*}
|D^\sigma h|^p \leq C|\gamma|^p \frac{|x-y|^{(s-\sigma)p}}{R^{sp}} \left( G(|D^s\bar{u}|) \fint_{v(y)}^{v(x)} \frac{\bar{g}^{\gamma}(t)}{G(t)} \,\mathrm{d}t + \bar{g}^{\gamma}(v(x)) \lor \bar{g}^{\gamma}(v(y)) \right).
\end{equation*}
Since $|x-y| < 2r < 2R$, we obtain
\begin{align*}
L
&\leq C|\gamma|^p \fint_{B_r} \int_{B_r} \frac{r^{\sigma p}}{R^{sp}} |x-y|^{(s-\sigma)p} G(|D^s\bar{u}|) \left( \fint_{v(y)}^{v(x)} \frac{\bar{g}^{\gamma}(t)}{G(t)} \,\mathrm{d}t \right) \frac{\mathrm{d}y \,\mathrm{d}x}{|x-y|^n} \\
&\quad + C|\gamma|^p \frac{r^{\sigma p}}{R^{sp}} \fint_{B_r} \bar{g}^{\gamma}(v(x)) \int_{B_r} |x-y|^{-n+(s-\sigma)p} \,\mathrm{d}y \,\mathrm{d}x + C \fint_{B_r} \bar{g}^{\gamma}(v) \,\mathrm{d}x \\
&\leq C|\gamma|^p \fint_{B_r} \int_{B_r} G(|D^s\bar{u}|) \left( \fint_{v(y)}^{v(x)} \frac{\bar{g}^{\beta}(t)}{t} \,\mathrm{d}t \right) \frac{\mathrm{d}y \,\mathrm{d}x}{|x-y|^n} + C \left( 1+|\gamma|^p \right) \fint_{B_r} \bar{g}^{\gamma}(v) \,\mathrm{d}x,
\end{align*}
where $C = C(n, p, q, \sigma) > 0$. The desired result now follows from \Cref{lem-Caccio1} and \Cref{lem-Caccio2}.

Let us next consider the case $p \geq n/\sigma$. In this case, we take $\tilde{p}=np\chi/(n+\sigma p\chi)$ so that $\tilde{p} \in (1, n/\sigma)$ and $\chi=\tilde{p}_{\sigma}^\ast/p>1$. Note that $\tilde{p}>1$ follows from the choice of $\chi$ in \eqref{eq-chi}. Then, the desired inequality can be proved as above by applying the fractional Sobolev inequality with $\tilde{p} < n/\sigma$, and then using Lemma~4.6 in Cozzi~\cite{Coz17} and H\"older's inequality.
\end{proof}

\begin{lemma}\label{lem-Lp}
Let $\gamma > 0$. Then there exist constants $C, \alpha>0$, depending only on $n$, $p$, $q$, and $\gamma$, such that
\begin{equation*}
(\bar{g}^\gamma)^{-1} \left( \fint_{B_R} \bar{g}^\gamma(|u|) \,\mathrm{d}x \right) \leq C \left( \fint_{B_R} |u|^\alpha \,\mathrm{d}x \right)^{1/\alpha}
\end{equation*}
for any $u \in L^\alpha(B_R)$.
\end{lemma}

\begin{proof}
We define a non-decreasing function $H(t) = \bar{g}^\gamma(t^{1/\alpha})$ with $\alpha$ to be determined. We have by using \eqref{eq-bar-g-pq}
\begin{equation*}
\left( \frac{H(t)}{t} \right)' = \left( \frac{\gamma}{\alpha} \frac{\bar{g}'(t^{1/\alpha}) t^{1/\alpha}}{\bar{g}(t^{1/\alpha})} - 1 \right) \frac{\bar{g}^{\gamma}(t^{1/\alpha})}{t^2} \leq \left( (q-1)\frac{\gamma}{\alpha} - 1 \right) \frac{\bar{g}^{\gamma}(t^{1/\alpha})}{t^{2}},
\end{equation*}
and hence the function $t \mapsto H(t)/t$ is non-increasing, provided that $\alpha \geq (q-1)\gamma$. Thus, by Lemma~2.2 in Ok~\cite{OK18} there exists a concave function $\widetilde{H}$ such that $\widetilde{H}/2 \leq H \leq \widetilde{H}$. We then have
\begin{align*}
\begin{split}
\left( \fint_{B_R} \bar{g}^{\gamma}(|u|) \,\mathrm{d}x \right)^{1/\gamma}
&\leq \left( \fint_{B_R} \widetilde{H}(|u|^\alpha) \,\mathrm{d}x \right)^{1/\gamma} \leq \widetilde{H}^{1/\gamma} \left( \fint_{B_R} |u|^\alpha \,\mathrm{d}x \right) \\
&\leq C H^{1/\gamma} \left( \fint_{B_R} |u|^\alpha \,\mathrm{d}x \right) \leq C \bar{g} \left( \left( \fint_{B_R} |u|^\alpha \,\mathrm{d}x \right)^{1/\alpha} \right)
\end{split}
\end{align*}
by means of Jensen's inequality.
\end{proof}

We are now ready to prove \Cref{thm-loc-bdd-int} and \Cref{thm-loc-bdd-bdry} by iterating the reverse H\"older inequality \Cref{lem-RHI-sub}.

\begin{proof}[Proofs of \Cref{thm-loc-bdd-int} and \Cref{thm-loc-bdd-bdry}]
Let $\bar{u}$ be defined as in \eqref{eq-bar-u}, and define $v=(\bar{u}+d)/R^s$. For $\varepsilon_0 \in (0,1)$ to be determined, choose $d>0$ to satisfy
\begin{align*}
	g(d/R^s)=\varepsilon_0 T(\bar{u}; x_0, R/2).
\end{align*}
In the case of \Cref{thm-loc-bdd-bdry}, we may assume that $u$ is essentially bounded since the desired result follows by approximating a subsolution by a sequence of essentially bounded subsolutions. We now fix a constant $\chi>1$ satisfying \eqref{eq-chi} and choose $\gamma>1$ so that
\begin{equation*}
	1-\frac{1}{\gamma}\frac{\chi}{\chi-1} \geq \frac{1}{2}.
\end{equation*}
Then by applying \Cref{lem-RHI-sub} and using $T(\bar{u}+d; x_0, R) \leq C\varepsilon_0^{-1}g(d/R^s)$, we have for $\sigma \in (0,s)$, and $R/2 \leq \rho_1 < \rho_2 \leq R$,
\begin{equation}\label{eq-iteration1}
\left( \fint_{B_{\rho_1}} \bar{g}^{\gamma\chi}(v) \,\mathrm{d}x \right)^{\frac{1}{\gamma\chi}} \leq \varepsilon_0^{-\frac{1}{\gamma}}\left( \frac{CR}{\rho_2-\rho_1} \right)^{\frac{n+q}{\gamma}} \left( \fint_{B_{\rho_2}} \bar{g}^{\gamma}(v) \,\mathrm{d}x \right)^{\frac{1}{\gamma}}.
\end{equation}
By iterating \eqref{eq-iteration1}, we obtain
\begin{align}\label{eq-g-gamma}
\begin{split}
\esssup_{B_{R/2}} \bar{g}\left( \frac{\bar{u}}{R^s} \right)
&\leq \esssup_{B_{R/2}} \bar{g}(v) \leq C\varepsilon_0^{-\frac{1}{\gamma}\frac{\chi}{\chi-1}} \left( \fint_{B_R} \bar{g}^\gamma(v) \,\mathrm{d}x \right)^{1/\gamma} \\
&\leq C \varepsilon_0^{1/2} \bar{g}\left( \frac{\mathrm{Tail}_g(\bar{u}; x_0, R/2)}{R^s} \right) + C_0 \left( \fint_{B_R} \bar{g}^\gamma\left( \frac{\bar{u}}{R^s} \right) \,\mathrm{d}x \right)^{1/\gamma},
\end{split}
\end{align}
where $C_0=C_0(n, p, q, s, \Lambda, \varepsilon_0) > 0$.

We prove \eqref{eq-loc-bdd-int} and \eqref{eq-loc-bdd-bdry} by using \eqref{eq-g-gamma} and a covering argument. Let $R/2 \leq \rho_1<\rho_2\leq R$ and $\rho=\rho_2-\rho_1$, then $B_\rho(y) \subset B_R$ for any $y \in B_{\rho_1}$. Applying \eqref{eq-g-gamma} in $B_\rho(y)$, we have
\begin{equation}\label{eq-ball-y}
\esssup_{B_{\rho/2}(y)} \bar{u} \leq C \varepsilon_0^{1/2} \, \mathrm{Tail}_g(\bar{u}; y, \rho/2) + C_0 \rho^s (\bar{g}^\gamma)^{-1} \left( \fint_{B_\rho(y)} \bar{g}^\gamma \left( \frac{\bar{u}}{\rho^s} \right) \,\mathrm{d}x \right).
\end{equation}
Let us estimate the right-hand side of \eqref{eq-ball-y} as follows. By using \Cref{lem-G}, we have
\begin{align*}
T(\bar{u}; y, \rho/2)
&\leq (\rho/2)^s \int_{B_{\rho_2} \setminus B_{\rho/2}(y)} g\left( \frac{\esssup_{B_{\rho_2}}\bar{u}}{|x-y|^s} \right) \frac{\mathrm{d}x}{|x-y|^{n+s}} \\
&\quad + C \left( \frac{R}{\rho_2-\rho_1} \right)^{n+sq} (R/2)^s \int_{\mathbb{R}^n \setminus B_{\rho_2}} g\left( \frac{\bar{u}(x)}{|x-x_0|^s} \right) \frac{\mathrm{d}x}{|x-x_0|^{n+s}} \\
&\leq C g\left( \frac{\esssup_{B_{\rho_2}}\bar{u}}{(\rho/2)^s} \right) + C \left( \frac{R}{\rho_2-\rho_1} \right)^{n+sq} T(\bar{u}; x_0, R/2),
\end{align*}
from which we deduce
\begin{equation}\label{eq-y-tail}
C\varepsilon_0^{1/2} \,\mathrm{Tail}_g(\bar{u}; y, \rho/2) \leq \frac{1}{4} \esssup_{B_{\rho_2}} \bar{u} + C \varepsilon_0^{1/2} \left( \frac{R}{\rho_2-\rho_1} \right)^{\frac{n+sq}{p-1}}\mathrm{Tail}_g(\bar{u}; x_0, R/2)
\end{equation}
by assuming that $\varepsilon_0$ is sufficiently small. For the last term in \eqref{eq-ball-y}, we use \Cref{lem-Lp} to find $\alpha>0$ such that
\begin{equation}\label{eq-alpha}
(\bar{g}^\gamma)^{-1} \left( \fint_{B_\rho(y)} \bar{g}^\gamma \left( \frac{\bar{u}}{\rho^s} \right) \,\mathrm{d}x \right) \leq C \left( \fint_{B_\rho(y)} \left( \frac{\bar{u}}{\rho^s} \right)^\alpha \,\mathrm{d}x \right)^{1/\alpha}.
\end{equation}
Suppose that $\alpha>p_0$, then Young's inequality yields
\begin{align*}
\left( \fint_{B_\rho(y)} \left( \frac{\bar{u}}{\rho^s} \right)^\alpha \,\mathrm{d}x \right)^{1/\alpha}
&\leq \esssup_{B_{\rho_2}} \left( \frac{\bar{u}}{\rho^s} \right)^{1-p_0/\alpha} \left( \fint_{B_\rho(y)} \left( \frac{\bar{u}}{\rho^s} \right)^{p_0} \,\mathrm{d}x \right)^{1/\alpha} \\
&\leq \varepsilon_1 \esssup_{B_{\rho_2}} \frac{\bar{u}}{\rho^s} + \varepsilon_1^{1-\frac{\alpha}{p_0}} \left( \fint_{B_\rho(y)} \left( \frac{\bar{u}}{\rho^s} \right)^{p_0} \,\mathrm{d}x \right)^{1/p_0}
\end{align*}
for any $\varepsilon_1>0$. Hence, we obtain
\begin{equation}\label{eq-y-Lp}
C_0 \rho^s (\bar{g}^\gamma)^{-1}\left( \fint_{B_\rho(y)} \bar{g}^\gamma \left( \frac{\bar{u}}{\rho^s} \right) \,\mathrm{d}x \right) \leq \frac{1}{4} \esssup_{B_{\rho_2}} \bar{u} + C \left( \frac{R}{\rho_2-\rho_1} \right)^{\frac{n}{p_0}} \left( \fint_{B_R} \bar{u}^{p_0} \,\mathrm{d}x \right)^{\frac{1}{p_0}}.
\end{equation}
by taking $\varepsilon_1=1/(4C_0)$. Note that \eqref{eq-y-Lp} is an obvious consequence of \eqref{eq-alpha} and H\"older's inequality when $\alpha \leq p_0$.

We combine \eqref{eq-ball-y}, \eqref{eq-y-tail}, and \eqref{eq-y-Lp} to derive
\begin{align*}
\esssup_{B_{\rho_1}} \bar{u}
&\leq \frac{1}{2} \esssup_{B_{\rho_2}} \bar{u} + C \left( \frac{R}{\rho_2-\rho_1} \right)^{n/p_0} \left( \fint_{B_R} \bar{u}^{p_0} \,\mathrm{d}x \right)^{1/p_0} \\
&\quad + C \varepsilon_0^{1/2} \left( \frac{R}{\rho_2-\rho_1} \right)^{\frac{n+sq}{p-1}} \mathrm{Tail}_g(\bar{u}; x_0, R/2).
\end{align*}
The desired estimates \eqref{eq-loc-bdd-int} and \eqref{eq-loc-bdd-bdry} are now obtained by using a standard iteration argument (see, for instance, Lemma~4.11 in Cozzi~\cite{Coz17}) and choosing $\varepsilon_0^{1/2} \leq \varepsilon/C$.
\end{proof}

%%%%%%%%%%%%%%%%%%%%%%%%%%%%%%%%%%%%%%%
\subsection{Weak Harnack inequality}
%%%%%%%%%%%%%%%%%%%%%%%%%%%%%%%%%%%%%%%

In this section we prove the weak Harnack inequalities \Cref{thm-WHI-int} and \Cref{thm-WHI-bdry}. The proofs use the reverse H\"older inequalities \Cref{lem-RHI-super} and the logarithmic estimate below.

\begin{lemma}\label{lem-log}
Let $\tau \in (0,1)$.
\begin{enumerate}[(i)]
\item
If $u$ is a supersolution of $\mathcal{L}u=0$ in $B_R=B_R(x_0)$ such that $u\geq 0$ in $B_R$ and if
\begin{equation}\label{eq-d-tail1}
d \geq \mathrm{Tail}_g(u_-; x_0, R),
\end{equation}
then
\begin{equation}\label{eq-log}
\left( \fint_{B_{\tau R}} \left( \frac{u+d}{R^s} \right)^{\varepsilon_{0}} \,\mathrm{d}x \right)^{1/\varepsilon_{0}} \leq C \left( \fint_{B_{\tau R}} \left( \frac{u+d}{R^s} \right)^{-\varepsilon_{0}} \,\mathrm{d}x \right)^{-1/\varepsilon_{0}}
\end{equation}
for some constants $\varepsilon_0 \in (0,1)$ and $C>0$ depending only on $n$, $p$, $q$, $s$, $\Lambda$, and $\tau$.
\item
Let $\Omega, \Omega'$ be open subsets of $\mathbb{R}^n$ such that $\Omega \Subset \Omega'$ and let $B_R=B_R(x_0)$ satisfy $B_R \cap \Omega \neq \emptyset$. If $u \in W^{s, G}(\Omega, \Omega')$ is a supersolution of $\mathcal{L}u=0$ in $\Omega$ such that $u \geq 0$ in $B_R$ and if
\begin{equation}\label{eq-d-tail2}
d \geq \mathrm{Tail}_g(u_m^-; x_0, R),
\end{equation}
then the estimate \eqref{eq-log} with $u$ replaced by $u_m^-$ holds, where $u_m^-$ is defined as in \eqref{eq-um}.
\end{enumerate}
\end{lemma}

\begin{proof}
We claim that there exists a constant $C_0 = C_0(n, p, q, s, \Lambda, \tau) > 0$ such that
\begin{equation}\label{eq-BMO}
I:= \fint_{B_r(z_0)} |\log (\bar{u}(x)+d) - (\log (\bar{u}+d))_{B_r(z_0)}| \,\mathrm{d}x \leq C_0
\end{equation}
for any $B_R(z_0) \subset B_{\tau R}(x_0)$ with $r < \tau_0R$, where $\tau_0=(1-\tau)\frac{2\tau}{1+\tau}$ and $\bar{u}$ is defined as in \eqref{eq-bar-u}. It means that $\log (\bar{u}+d)$ is a function of bounded mean oscillation. Once we prove the claim \eqref{eq-BMO}, the John--Nirenberg embedding shows that there exist constants $\varepsilon_0 \in (0,1)$ and $C > 0$, depending only on $n$ and $C_0$, such that
\begin{equation*}
\fint_{B_{\tau R}(x_0)} \exp \left( \varepsilon_0 |\log (\bar{u}(x)+d)-(\log (\bar{u}+d))_{B_{\tau R}(x_0)}| \right) \,\mathrm{d}x \leq C,
\end{equation*}
and hence \eqref{eq-log} follows.

Let us now prove the claim \eqref{eq-BMO}. Let $R_0=\frac{1+\tau}{2\tau}r$, then for $r < \tau_0R$ we have
\begin{equation*}
B_r(z_0) \Subset B_{R_0}(z_0) \subset B_R(x_0).
\end{equation*}
Indeed, we observe that if $x \in B_{R_0}(z_0)$ then $|x-x_0| \leq |x-z_0| + |z_0-x_0| < R_0 + \tau R < \frac{1+\tau}{2\tau} \tau_0R + \tau R = R$, so $x \in B_R(x_0)$. Let $\sigma \in (0,s)$. By applying \Cref{lem-Caccio3} to $\bar{u}$ in $B_r(z_0) \Subset B_{R_0}(z_0)$, we obtain
\begin{equation*}
[\log (\bar{u}+d)]_{W^{\sigma, p}(B_r(z_0))}^p \leq CR_0^{n-\sigma p} \left( 1 + \frac{T((\bar{u}+d)_-; z_0, R_0)}{g(d/R_0^s)} \right).
\end{equation*}
Recalling that $u$ is nonnegative in $B_R(x_0)$ and observing that
\begin{equation*}
|y-x_0| \leq |y-z_0| + |z_0-x_0| < |y-z_0| + \tau R \leq (1+\tau) |y-z_0|,
\end{equation*}
we have
\begin{equation*}
\frac{T((\bar{u}+d)_-; z_0, R_0)}{g(d/R_0^s)} \leq \frac{R_0^s}{g(d/R_0^s)} \int_{\mathbb{R}^n \setminus B_R(x_0)} g\left( \frac{\bar{u}_-(y)}{|y-z_0|} \right) \frac{\mathrm{d}y}{|y-z_0|^{n+s}} \leq C \frac{T(\bar{u}_-; x_0, R)}{g(d/R^s)} \leq C,
\end{equation*}
where we used $(\bar{u}+d)_- \leq \bar{u}_-$, \Cref{lem-G}, and the assumption \eqref{eq-d-tail1} (or \eqref{eq-d-tail2}). Therefore, we have
\begin{equation*}
[\log (\bar{u}+d)]_{W^{\sigma, p}(B_r(z_0))}^p \leq CR_0^{n-\sigma p}.
\end{equation*}
By Jensen's inequality and the fractional Poincar\'e inequality, we arrive at
\begin{align*}
I
&\leq \left( \fint_{B_r(z_0)} |\log (\bar{u}(x)+d) - (\log (\bar{u}+d))_{B_r(z_0)}|^p \,\mathrm{d}x \right)^{1/p} \\
&\leq C r^{-n/p+\sigma} [\log (\bar{u}+d)]_{W^{\sigma, p}(B_r(z_0))} \leq C (R_0/r)^{\frac{n-\sigma p}{p}} \leq C_0,
\end{align*}
where $C_0 = C_0(n, s, p, q, \Lambda, \tau) > 0$. This proves the claim \eqref{eq-BMO}.
\end{proof}

We are now ready to provide the proofs of \Cref{thm-WHI-int} and \Cref{thm-WHI-bdry} by using the previous results.

\begin{proof}[Proofs of \Cref{thm-WHI-int} and \Cref{thm-WHI-bdry}]
Let $\bar{u}$ be defined as in \eqref{eq-bar-u}. For $d=\mathrm{Tail}_g(\bar{u}_-; x_0, R)$ we define $v=(\bar{u}+d)/R^s$. For the sake of brevity, we write
\begin{equation*}
\Phi(\alpha, r) = \left( \fint_{B_{r}} \bar{g}^{\alpha}(v) \,\mathrm{d}x \right)^{1/\alpha}
\end{equation*}
for $\alpha \in \mathbb{R}$. Let $\tau = (1+(\tau_{1}\lor \tau_{2}))/2$ so that $\tau_{1} \lor \tau_{2}<\tau<1$. We claim that there exist $\varepsilon_1, \varepsilon_2 \in (0,1)$ such that
\begin{equation}\label{eq-claim-eps}
\Phi(\varepsilon_1, \tau R) \leq C \Phi(-\varepsilon_2, \tau R)
\end{equation}
for some $C > 0$. Indeed, for $\varepsilon_0$ given in \Cref{lem-log}, a similar argument as in \Cref{lem-Lp} allows us to find $\varepsilon_1, \varepsilon_2 \in (0,1)$ such that
\begin{equation}\label{eq-eps}
\Phi(\varepsilon_1, \tau R) \leq C \bar{g} \left( \left( \fint_{B_{\tau R}} v^{\varepsilon_0} \,\mathrm{d}x \right)^{1/\varepsilon_0} \right) \quad\text{and}\quad \bar{g} \left( \left( \fint_{B_{\tau R}} v^{-\varepsilon_0} \,\mathrm{d}x \right)^{-1/\varepsilon_0} \right) \leq C \Phi(-\varepsilon_2, \tau R).
\end{equation}
By combining \eqref{eq-log} and \eqref{eq-eps}, we arrive at \eqref{eq-claim-eps}.

We now prove
\begin{equation}\label{eq-Moser1}
\Phi(\delta, \tau_1 R) \leq C \Phi(\varepsilon_1, \tau R)
\end{equation}
and
\begin{equation}\label{eq-Moser2}
\Phi(-\varepsilon_2, \tau R) \leq C \essinf_{B_{\tau_2R}} \bar{g}(v)
\end{equation}
by using Moser's iteration. Let us first prove \eqref{eq-Moser1}. If $\delta \leq \varepsilon_1$, then the result follows from Jensen's inequality. Thus, let us consider the remaining case $\delta > \varepsilon_1$. If $sp<n$, we take $\gamma_0 \in (0,1)$ sufficiently close to 1 and $\sigma \in (0, s)$ sufficiently close to $s$ so that $\delta < \gamma_0 \chi$, where $\chi = \frac{n}{n-\sigma p}$. This is possible since $\delta \in (0, \frac{n}{n-sp})$. If $sp \geq n$, we take any $\gamma_0 \in (0,1)$ and $\sigma \in (0, n/p)$ sufficiently close to $n/p$ so that $\delta < \gamma_0 \chi$. By \Cref{lem-RHI-super} (i), we have for any $\gamma \in (0,1)$ and $\tau_1 \leq \rho < \tilde{\rho} \leq \tau$,
\begin{equation}\label{eq-RHI-pos}
\Phi(\gamma \chi, \rho R) \leq \left( \frac{C}{\tilde{\rho}-\rho} \right)^{\frac{n+q}{\gamma}} \left( \frac{1}{\tau_1} \right)^{\frac{n}{\gamma}} \Phi(\gamma, \tilde{\rho}R),
\end{equation}
where we used $\mathrm{Tail}_g((\bar{u}+d)_-; x_0, \tilde{\rho} R) \leq d$. Here, the constant $C$ depends only on $n$, $p$, $q$, $s$, $\sigma$, $\Lambda$, $\gamma$, and is bounded when $\gamma$ is bounded away from 1. Let $N \in \mathbb{N}$ be the smallest integer satisfying $\gamma_0 \chi^{1-N} \leq \varepsilon_1$ and let $\rho_j=\tau_1(\tau/\tau_1)^{j/N}$ for $j=0, \dots, N$. By iterating \eqref{eq-RHI-pos} with $\rho=\rho_j$, $\tilde{\rho}=\rho_{j+1}$ and $\gamma=\gamma_0\chi^{-j} \in (0,1)$ for $j=0, \dots, N-1$, we obtain
\begin{align*}
\Phi(\delta, \tau_1R)
&\leq \Phi(\gamma_0\chi, \rho_0R) \\
&\leq \left( \prod_{j=0}^{N-1} \left( \frac{C}{\rho_{j+1}-\rho_{j}} \right)^{\frac{n+q}{\gamma_0}\chi^j} \left( \frac{1}{\tau_1} \right)^{\frac{n\chi^j}{\gamma_0}} \right) \Phi(\gamma_0\chi^{1-N}, \rho_NR) \\
&\leq C \Phi(\varepsilon_1, \tau R),
\end{align*}
where $C = C(n, p, q, s, \sigma, \Lambda, \gamma_0, \tau, \tau_1) = C(n, p, q, s, \Lambda, \delta, \tau_1, \tau_2) > 0$. This proves \eqref{eq-Moser1}.

Let us next prove \eqref{eq-Moser2}. By \Cref{lem-RHI-super} (ii), we have for any $\gamma < 0$ and $\tau_2 \leq \rho < \tilde{\rho} \leq \tau$,
\begin{equation}\label{eq-RHI-neg}
\Phi(\gamma, \tilde{\rho}R) \leq \left( \frac{C(1+|\gamma|)}{\tilde{\rho}-\rho} \right)^{\frac{n+q}{|\gamma|}} \left( \frac{1}{\tau_2} \right)^{\frac{n}{|\gamma|}} \Phi(\gamma \chi, \rho R).
\end{equation}
Let $\rho_j=\tau_2+(\tau-\tau_2)2^{-j}$ for $j=0, 1, \dots$, then by iterating \eqref{eq-RHI-neg} with $\tilde{\rho}=\rho_j$, $\rho = \rho_{j+1}$ and $\gamma=-\varepsilon_2\chi^j < 0$ for $j=0, 1, \dots, M-1$, we obtain
\begin{align*}
\Phi(-\varepsilon_{2}, \tau R)
&\leq \left( \prod_{j=0}^{M-1} \left( C \frac{2^j(1+\chi^j)}{\tau-\tau_2} \right)^{\frac{n+q}{\varepsilon_2\chi^j}} \left( \frac{1}{\tau_2} \right)^{\frac{n}{\varepsilon_2\chi^j}} \right) \Phi(-\varepsilon_2\chi^M, \rho_{M}R) \\
&\leq (C\chi)^{\frac{n+1}{\varepsilon_2} \sum j\chi^{-j}} C^{\sum \chi^{-j}} \Phi(-\varepsilon_2\chi^M, \rho_{M}R) \\
&\leq C \Phi(-\varepsilon_2\chi^M, \rho_{M}R),
\end{align*}
where $C = C(n, p, q, s, \Lambda, \tau_1, \tau_2) > 0$. Letting $M \to \infty$ leads us to \eqref{eq-Moser2}.

Finally, we combine \eqref{eq-claim-eps}, \eqref{eq-Moser1}, \eqref{eq-Moser2}, and use \Cref{lem-G} to conclude
\begin{align*}
\fint_{B_{\tau_1R}} \bar{g}^{\delta}\left( \frac{\bar{u}}{R^s} \right) \,\mathrm{d}x
&\leq \Phi^{\delta}(\delta, \tau_1 R) \leq C \bar{g}^{\delta} \left( \essinf_{B_{\tau_2R}} v \right) \\
&\leq C \bar{g}^{\delta} \left( \essinf_{B_{\tau_2R}} \frac{u}{R^s} \right) + C \bar{g}^{\delta}\left( \frac{\mathrm{Tail}_g(\bar{u}_-; x_0, R)}{R^s} \right),
\end{align*}
where $C = C(n, p, q, s, \Lambda, \delta, \tau_1, \tau_2) > 0$.
\end{proof}

%%%%%%%%%%%%%%%%%%%%%%%%%%%%%%%%%%%%%%%
\section{The obstacle problem}\label{sec-obstacle}
%%%%%%%%%%%%%%%%%%%%%%%%%%%%%%%%%%%%%%%

The nonlocal obstacle problem for the case $G(t)=t^p$ was first introduced in Korvenp\"a\"a--Kuusi--Palatucci~\cite{KKP16} as a tool in the development of nonlocal nonlinear potential theory. In this section, we generalize the nonlocal obstacle problem in two aspects: on the one hand, we study the nonlocal obstacle problem with Orlicz growth. On the other hand, we improve the classical nonlocal obstacle problem with the standard growth by considering a larger admissible set. See \Cref{rmk-obstacle}. We also refer the reader to Lieberman~\cite{Lie91b} for local obstacle problems with Orlicz growth and to Chlebicka--Karppinen~\cite{CK21} and Karppinen--Lee~\cite{KL22} for local obstacle problems with nonstandard growth.

Throughout this section, we assume that $\Omega$ and $\Omega'$ are bounded subsets of $\mathbb{R}^n$ such that $\Omega \Subset \Omega'$. Let $\psi: \Omega \to [-\infty, \infty]$ and $\vartheta \in W^{s, G}(\Omega, \Omega') \cap L^g_s(\mathbb{R}^n)$. We define a class of functions
\begin{equation}\label{eq-admissible}
\mathcal{K}_{\psi, \vartheta}(\Omega, \Omega') = \{u \in W^{s, G}(\Omega, \Omega'): u \geq \psi ~\text{a.e.\ in}~\Omega, ~ u-\vartheta \in W^{s, G}_0(\Omega, \Omega') \}.
\end{equation}
The function $\psi$ is called an \emph{obstacle}. Let us make some remarks on the class $\mathcal{K}_{\psi,\vartheta}(\Omega, \Omega')$.

\begin{remark}\label{rmk-obstacle}
Our definition of the admissible set \eqref{eq-admissible} is a bit different from the one
\begin{equation}\label{eq-admissible-KKP16}
\widetilde{\mathcal{K}}_{\psi, \vartheta}(\Omega, \Omega') = \{u \in W^{s, p}(\Omega'): u \geq \psi ~\text{a.e.~in}~\Omega, ~ u=\vartheta ~\text{a.e.\ on}~ \mathbb{R}^n \setminus \Omega \}
\end{equation}
considered in \cite{KKP16} even when $G(t)=t^p$. First of all, recall that, in the local case, the admissible class consists of functions in $W^{1, p}(\Omega)$ (see Chapter~3 in Heinonen--Kilpel\"ainen--Martio~\cite{HKM06} for instance). Thus, it is more natural to have $W^{s, p}(\Omega, \Omega')$ than $W^{s, p}(\Omega')$ since
\begin{equation*}
(1-s) \int_\Omega \int_{\Omega'} \frac{|u(x)-u(y)|^p}{|x-y|^{n+sp}} \,\mathrm{d}y \,\mathrm{d}x \to C_{n, p} \int_\Omega |\nabla u|^p \,\mathrm{d}x
\end{equation*}
as $s \nearrow 1$ whereas $(1-s)[u]_{W^{s, p}(\Omega')}^p \to C_{n, p} \|\nabla u\|_{L^p(\Omega')}^p$ as $s \nearrow 1$, see Foghem Gounoue--Kassmann--Voigt~\cite{FGKV20}. An advantage of using \eqref{eq-admissible} instead of \eqref{eq-admissible-KKP16} is that some results for another obstacle problem with an obstacle $\vartheta \in V^{s, G}(\Omega)$ and an admissible set
\begin{equation}\label{eq-admissible-V}
\mathcal{K}_{\psi, \vartheta}(\Omega) = \{u \in V^{s, G}(\Omega): u \geq \psi ~\text{a.e.~in}~\Omega, ~ u-\vartheta \in V^{s, G}_0(\Omega) \}
\end{equation}
can be deduced from the results for our obstacle problem. See \Cref{sec-other-obs} below. We note however that one cannot use the space $W^{s, p}(\Omega)$ in the definition of admissible set because we need functions that are regular across the boundary of $\Omega$ as we will see later.

Another difference between \eqref{eq-admissible} and \eqref{eq-admissible-KKP16} is that functions in \eqref{eq-admissible} achieve the boundary data $\vartheta$ in the weak sense. Notice that this coincides with the requirement $u=\vartheta ~\mathrm{a.e.}$ in $\mathbb{R}^n \setminus \Omega$ when the boundary of $\Omega$ is sufficiently regular. However, these two requirements are different in general, see Remark~2.2~(ii) in Kim--Lee--Lee \cite{KLL23}.
\end{remark}

Let us provide the definition of solution to the nonlocal obstacle problem with Orlicz growth.

\begin{definition}
We say that $u \in \mathcal{K}_{\psi, \vartheta}(\Omega, \Omega')$ is a \emph{solution to the (nonlocal) obstacle problem in $\mathcal{K}_{\psi, \vartheta}(\Omega, \Omega')$} if
\begin{equation*}
\mathcal{E}(u, v-u) \geq 0
\end{equation*}
for all $v \in \mathcal{K}_{\psi, \vartheta}(\Omega, \Omega')$.
\end{definition}

Note that $\mathcal{E}(u, w)$ is well defined for $u \in \mathcal{K}_{\psi, \vartheta}(\Omega, \Omega')$ and $w \in W^{s, G}_0(\Omega, \Omega')$ by \Cref{lem-finiteness}.

%%%%%%%%%%%%%%%%%%%%%%%%%%%%%%%%%%%%%%%
\subsection{Solvability of the obstacle problem}
%%%%%%%%%%%%%%%%%%%%%%%%%%%%%%%%%%%%%%%

We discuss the solvability of the obstacle problem in $\mathcal{K}_{\psi, \vartheta}(\Omega, \Omega')$. We need the following assumption:
\begin{equation}\label{eq-q-1}
t \mapsto g(t)/t^{q-1} \quad\text{is non-increasing}.
\end{equation}

\begin{theorem}\label{thm-obstacle}
Assume \eqref{eq-q-1}. If $\mathcal{K}_{\psi, \vartheta}(\Omega, \Omega')$ is non-empty, then there exists a unique solution to the obstacle problem in $\mathcal{K}_{\psi, \vartheta}(\Omega, \Omega')$. Moreover, the solution to the obstacle problem is a supersolution of $\mathcal{L}u=0$ in $\Omega$.
\end{theorem}

We prove the existence result in \Cref{thm-obstacle} by using a general result in the theory of monotone operators. We define a mapping $\mathcal{A}: \mathcal{K}_{\psi, \vartheta}(\Omega, \Omega') \to (W^{s, G}_0(\Omega, \Omega'))'$ by
\begin{equation*}
\langle \mathcal{A}u, w \rangle = \mathcal{E}(u, w).
\end{equation*}
Notice that a computation similar to \eqref{eq-finiteness} shows that
\begin{equation*}
|\mathcal{E}(u, w)| \leq C \left( \|u\|_{W^{s, G}(\Omega, \Omega')}^{p-1} + T(u; x_0, d) \right) \|w\|_{W^{s, G}(\Omega, \Omega')},
\end{equation*}
where $x_0$ is any point in $\Omega$ and $d=\mathrm{dist}(\Omega, \partial \Omega')$. Thus, the functional $\mathcal{A}u$ belongs to $(W^{s, G}_0(\Omega, \Omega'))'$. In order to solve the obstacle problem, it is enough to prove the following lemma, see, for instance, Corollary~III~1.8 in Kinderlehrer--Stampacchia~\cite{KS80}.

\begin{lemma}\label{lem-A}
Assume \eqref{eq-q-1}. Suppose that $\mathcal{K}_{\psi, \vartheta}(\Omega, \Omega')$ is non-empty. The mapping $\mathcal{A}$ is monotone, weakly continuous, and coercive on the set $\mathcal{K}_{\psi, \vartheta}(\Omega, \Omega')$.
\end{lemma}

\begin{proof}
We first observe that the fundamental theorem of calculus yields
\begin{equation}\label{eq-FTC}
\left( g(|a|) \frac{a}{|a|} - g(|b|) \frac{b}{|b|} \right)(a-b) \geq 0 \quad\text{for all}~ a, b \in \mathbb{R}.
\end{equation}
This proves the monotonicity of $\mathcal{A}$, i.e., $\langle \mathcal{A}u-\mathcal{A}v, u-v \rangle \geq 0$ holds for every $u, v \in \mathcal{K}_{\psi, \vartheta}(\Omega, \Omega')$.

Let us next prove the weak continuity of $\mathcal{A}$. Suppose that a sequence $\{u_j\} \subset \mathcal{K}_{\psi, \vartheta}(\Omega, \Omega')$ converges to $u \in \mathcal{K}_{\psi, \vartheta}(\Omega, \Omega')$ in $W^{s, G}(\Omega, \Omega')$. We claim that $\langle \mathcal{A}u_j-\mathcal{A}u, w \rangle \to 0$ as $j \to \infty$ for any $w \in W^{s, G}_0(\Omega, \Omega')$. Indeed, by applying \Cref{lem-ineq5} we have for any $\varepsilon>0$,
\begin{align*}
|\langle \mathcal{A}u_j - \mathcal{A}u, w \rangle|
&\leq 2\Lambda \int_{\Omega} \int_{\mathbb{R}^n} \left| g(|D^su_j|) \frac{D^su_j}{|D^su_j|} - g(|D^su|) \frac{D^su}{|D^su|} \right| |D^s w| \frac{\mathrm{d}y \,\mathrm{d}x}{|x-y|^n} \\
&\leq C \int_{\Omega} \int_{\mathbb{R}^n} g(|D^s (u_j-u)|) |D^s w| \frac{\mathrm{d}y \,\mathrm{d}x}{|x-y|^n} \\
&\quad + 2\Lambda \varepsilon \int_{\Omega} \int_{\mathbb{R}^n} (g(|D^s u_j|) + g(|D^s u|)) |D^s w| \frac{\mathrm{d}y \,\mathrm{d}x}{|x-y|^n}.
\end{align*}
Following the computation in \eqref{eq-finiteness} and noticing that $u_j-u \in W^{s, G}_0(\Omega, \Omega')$, we obtain
\begin{align*}
|\langle \mathcal{A}u_j - \mathcal{A}u, w \rangle|
&\leq C \|u_j-u\|_{W^{s, G}(\Omega, \Omega')}^{p-1} \|w\|_{W^{s, G}(\Omega, \Omega')} \\
&\quad + C \varepsilon \left( \|u_j\|_{W^{s, G}(\Omega, \Omega')}^{p-1} + T(u_j; x_0, d) \right) \|w\|_{W^{s, G}(\Omega, \Omega')} \\
&\quad + C \varepsilon \left( \|u\|_{W^{s, G}(\Omega, \Omega')}^{p-1} + T(u; x_0, d) \right) \|w\|_{W^{s, G}(\Omega, \Omega')},
\end{align*}
where $x_0$ is any point in $\Omega$ and $d=\mathrm{dist}(\Omega, \partial \Omega')$. Since $u_j \to u$ in $W^{s, G}(\Omega, \Omega')$, we arrive at
\begin{equation*}
\limsup_{j \to \infty} |\langle \mathcal{A}u_j - \mathcal{A}u, w \rangle| \leq C \varepsilon \left( \|u\|_{W^{s, G}(\Omega, \Omega')}^{p-1} + T(u; x_0, d) \right) \|w\|_{W^{s, G}(\Omega, \Omega')}.
\end{equation*}
Letting $\varepsilon \to 0$ implies the weak continuity of $\mathcal{A}$.

Finally, we prove the coercivity of $\mathcal{A}$ on $\mathcal{K}_{\psi, \vartheta}(\Omega, \Omega')$, i.e., we show that there exists a function $v \in \mathcal{K}_{\psi, \vartheta}(\Omega, \Omega')$ such that
\begin{equation*}
\frac{\langle \mathcal{A}u_j-\mathcal{A}v, u_j-v \rangle}{\|u_j-v\|_{W^{s, G}(\Omega, \Omega')}} \to \infty
\end{equation*}
whenever $u_j$ is a sequence in $\mathcal{K}_{\psi, \vartheta}(\Omega, \Omega')$ with $\|u_j\|_{W^{s, G}(\Omega, \Omega')} \to \infty$. Since $\mathcal{K}_{\psi, \vartheta}(\Omega, \Omega')$ is non-empty, we can take a function $v \in \mathcal{K}_{\psi, \vartheta}(\Omega, \Omega')$. Then we have
\begin{align*}
\langle \mathcal{A}u_j-\mathcal{A}v, u_j-v \rangle
&\geq \frac{p}{\Lambda} \int_\Omega \int_{\mathbb{R}^n} (G(|D^s u_j|) + G(|D^s v|)) \frac{\mathrm{d}y\,\mathrm{d}x}{|x-y|^n} \\
&\quad - 2\Lambda \int_\Omega \int_{\mathbb{R}^n} (g(|D^s u_j|) |D^s v| + g(|D^s v|) |D^s u_j|) \frac{\mathrm{d}y\,\mathrm{d}x}{|x-y|^n} = I_1+I_2.
\end{align*}
By using \Cref{lem-G} and \Cref{thm-Poincare}, we obtain
\begin{equation*}
I_1 \geq C \int_\Omega \int_{\mathbb{R}^n} G(|D^s (u_j-v)|) \frac{\mathrm{d}y\,\mathrm{d}x}{|x-y|^n} \geq C \left( \|u_j-v\|_{W^{s, G}(\Omega, \Omega')}^p \land \|u_j-v\|_{W^{s, G}(\Omega, \Omega')}^q \right).
\end{equation*}
For $I_2$, we utilize the techniques used above to estimate
\begin{align*}
I_2
&\geq - C \int_\Omega \int_{\mathbb{R}^n} (g(|D^s (u_j-v)|) + g(|D^s v|)) |D^s v| \frac{\mathrm{d}y\,\mathrm{d}x}{|x-y|^n} \\
&\quad - C \int_\Omega \int_{\mathbb{R}^n} g(|D^s v|) (|D^s(u_j-v)|+|D^sv|) \frac{\mathrm{d}y\,\mathrm{d}x}{|x-y|^n} \\
&\geq - C \|u_j-v\|_{W^{s, G}(\Omega, \Omega')}^{p-1} \|v\|_{W^{s, G}(\Omega, \Omega')} - C \left( \|v\|_{W^{s, G}(\Omega, \Omega')}^p \land \|v\|_{W^{s, G}(\Omega, \Omega')}^q \right) \\
&\quad - C \left( \|v\|_{W^{s, G}(\Omega, \Omega')}^{p-1} + T(\vartheta; x_0, d) \right) \|u_j-v\|_{W^{s, G}(\Omega, \Omega')}.
\end{align*}
Since $\|u_j-v\|_{W^{s, G}(\Omega, \Omega')} \to \infty$ as $\|u_j\|_{W^{s, G}(\Omega, \Omega')} \to \infty$, we conclude the coercivity of $\mathcal{A}$.
\end{proof}

\begin{corollary}\label{cor-obstacle}
Let $u$ be the solution to the obstacle problem in $\mathcal{K}_{\psi, \vartheta}(\Omega, \Omega')$. If $B_R =B_R(x_0) \subset \Omega$ is such that
	\begin{equation*}
		\essinf_{B_R}(u-\psi)>0,
	\end{equation*}
	then $u$ is a solution of $\mathcal{L}u=0$ in $B_R$. In particular, if $u$ is lower semicontinuous and $\psi$ is upper semicontinuous in $\Omega$, then $u$ is a solution of $\mathcal{L}u=0$ in $\{x \in \Omega: u(x)>\psi(x)\}$.
\end{corollary}

For the proofs of \Cref{thm-obstacle} and \Cref{cor-obstacle}, we refer the reader to the proofs of Theorem~1 and Corollary~1 in Korvenp\"a\"a--Kuusi--Palatucci~\cite{KKP16}.

%%%%%%%%%%%%%%%%%%%%%%%%%%%%%%%%%%%%%%%
\subsection{Local boundedness}
%%%%%%%%%%%%%%%%%%%%%%%%%%%%%%%%%%%%%%%

Solutions to the obstacle problems enjoy some regularity properties. This section is devoted to the the local boundedness of solutions up to the boundary. Notice that we interpret $\esssup_A = -\infty$ and $\essinf_A = \infty$ when $A=\emptyset$.

\begin{theorem}\label{thm-obs-bdd}
Let $u$ be the solution to the obstacle problem in $\mathcal{K}_{\psi, \vartheta}(\Omega, \Omega')$ and let $B_R=B_R(x_0)$ satisfy $B_R \cap \Omega \neq \emptyset$. Let $\varepsilon \in (0,1)$ and $p_0>0$.
\begin{enumerate}[(i)]
\item
If
\begin{equation}\label{eq-obs-M}
M:= \max \left\lbrace \esssup_{B_R \setminus \Omega} \vartheta, \esssup_{B_R \cap \Omega} \psi \right\rbrace < \infty,
\end{equation}
then
\begin{equation}\label{eq-obs-bdd-upper}
\esssup_{B_{R/2}}\,(u-k)_+ \leq \varepsilon\, \mathrm{Tail}_g((u-k)_+; x_0, R/2) + C \left( \fint_{B_R} (u-k)_+^{p_0} \,\mathrm{d}x \right)^{1/p_0}
\end{equation}
for all $k \geq M$, where $C=C(n, p, p_0, q, s, \Lambda, \varepsilon)>0$.
\item
If
\begin{equation}\label{eq-obs-m}
m:= \essinf_{B_R \setminus \Omega} \vartheta > -\infty,
\end{equation}
then
\begin{equation*}
\esssup_{B_{R/2}}\,(u-k)_- \leq \varepsilon\, \mathrm{Tail}_g((u-k)_-; x_0, R/2) + C \left( \fint_{B_R} (u-k)_-^{p_0} \,\mathrm{d}x \right)^{1/p_0}
\end{equation*}
for all $k \leq m$, where $C=C(n, p, p_0, q, s, \Lambda, \varepsilon)>0$.
\end{enumerate}
\end{theorem}

\begin{proof}
Let us first assume \eqref{eq-obs-M}. Let $B_\rho(y) \subset B_R$. Then for any $k \geq M$ and for any functions $\eta \in C^\infty_c(B_\rho(y))$ satisfying $\eta \in [0,1]$, we have $v=u-(u-k)_+ \eta^q \in \mathcal{K}_{\psi, \vartheta}(\Omega, \Omega')$, and hence $\mathcal{E}(u, (u-k)_+\eta^q) \leq 0$. Thus, $u$ satisfies a Caccioppoli-type estimate (see Equation~(3.10) in Chaker--Kim--Weidner~\cite{CKW22}) for all $k \geq M$, which is enough to deduce
\begin{equation}\label{eq-est-above}
\esssup_{B_{\rho/2}(y)}\, (u-k)_+ \leq \varepsilon_0 \mathrm{Tail}_g((u-k)_+; y, \rho/2) + C\rho^s G^{-1} \left( \fint_{B_{\rho}(y)} G\left( \frac{(u-k)_+}{\rho^s} \right) \,\mathrm{d}x \right),
\end{equation}
for all $k \geq M$ and $\varepsilon_0>0$, where $C>0$ is a constant depending only on $n$, $p$, $q$, $s$, $\Lambda$, and $\varepsilon_0$. See the proof of Theorem~3.1 in Chaker--Kim--Weidner~\cite{CKW23}. One can now follow the argument in the proof of \Cref{thm-loc-bdd-bdry} and use
\begin{equation*}
G^{-1} \left( \fint_{B_\rho(y)} G \left( \frac{(u-k)_+}{\rho^s} \right) \,\mathrm{d}x \right) \leq C \left( \fint_{B_\rho(y)} \left( \frac{(u-k)_+}{\rho^s} \right)^\alpha \,\mathrm{d}x \right)^{1/\alpha}
\end{equation*}
for large $\alpha>1$ (by modifying \Cref{lem-Lp}) instead of \eqref{eq-alpha} to deduce \eqref{eq-obs-bdd-upper} from \eqref{eq-est-above}.

If \eqref{eq-obs-m} holds, then we repeat the same argument with $v=u+(u-k)_- \eta^q \in \mathcal{K}_{\psi, \vartheta}(\Omega, \Omega')$ for $k \leq m$.
\end{proof}

%%%%%%%%%%%%%%%%%%%%%%%%%%%%%%%%%%%%%%%
\subsection{H\"older regularity}
%%%%%%%%%%%%%%%%%%%%%%%%%%%%%%%%%%%%%%%
We prove H\"older continuity of solutions to the obstacle problem. Let us begin with the interior estimates. Here, for an obstacle $\psi$, we define
\begin{align*}
	\omega_{\psi}(\rho)\equiv\omega_{\psi}(\rho, x_0):=\osc_{B_{\rho}(x_0) \cap \Omega}\psi \quad \text{for $\rho>0$}.
\end{align*}

\begin{theorem}\label{thm-Holder-int}
Suppose that $\psi$ is locally H\"older continuous in $\Omega$ or $\psi \equiv -\infty$. Then the solution $u$ to the obstacle problem in $\mathcal{K}_{\psi,\vartheta}(\Omega, \Omega')$ has a representative which is locally H\"older continuous in $\Omega$.
\end{theorem}

\begin{proof}
Let us fix $x_0 \in \Omega$. If
\begin{equation*}
\essinf_{B_R} (u-\psi) > 0
\end{equation*}
for some $B_R=B_R(x_0) \subset \Omega$, then $u$ is a solution of $\mathcal{L}u=0$ in $B_R$ by \Cref{cor-obstacle}, and so is $u-\psi(x_0)$. Thus, by Theorem~1.2 in Chaker--Kim--Weidner~\cite{CKW22}
\begin{equation*}
\essosc_{B_r}u \leq C \left( \frac{r}{R} \right)^\alpha \left( \mathrm{Tail}_g(u-\psi(x_0); x_0, R/2) + \|u-\psi(x_0)\|_{L^\infty(B_{R/2})} \right)
\end{equation*}
for all $r \in (0, R/8]$.

Suppose now that
\begin{equation}\label{eq-alternative}
\essinf_{B_R} (u-\psi) = 0
\end{equation}
for every $B_R=B_R(x_0) \subset \Omega$. We claim that that there exist constants $\tau \in (0,1)$ and $C>0$, depending only on $n$, $p$, $q$, $s$, and $\Lambda$, such that
\begin{equation}\label{eq-claim-Holder}
\essosc_{B_{\tau R}}u + \mathrm{Tail}_g(u-\psi(x_0); x_0, \tau R) \leq \frac{1}{2} \left( \essosc_{B_R}u + \mathrm{Tail}_g(u-\psi(x_0); x_0, R) \right) + C\omega_\psi(R)
\end{equation}
whenever $B_R \subset \Omega$. To prove the claim, we consider a function $u_d:=u-d$, where $d=\psi(x_0)-\omega_\psi(R)$, which is nonnegative in $B_R$ by \eqref{eq-alternative}. Let $r \in (0, R/4]$. Since $u_d$ is a supersolution of $\mathcal{L}u_d=0$ in $B_{4r} \subset B_R$, \Cref{thm-WHI-int} provides
\begin{align*}
(4r)^s g^{-1} \left( \fint_{B_{2r}} g\left( \frac{u_d}{(4r)^{s}} \right) \,\mathrm{d}x \right)
&\leq C \essinf_{B_{2r}} u_d + C\, \mathrm{Tail}_{g}((u_d)_-; x_0, 4r) \\
&\leq C \omega_\psi(R) + C \left(\frac{r}{R} \right)^{\frac{s}{q-1}} \mathrm{Tail}_{g}(u_d; x_{0}, R).
\end{align*}
On the other hand, \Cref{thm-obs-bdd} with $k=d+2\omega_\psi(R)$ shows that
\begin{equation}\label{eq-sup-est}
\esssup_{B_r}\,u_d \leq 2\omega_\psi(R) + \varepsilon \,\mathrm{Tail}_g((u_d)_+; x_0, r) + C(\varepsilon) \left( \fint_{B_{2r}} u_d^{p_0} \,\mathrm{d}x \right)^{1/p_0}
\end{equation}
for any $p_0>0$ and $\varepsilon \in (0, 1)$. By taking sufficiently small $p_0$ (and using \Cref{lem-Lp}), we can combine two displays above. Then we obtain
\begin{equation*}
\essosc_{B_r} u \leq \esssup_{B_r}\,u_d \leq C\omega_\psi(R) + \varepsilon \,\mathrm{Tail}_g(u_d; x_0, r) + C \left(\frac{r}{R} \right)^{\frac{sq}{q-1}} \mathrm{Tail}_g(u_d; x_0, R).
\end{equation*}
Since
\begin{align*}
\mathrm{Tail}_g(u_d; x_0, r)
&\leq C\esssup_{B_R} u_d + C\left( \frac{r}{R} \right)^{\frac{sq}{q-1}} \mathrm{Tail}_g(u_d; x_0, R) \\
&\leq C\left( \essosc_{B_R} u + 2\omega_\psi(R) \right) + C\left( \frac{r}{R} \right)^{\frac{sq}{q-1}} \mathrm{Tail}_g(u_d; x_0, R),
\end{align*}
we deduce
\begin{equation*}
\essosc_{B_r} u \leq C \varepsilon \essosc_{B_R} u + C\omega_\psi(R) + C(\varepsilon) \left( \frac{r}{R} \right)^{\frac{sq}{q-1}} \mathrm{Tail}_g(u_d; x_0, R).
\end{equation*}
For any $\varepsilon' \in (0,1)$, we choose $\varepsilon$ and then $\tilde{\tau} \in (0,1/4 )$ accordingly so that $C\varepsilon \leq \varepsilon'$ and $C(\varepsilon) \tilde{\tau}^{\frac{sq}{q-1}} \leq \varepsilon'$. Then for $r=\tilde{\tau}R$ we have
\begin{equation}\label{eq-tilde-tau}
\essosc_{B_{\tilde{\tau}R}} u \leq \varepsilon' \left( \essosc_{B_R} u + \mathrm{Tail}_g(u-\psi(x_0); x_0, r) \right)+ C\omega_\psi(r).
\end{equation}
Now, for any $\tau \in (0, \tilde{\tau})$ we obtain from \eqref{eq-alternative} and \eqref{eq-tilde-tau} that
\begin{align*}
\mathrm{Tail}_g(u-\psi(x_0); x_0, \tau R)
&\leq C \essosc_{B_{\tilde{\tau}R}} u + C\omega_\psi(R) + C \left( \frac{\tau}{\tilde{\tau}} \right)^{\frac{sq}{q-1}} \mathrm{Tail}_g(u-\psi(x_0); x_0, \tilde{\tau}R) \\
&\leq C \varepsilon' \left( \essosc_{B_R} u + \mathrm{Tail}_g(u-\psi(x_0); x_0, R) \right) + C\omega_\psi(R) \\
&\quad + C \left( \frac{\tau}{\tilde{\tau}} \right)^{\frac{sq}{q-1}} \left( \essosc_{B_R}u + \omega_\psi(R) + \mathrm{Tail}_g(u-\psi(x_0); x_0, R) \right).
\end{align*}
Therefore, we arrive at the claim \eqref{eq-claim-Holder} by taking $\tau$ and $\varepsilon'$ so small that $(\tau/\tilde{\tau})^{sq/(q-1)} \leq \varepsilon'$ and $2C\varepsilon' \leq 1/2$.

The desired result follows by iterating \eqref{eq-claim-Holder}, using \eqref{eq-sup-est}, and recalling that $\psi$ is locally H\"older continuous.
\end{proof}

We move our attention to the H\"older continuity of solutions up to the boundary. To this end, we need to impose some regularity assumption on $\Omega$. We say that a measurable set $E \subset \mathbb{R}^n$ satisfies a \emph{measure density condition} if there exist $R_0>0$ and $\delta \in (0, 1)$ such that
\begin{equation*}
\inf_{R \in (0, R_0)} \frac{|E \cap B_R(x_0)|}{|B_R(x_0)|} \geq \delta
\end{equation*}
for every $x_0 \in \partial E$. Note that if $D$ and $\Omega$ are open sets such that $D \Subset \Omega$, there always exists an open set $U$ such that $D \Subset U \Subset \Omega$ with $\mathbb{R}^n \setminus U$ satisfying the measure density condition.

\begin{theorem}\label{thm-Holder-bdry}
Assume that $\mathbb{R}^n \setminus \Omega$ satisfies the measure density condition with $R_0>0$ and $\delta \in (0,1)$. Let $B_R=B_R(x_0)$ satisfy $B_R \cap \Omega \neq \emptyset$ and $R\leq R_0$. If $\vartheta \in \mathcal{K}_{\psi, \vartheta}(\Omega, \Omega')$ is H\"older continuous in $B_R \setminus \Omega$ and if $\psi$ is H\"older continuous in $B_R \cap \Omega$ or $h \equiv -\infty$, then the solution $u$ to the obstacle problem in $\mathcal{K}_{\psi,\vartheta}(\Omega, \Omega')$ has a representative which is H\"older continuous in $B_R$.
\end{theorem}

\begin{proof}
If $\essinf_{B_r}(u-\psi)=0$ for every $r \in (0, R)$, then one can follow the lines of proof of \Cref{thm-Holder-int} and use \Cref{thm-WHI-bdry} instead of \Cref{thm-WHI-int} to conclude theorem. Thus, the remaining part of the proof focuses on the case
\begin{equation*}
\essinf_{B_{r_0}}(u-\psi) > 0 \quad\text{for some}~ r_0 \in (0, R).
\end{equation*}
Note that in this case $u$ is a solution of $\mathcal{L}u=0$ in $B_{r_0} \cap \Omega$ by \Cref{cor-obstacle}.

We claim that there exist constants $\alpha \in (0,1)$, $\lambda >0$, a non-increasing sequence $\{M_j\}_{j\geq0}$, and a non-decreasing sequence $\{m_j\}_{j \geq0}$ such that $m_j \leq \essinf_{B^j}u \leq \esssup_{B^j} u \leq M_j$ and $M_j-m_j=\lambda r_j^\alpha$, where $r_j=2^{-j}r_0$ and $B^j = B_{r_j}$ for $j \geq 0$. We argue by induction on $j$. For $j=0$, we choose $M_0=\esssup_{B^0}u$, $m_0=M_0-\lambda r_0^\alpha$, and
\begin{equation}\label{eq-lambda}
\lambda \geq 2r_0^{-\alpha} \|u\|_{L^\infty(B^0)}
\end{equation}
so that $\essinf_{B^0}u \geq m_0$.

Assume now that we have constructed sequences $\{M_j\}$ and $\{m_j\}$ up to the index $j$. Since $v:=M_j-u$ is a solution of $\mathcal{L}v=0$ in $B_{r_0} \cap \Omega$ such that $v \geq 0$ in $B^j$, \Cref{thm-WHI-bdry} shows that
\begin{equation*}
r_j^s g^{-1}\left( \fint_{B^{j+1}} g \left( r_j^{-s}w \right) \,\mathrm{d}x \right) \leq C \essinf_{B^{j+1}} w + C\,\mathrm{Tail}_g(w_-; x_0, r_j),
\end{equation*}
where $w=v_{M_j-\widetilde{M}_j}^-$ and $\widetilde{M}_j=\sup_{B^j \setminus \Omega}\vartheta$. On the one hand, we have from $w = M_j-\widetilde{M}_j$ in $B^{j+1} \setminus \Omega$ that
\begin{equation*}
r_j^s g^{-1}\left( \fint_{B^{j+1}} g \left( r_j^{-s}w \right) \,\mathrm{d}x \right) \geq r_j^s g^{-1}\left( \frac{|B^{j+1} \setminus \Omega|}{|B^{j+1}|} g \left( \frac{M_j-\widetilde{M}_j}{r_j^s} \right) \right) \geq c \left( M_j - \widetilde{M}_j \right)
\end{equation*}
for some $c = c(p, q, \delta)>0$, where we used the measure density condition and \Cref{lem-G}. On the other hand, we have
\begin{equation*}
\essinf_{B^{j+1}} w \leq M_j - M_{j+1},
\end{equation*}
and hence
\begin{equation}\label{eq-induction}
M_j - \widetilde{M}_j \leq C(M_j - M_{j+1}) + C \,\mathrm{Tail}_g(w_-; x_0, r_j).
\end{equation}
In order to estimate the tail term on the right-hand side of \eqref{eq-induction}, we use the inductive hypothesis; for $0 \leq k \leq j-1$, we have
\begin{align*}
&w_- \leq M_k - M_j \leq M_k-m_k -(M_j-m_j) \leq \lambda (r_k^\alpha - r_j^\alpha) \quad\text{in}~ B^k \setminus B^{k+1} ~\text{and} \\
&w_- \leq (M_j-u)_- \leq \|u\|_{L^\infty(B^0)} + |u| \quad\text{outside}~ B^0.
\end{align*}
Thus, we obtain $\mathrm{Tail}_g(w_-; x_0, r_j) \leq C(I_1+I_2)$, where
\begin{align*}
I_1 &= r_j^s g^{-1} \left( r_j^s \sum_{k=0}^{j-1} \int_{B^k \setminus B^{k+1}} g\left( \lambda \frac{r_k^\alpha - r_j^\alpha}{|x-x_0|^s} \right) \frac{\mathrm{d}x}{|x-x_0|^{n+s}} \right) \quad\text{and} \\
I_2 &= r_j^s g^{-1} \left( r_j^s \int_{\mathbb{R}^n \setminus B^0} g\left( \frac{\|u\|_{L^\infty(B^0)} + |u|}{|x-x_0|^s} \right) \frac{\mathrm{d}x}{|x-x_0|^{n+s}} \right).
\end{align*}
We observe that we have
\begin{align*}
I_1
&\leq C r_j^s g^{-1}\left( r_j^s \sum_{k=0}^{j-1} g\left( \lambda \frac{r_k^\alpha - r_j^\alpha}{r_{k+1}^s} \right) \frac{1}{r_{k+1}^s} \right) \\
&\leq C r_j^s g^{-1} \left( \sum_{k=0}^{j-1} g\left( \lambda r_j^{\alpha-s} \frac{2^{\alpha(j-k)} - 1}{2^{s(j-k)}} \right) \frac{1}{2^{s(j-k)}} \right) \\
&\leq C r_j^s g^{-1} \left( g( \lambda r_j^{\alpha-s}) \sum_{k=0}^{j-1} \frac{(2^{\alpha(j-k)} - 1)^{p-1}}{2^{sp(j-k)}} \right) \\
&\leq C S(\alpha) \lambda r_j^\alpha
\end{align*}
by using \Cref{lem-G}, where
\begin{equation*}
S(\alpha) = \left( \sum_{k=0}^\infty \frac{(2^{\alpha k} - 1)^{p-1}}{2^{spk}} \right)^{\frac{1}{p-1}} \lor \left( \sum_{k=0}^\infty \frac{(2^{\alpha k} - 1)^{p-1}}{2^{spk}} \right)^{\frac{1}{q-1}}.
\end{equation*}
Note that $S(\alpha) \to 0$ as $\alpha \to 0$. Moreover, we obtain
\begin{align*}
I_2
&\leq r_j^s g^{-1} \left( r_j^s \int_{\mathbb{R}^n \setminus B^0} g\left( \frac{\|u\|_{L^\infty(B^0)}}{|x-x_0|^s} \right) \frac{\mathrm{d}x}{|x-x_0|^{n+s}} + \left( \frac{r_j}{r_0} \right)^s T(u; x_0, r_0) \right) \\
&\leq C r_j^s g^{-1} \left( \left( \frac{r_j}{r_0} \right)^s g\left( \frac{\|u\|_{L^\infty(B^0)}}{r_0^s} \right) + \left( \frac{r_j}{r_0} \right)^s T(u; x_0, r_0) \right) \\
&\leq C (r_j/r_0)^{\frac{sq}{q-1}} \left( \|u\|_{L^\infty(B^0)} + \mathrm{Tail}(u; x_0, r_0) \right) \\
&\leq C \frac{\|u\|_{L^\infty(B^0)} + \mathrm{Tail}_g(u; x_0, r_0)}{r_0^\alpha} r_j^\alpha
\end{align*}
by using \Cref{lem-G} and assuming $\alpha < sq/(q-1)$. Therefore, we arrive at
\begin{equation*}
M_j - \widetilde{M}_j \leq C(M_j - M_{j+1}) + C \left( S(\alpha) \lambda + \frac{\|u\|_{L^\infty(B^0)} + \mathrm{Tail}_g(u; x_0, r_0)}{r_0^\alpha} \right) r_j^\alpha.
\end{equation*}
Similarly, by using $u-m_j$ we obtain
\begin{equation*}
\widetilde{m}_j-m_j \leq C (m_{j+1}-m_j) + C \left( S(\alpha) \lambda + \frac{\|u\|_{L^\infty(B^0)} + \mathrm{Tail}_g(u; x_0, r_0)}{r_0^\alpha} \right) r_j^\alpha,
\end{equation*}
where $\widetilde{m}_j = \inf_{B^j \setminus \Omega} \vartheta$. By addition we get
\begin{align*}
\essosc_{B^{j+1}}u
&\leq (1-\gamma) \lambda r_j^\alpha + \osc_{B^j \setminus \Omega}\vartheta + C \left( S(\alpha) \lambda + \frac{\|u\|_{L^\infty(B^0)} + \mathrm{Tail}_g(u; x_0, r_0)}{r_0^\alpha} \right) r_j^\alpha \\
&\leq  2^\alpha(1-\gamma+CS(\alpha)) \lambda r_{j+1}^\alpha + [\vartheta]_{C^\beta(\overline{B^0})} (2r_j)^\beta + C \frac{\|u\|_{L^\infty(B^0)} + \mathrm{Tail}_g(u; x_0, r_0)}{r_0^\alpha} r_{j+1}^\alpha
\end{align*}
for some $\gamma \in (0,1)$, where $\beta$ is the H\"older exponent of $\vartheta$. Assuming $\alpha < \beta$, we have $(2r_j)^\beta \leq 2^{\alpha+\beta} r_0^{\beta-\alpha} r_{j+1}^\alpha$. We now choose $\alpha \in (0, \min\{\frac{sq}{q-1}, \beta \})$ sufficiently small so that $2^\alpha(1-\gamma+CS(\alpha)) \leq 1-\gamma/2$ and then we set
\begin{equation*}
\lambda = \frac{2}{\gamma} \left( 2^{\alpha+\beta} r_0^\beta [\vartheta]_{C^\beta(\overline{B^0})} + C \|u\|_{L^\infty(B^0)} + C\,\mathrm{Tail}_g(u; x_0, r_0) \right) r_0^\alpha,
\end{equation*}
which is in accordance with \eqref{eq-lambda}. Then we have
\begin{equation*}
\essosc_{B^{j+1}}u \leq \lambda r_{j+1}^\alpha.
\end{equation*}
We may choose $M_{j+1}$ and $m_{j+1}$ so that
\begin{equation*}
m_j \leq m_{j+1} \leq \essinf_{B^{j+1}}u \leq \esssup_{B^{j+1}}u \leq M_{j+1} \leq M_j, \quad M_{j+1} - m_{j+1} = \lambda r_{j+1}^\alpha,
\end{equation*}
which completes the induction.
\end{proof}

A slight modification of the proof of \Cref{thm-Holder-int} (or \Cref{thm-Holder-bdry}) reveals that $u$ has a representative which is continuous in $\Omega$ (or $B_R$, respectively) if $\psi$ is continuous in $\Omega$ or $\psi \equiv -\infty$ (or $\vartheta \in \mathcal{K}_{\psi, \vartheta}(\Omega, \Omega')$ is continuous in $B_R \setminus \Omega $ and if $\psi$ is continuous in $B_R \cap \Omega$ or $h \equiv -\infty$).

%%%%%%%%%%%%%%%%%%%%%%%%%%%%%%%%%%%%%%%
\subsection{Other obstacle problems}\label{sec-other-obs}
%%%%%%%%%%%%%%%%%%%%%%%%%%%%%%%%%%%%%%%

As mentioned in \Cref{rmk-obstacle}, the results obtained in this section can be translated into those for other types of obstacle problems. For instance, we say that $u \in \mathcal{K}_{\psi, \vartheta}(\Omega)$, where $\mathcal{K}_{\psi, \vartheta}(\Omega)$ is defined as in \eqref{eq-admissible-V}, is a solution to the obstacle problem in $\mathcal{K}_{\psi, \vartheta}(\Omega)$ if $\mathcal{E}(u, v-u) \geq 0$ for all $v \in \mathcal{K}_{\psi, \vartheta}(\Omega)$. Then we have the following results.

\begin{theorem}\label{cor-1}
If $\mathcal{K}_{\psi, \vartheta}(\Omega)$ is non-empty, then there exists a unique solution $u$ to the obstacle problem in $\mathcal{K}_{\psi, \vartheta}(\Omega)$. Moreover, $u$ is a supersolution of $\mathcal{L}u=0$ in $\Omega$.
\end{theorem}

We point out that the condition \eqref{eq-q-1} is not required in \Cref{cor-1}.

\begin{proof}[Proof of \Cref{cor-1}]
Recall that the assumption \eqref{eq-q-1} is used only when we check the weak continuity of the mapping $\mathcal{A}$ in \Cref{lem-A}. Since the monotonicity and coercivity of $\mathcal{A}: \mathcal{K}_{\psi,\vartheta}(\Omega) \to (V_0^{s, G}(\Omega))'$ can be proved in the same way, let us provide the proof of the weak continuity. Suppose that a sequence $\{u_j\} \subset \mathcal{K}_{\psi, \vartheta}(\Omega)$ converges to $u \in \mathcal{K}_{\psi, \vartheta}(\Omega)$ in $V^{s, G}(\Omega)$. We choose a subsequence $\{u_{j_i}\}$ such that $u_{j_i}\to u$ a.e.\ in $\Omega$. Since $L^{G^\ast}(\Omega \times \mathbb{R}^n; |x-y|^{-n}\,\mathrm{d}y\,\mathrm{d}x)$-norms of $g(|D^su_{j_i}|)\frac{D^su_{j_i}}{|D^su_{j_i}|}$ are uniformly bounded, $g(|D^su_{j_i}|)\frac{D^su_{j_i}}{|D^su_{j_i}|}$ weakly converges to $g(|D^su|)\frac{D^su}{|D^su|}$ in $L^{G^\ast}(\Omega \times \mathbb{R}^n; |x-y|^{-n}\,\mathrm{d}y\,\mathrm{d}x)$. Since the weak limit is independent of the choice of the subsequence, $g(|D^su_j|)\frac{D^su_j}{|D^su_j|}$ weakly converges to $g(|D^su|)\frac{D^su}{|D^su|}$ in $L^{G^\ast}(\Omega \times \mathbb{R}^n; |x-y|^{-n}\,\mathrm{d}y\,\mathrm{d}x)$. Therefore, we have $\langle \mathcal{A}u_j-\mathcal{A}u, w \rangle \to 0$ as $j \to \infty$ for all $w \in V^{s, G}_0(\Omega)$.
\end{proof}

\begin{corollary}\label{cor-2}
Let $u$ be the solution to the obstacle problem in $\mathcal{K}_{\psi, \vartheta}(\Omega)$. If $B_R =B_R(x_0) \subset \Omega$ is such that
\begin{equation*}
\essinf_{B_R}(u-\psi)>0,
\end{equation*}
then $u$ is a solution of $\mathcal{L}u=0$ in $B_R$. In particular, if $u$ is lower semicontinuous and $\psi$ is upper semicontinuous in $\Omega$, then $u$ is a solution of $\mathcal{L}u=0$ in $\{x \in \Omega: u(x)>\psi(x)\}$.
\end{corollary}

\begin{theorem}\label{cor-3}
Let $u$ be the solution to the obstacle problem in $\mathcal{K}_{\psi, \vartheta}(\Omega)$ and let $B_R=B_R(x_0) \subset \Omega'$ satisfy $B_R \cap \Omega \neq \emptyset$. If \eqref{eq-obs-M} (or \eqref{eq-obs-m}, respectively) holds, then $u$ is locally essentially bounded from above (or below, respectively) in $B_R$.
\end{theorem}

\begin{theorem}\label{cor-4}
If $\psi$ is continuous (locally H\"older continuous, respectively) or $\psi \equiv -\infty$, then the solution $u$ to the obstacle problem in $\mathcal{K}_{\psi,\vartheta}(\Omega)$ has a representative which is continuous (locally H\"older continuous, respectively) in $\Omega$.
\end{theorem}

\begin{proof}[Proofs of \Cref{cor-2}--\Cref{cor-4}]
One can just follow the lines of proofs of \Cref{cor-obstacle}, \Cref{thm-obs-bdd}, and \Cref{thm-Holder-int}. Indeed, the only difference is the function spaces for test functions. Namely, one has to check if test functions belong to $\mathcal{K}_{\psi, \vartheta}(\Omega)$, not $\mathcal{K}_{\psi, \vartheta}(\Omega, \Omega')$, which is obvious.
\end{proof}

%%%%%%%%%%%%%%%%%%%%%%%%%%%%%%%%%%%%%%%
\section{Properties of supersolutions}\label{sec-supersolution}
%%%%%%%%%%%%%%%%%%%%%%%%%%%%%%%%%%%%%%%

In this section, we provide some properties of supersolutions, which are crucial in the development of nonlocal nonlinear potential theory.

%%%%%%%%%%%%%%%%%%%%%%%%%%%%%%%%%%%%%%%
\subsection{Lower semicontinuity}
%%%%%%%%%%%%%%%%%%%%%%%%%%%%%%%%%%%%%%%

In Byun--Kim--Ok~\cite{BKO23} and Chaker--Kim--Weidner~\cite{CKW22}, it has been shown that solutions of $\mathcal{L}u=0$ in $\Omega$ have a representative which is H\"older continuous in $\Omega$. More generally, we proved in the previous section that the solution to the nonlocal obstacle problem with Orlicz growth has a representative which is (H\"older) continuous in $\Omega$ when the obstacle is (H\"older) continuous. For a supersolution of $\mathcal{L}u=0$ in $\Omega$, one can find a representative which is lower semicontinuous in $\Omega$.

\begin{theorem}
Let $u$ be a supersolution of $\mathcal{L}u=0$ in $\Omega$. Then
	\begin{equation*}
		u(x)=\essliminf_{y \to x}u(y) \quad \text{for a.e.~}x \in \Omega.
	\end{equation*}
	In particular, $u$ has a representative which is lower semicontinuous in $\Omega$.
\end{theorem}

\begin{proof}
The proof is essentially the same as in Theorem~9 in Korvenp\"a\"a--Kuusi--Palatucci~\cite{KKP17}. One can utilize \Cref{thm-loc-bdd-int} instead of \cite[Theorem~4]{KKP17} in the proof.
\end{proof}

%%%%%%%%%%%%%%%%%%%%%%%%%%%%%%%%%%%%%%%
\subsection{Comparison principle}
%%%%%%%%%%%%%%%%%%%%%%%%%%%%%%%%%%%%%%%

Supersolutions and subsolutions satisfy the following form of the comparison principle.

\begin{lemma}[Comparison principle]\label{lem-comparison}
Let $\Omega, \Omega'$ be open subsets of $\mathbb{R}^n$ such that $\Omega \Subset \Omega'$. Let $u \in W^{s, G}(\Omega, \Omega')$ be a supersolution of $\mathcal{L}u=0$ in $\Omega$ and $v \in W^{s, G}(\Omega, \Omega')$ be a subsolution of $\mathcal{L}v=0$ in $\Omega$ such that $(u-v)_- \in W^{s, G}_0(\Omega, \Omega')$. Then $u \geq v$ a.e.\ in $\mathbb{R}^n$.
\end{lemma}

\begin{proof}
By using $\varphi:=(u-v)_-$ as a test function, we obtain
\begin{equation*}
0 \leq \mathcal{E}(u, \varphi) - \mathcal{E}(v, \varphi) = I_1+I_2+I_3,
\end{equation*}
where
\begin{align*}
I_1 &= \int_{\{u < v\}}\int_{\{u < v\}} \left( g(|D^su|)\frac{D^su}{|D^su|} - g(|D^sv|)\frac{D^sv}{|D^sv|} \right) D^s\varphi \frac{k(x, y)}{|x-y|^n} \,\mathrm{d}y \,\mathrm{d}x, \\
I_2 &= \int_{\{u < v\}}\int_{\{u \geq v\}} \left( g(|D^su|)\frac{D^su}{|D^su|} - g(|D^sv|)\frac{D^sv}{|D^sv|} \right) \frac{\varphi(x)}{|x-y|^s} \frac{k(x, y)}{|x-y|^n} \,\mathrm{d}y \,\mathrm{d}x, \\
I_3 &= \int_{\{u \geq v\}}\int_{\{u < v\}} \left( g(|D^su|)\frac{D^su}{|D^su|} - g(|D^sv|)\frac{D^sv}{|D^sv|} \right) \frac{-\varphi(y)}{|x-y|^s} \frac{k(x, y)}{|x-y|^n} \,\mathrm{d}y \,\mathrm{d}x.
\end{align*}
Since $D^s\varphi = -(D^su-D^sv)$ on $\{u<v\} \times \{u<v\}$, \eqref{eq-FTC} yields $I_1 \leq 0$. If $u(x)<v(x)$ and $u(y)\geq v(y)$, then $u(x)-u(y)<v(x)-v(y)$ and hence the integrand of $I_2$ is non-positive. This shows that $I_2 \leq 0$, and similarly $I_3 \leq 0$. Therefore, each integral term must vanish and, in particular, $I_1=0$. In other words, we have $|\{u<v\}|=0$ and so $u \geq v$ a.e.\ in $\mathbb{R}^n$.
\end{proof}

%%%%%%%%%%%%%%%%%%%%%%%%%%%%%%%%%%%%%%%
\subsection{Convergence results}
%%%%%%%%%%%%%%%%%%%%%%%%%%%%%%%%%%%%%%%

This section is devoted to convergence results of supersolutions. We provide some sufficient conditions for the pointwise limit of a sequence of supersolutions to be a supersolution again.

Before we state the convergence results, let us establish the following energy estimate.

\begin{lemma}\label{lem-energy}
Let $M>0$ and $h \in L^g_s(\mathbb{R}^n)$ satisfy $h \leq M$ a.e.\ in $\Omega$. Let $u$ be a supersolution of $\mathcal{L}u=0$ in $\Omega$ such that $u \geq h$ a.e.\ in $\mathbb{R}^n$ and $u \leq M$ a.e.\ in $\Omega$. For each open set $D \Subset \Omega$, there exists a constant $C=C(n, p, q, s, \Lambda, M, h, \Omega, D)$ such that
\begin{equation*}
\int_D\int_D G(|D^su|) \frac{\mathrm{d}y\,\mathrm{d}x}{|x-y|^n} \leq C.
\end{equation*}
\end{lemma}

To prove \Cref{lem-energy}, it is enough to prove the following two lemmas. We refer the reader to Lemma~5 in Korvenp\"a\"a--Kuusi--Palatucci~\cite{KKP17} for the proof of \Cref{lem-energy}.

\begin{lemma}
Let $u$ be a supersolution of $\mathcal{L}u=0$ in $\Omega$ and let $h \in L^g_s(\mathbb{R}^n)$. Assume that $u \geq h$ a.e.\ in $\mathbb{R}^n$ and $u \leq 0$ a.e.\ in $\Omega$. For each open set $D \Subset \Omega$, there exists a constant $C=C(n, p, q, s, \Lambda, h, \Omega, D)>0$ such that 
	\begin{equation}\label{eq-apriori1}
		\essinf_D u \geq -C.
	\end{equation}
\end{lemma}

\begin{proof}
We claim that
\begin{equation}\label{eq-claim-apriori1}
\esssup_{B_R(x_0)} u_- \leq C_0
\end{equation}
whenever $B_{2R}(x_0) \subset \Omega$, where $C_0$ is a constant depending only on $n$, $p$, $q$, $s$, $\Lambda$, and $h$. Once we prove the claim \eqref{eq-claim-apriori1}, the desired estimate \eqref{eq-apriori1} follows by a covering argument. Indeed, if we cover $D$ by finitely many balls $B_{R_i}(x_i)$, $i=1, \dots, N$, with $B_{2R_i}(x_i) \subset \Omega$, then
\begin{equation*}
\essinf_D u \geq -\max_{1 \leq i \leq N} \esssup_{B_{R_i}(x_i)}u_- \geq -NC_0,
\end{equation*}
where we used the assumption that $u \leq 0$ a.e.\ in $\Omega$.

Let us prove the claim \eqref{eq-claim-apriori1}. Since $u_-$ is a nonnegative subsolution in $\Omega$, \Cref{thm-loc-bdd-int} with $\varepsilon=1$ shows that
\begin{equation*}
\esssup_{B_R(x_0)} u_- \leq \mathrm{Tail}_g(u_-; x_0, R) + C g^{-1}\left( \fint_{B_{2R}(x_0)} g(u_-) \,\mathrm{d}x \right).
\end{equation*}
Since $u \geq h$ a.e.\ in $\mathbb{R}^n$, the right-hand side of the display above is bounded by
\begin{equation*}
\mathrm{Tail}_g(h_-; x_0, R) + C g^{-1}\left( \fint_{B_{2R}(x_0)} g(h_-) \,\mathrm{d}x \right),
\end{equation*}
which is finite by the fact that $h_- \in L^g_s(\mathbb{R}^n)$.
\end{proof}

\begin{lemma}
Let $u$ be a supersolution of $\mathcal{L}u=0$ in $B_{2R}=B_{2R}(x_0)$ which is essentially bounded from above in $B_{3R/2}$. Then there exists a constant $C=C(n, p, q, s, \Lambda)>0$ such that
	\begin{equation*}
		\fint_{B_R}\int_{B_R}G(|D^su|) \frac{\mathrm{d}y\,\mathrm{d}x}{|x-y|^n} \leq C G\left(\frac{H}{R^s}\right),
	\end{equation*}
where
\begin{equation*}
H:=\sup_{B_{3R/2}}u_+ + R^sG^{-1}\left(\fint_{B_{3R/2}}G\left(\frac{u_-}{R^s}\right)\,\mathrm{d}x\right) + \mathrm{Tail}_g(u_-; x_0, 3R/2).
\end{equation*}
\end{lemma}

\begin{proof}
Let $\eta \in C_c^{\infty}(B_{5R/4})$ be a cut-off function such that $0 \leq \eta \leq 1$, $\eta =1$ in $B_R$, and $|D \eta| \leq C/R$. Let us define a function $w:=2H-u$, which is nonnegative in $B_{3R/2}$. By using $\varphi:=w\eta^q$ as a test function, we obtain
\begin{align*}
0
&\leq \int_{\mathbb{R}^n}\int_{\mathbb{R}^n} g(|D^su|) \frac{D^su}{|D^su|} D^s\varphi \frac{k(x, y)}{|x-y|^n} \,\mathrm{d}y\, \mathrm{d}x \\
&= - \int_{B_{3R/2}}\int_{B_{3R/2}} g(|D^sw|) \frac{D^sw}{|D^sw|} D^s\varphi \frac{k(x, y)}{|x-y|^n} \,\mathrm{d}y\, \mathrm{d}x \\
&\quad + 2\int_{B_{3R/2}}\int_{\mathbb{R}^n \setminus B_{3R/2}} g(|D^su|) \frac{D^su}{|D^su|} \frac{(2H-u(x))\eta^q(x)}{|x-y|^s} \frac{k(x, y)}{|x-y|^n} \,\mathrm{d}y\, \mathrm{d}x \\
&=: I_1+I_2.
\end{align*}
To estimate the term $I_1$, we apply \Cref{lem-ineq6} to have
\begin{align*}
	I_1
	&\leq - c \int_{B_{3R/2}} \int_{B_{3R/2}} G(|D^su|) (\eta(x) \land \eta(y))^{q} \frac{\mathrm{d}y \,\mathrm{d}x}{|x-y|^{n}} \\
	&\quad + C \int_{B_{3R/2}} G\left(\frac{w(x)}{R^s}\right) \int_{B_{3R/2}} ((R^s|D^s\eta|)^p + (R^s|D^s\eta|)^q) \frac{\mathrm{d}y \,\mathrm{d}x}{|x-y|^{n}} \\
	&\leq - c \int_{B_R} \int_{B_R} G(|D^su|) \frac{\mathrm{d}y \,\mathrm{d}x}{|x-y|^{n}} + C \int_{B_{3R/2}} G\left(\frac{w(x)}{R^s}\right) \,\mathrm{d}x.
\end{align*}
Since $w \leq 2H+u_-$, we arrive at
\begin{equation*}
I_1 \leq - c \int_{B_R} \int_{B_R} G(|D^su|) \frac{\mathrm{d}y \,\mathrm{d}x}{|x-y|^{n}} + C R^n G\left(\frac{H}{R^s}\right).
\end{equation*}

Let us next estimate the term $I_2$ as follows:
\begin{align*}
I_2
&\leq 2\Lambda \int_{B_{5R/4}}\int_{\mathbb{R}^n \setminus B_{3R/2}} g\left( \frac{(u(x)-u(y))_+}{|x-y|^s} \right) \frac{2H-u(x)}{|x-y|^s} \frac{\mathrm{d}y\, \mathrm{d}x}{|x-y|^n} \\
&\leq C \left( \int_{B_{5R/4}} (2H+u_-(x)) \,\mathrm{d}x \right) \int_{\mathbb{R}^n \setminus B_{3R/2}} g\left( \frac{H+u_-(y)}{|y-x_0|^s} \right) \frac{\mathrm{d}y}{|y-x_0|^{n+s}}.
\end{align*}
By Jensen's inequality, we have
\begin{equation*}
	\fint_{B_{5R/4}}u_-(x)\,\mathrm{d}x \leq CR^s G^{-1}\left(\fint_{B_{3R/2}}G\left(\frac{u_-(x)}{R^s}\right)\,\mathrm{d}x\right)\leq CH,
\end{equation*}
Moreover, we obtain
\begin{align*}
\int_{\mathbb{R}^n \setminus B_{3R/2}} g\left( \frac{H+u_-(y)}{|y-x_0|^s} \right) \frac{\mathrm{d}y}{|y-x_0|^{n+s}}
&\leq \frac{C}{R^s} g\left( \frac{H}{R^s} \right) + \frac{C}{R^s} g\left( \frac{\mathrm{Tail}_g(u_-; x_0, 3R/2)}{R^s} \right) \\
&\leq \frac{C}{R^s} g\left( \frac{H}{R^s} \right),
\end{align*}
and hence
\begin{equation*}
I_2 \leq CR^n G\left( \frac{H}{R^s} \right).
\end{equation*}
Combining the estimates for $I_1$ and $I_2$ finishes the proof.
\end{proof}

We now move our attention to the convergence results for supersolutions. Note that we need the assumption \eqref{eq-q-1} in the following theorem.

\begin{theorem}\label{thm-conv}
Assume \eqref{eq-q-1}. Let $h_1, h_2 \in L^{p-1}_{sp}(\mathbb{R}^n)$ be such that $h_1 \leq h_2$ a.e.\ in $\mathbb{R}^n$. Let $\{u_j\}$ be an increasing sequence of supersolutions of $\mathcal{L}u=0$ in $\Omega$ such that $h_1 \leq u_j \leq h_2$ a.e.\ in $\mathbb{R}^n$ and that $\{u_j\}$ is uniformly locally essentially bounded from above in $\Omega$. If $u_j$ converges to a function $u$ pointwise almost everywhere in $\Omega$, then $u$ is a supersolution of $\mathcal{L}u=0$ in $\Omega$.
\end{theorem}

\begin{proof}
We first observe that $u \in W^{s, G}_\mathrm{loc}(\Omega) \cap L^g_s(\mathbb{R}^n)$. Indeed, for each open sets $D_2 \Subset D_1 \Subset \Omega$ the function $u_j$ is uniformly essentially bounded from above in $D_1$. We may assume that $u_j \leq 0$ in $D_1$ for all $j$. Note that
\begin{equation*}
\int_{D_2} G(|u_j|) \,\mathrm{d}x \leq \int_{D_2} G(|u_1|) \,\mathrm{d}x \leq C \quad\text{uniformly in}~j.
\end{equation*}
Moreover, by \Cref{lem-energy}
\begin{equation}\label{eq-apriori}
\int_{D_2} \int_{D_2} G(|D^su_j|) \frac{\mathrm{d}y\,\mathrm{d}x}{|x-y|^n} \leq C \quad\text{uniformly in}~j.
\end{equation}
Thus, Fatou's lemma shows that $u \in W^{s, G}(D_2)$. Moreover, it follows from $u_j \to u$ a.e.\ in $\mathbb{R}^n$ that $h_1 \leq u \leq h_2$ a.e.\ in $\mathbb{R}^n$, meaning that $u \in L^g_s(\mathbb{R}^n)$.

Since $\mathcal{E}(u_j, \varphi) \geq 0$ for all nonnegative functions $\varphi \in C^\infty_c(\Omega)$, it is enough to prove that $\mathcal{E}(u_j, \varphi) \to \mathcal{E}(u, \varphi)$ as $j \to \infty$. We may assume that $\mathrm{supp}\,\varphi=D_3 \Subset D_2$. For $\theta \in (0,1)$ we define a set $A_\theta:=\{x \in D_2 : |u_j(x)-u(x)| <\theta\}$. We write
\begin{align*}
\mathcal{E}(u_j, \varphi) - \mathcal{E}(u, \varphi)
&= \int_{A_\theta} \int_{A_\theta} \left( g(|D^su_j|) \frac{D^su_j}{|D^su_j|} - g(|D^su|) \frac{D^su}{|D^su|} \right) D^s\varphi \frac{k(x, y)}{|x-y|^n} \,\mathrm{d}y \,\mathrm{d}x \\
&\quad + \iint_{D_2^2 \setminus A_\theta^2} \left( g(|D^su_j|) \frac{D^su_j}{|D^su_j|} - g(|D^su|) \frac{D^su}{|D^su|} \right) D^s\varphi \frac{k(x, y)}{|x-y|^n} \,\mathrm{d}y \,\mathrm{d}x \\
&\quad + 2\int_{D_2} \int_{\mathbb{R}^n \setminus D_2} \left( g(|D^su_j|) \frac{D^su_j}{|D^su_j|} - g(|D^su|) \frac{D^su}{|D^su|} \right) D^s\varphi \frac{k(x, y)}{|x-y|^n} \,\mathrm{d}y \,\mathrm{d}x \\
&= I_1 + I_2 + I_3.
\end{align*}
Let $\delta \in (0, \frac{1}{q-1}]$ and apply \Cref{lem-ineq5} with $a=D^su_j$ and $b=D^su$, then we have
\begin{equation*}
|I_1| \leq \frac{C}{\delta} g^\delta(\theta) \int_{A_\theta} \int_{A_\theta} \left( g^{1-\delta}(|D^su_j|) + g^{1-\delta}(|D^su|) \right) \frac{|D^s\varphi|}{|x-y|^{(q-1)\delta s}} \frac{\mathrm{d}y \,\mathrm{d}x}{|x-y|^n}.
\end{equation*}
We define a Young function $H(t) = G^\ast(t^{1/(1-\delta)})$. By applying Young's inequality \Cref{lem-H}, we obtain
\begin{equation*}
|I_1| \leq Cg^\delta(\theta) \int_{D_2} \int_{D_2} \left( G(|D^su_j|) + G(|D^su|) + H^\ast\left( \frac{|\varphi(x)-\varphi(y)|}{|x-y|^{s+(q-1)\delta s}} \right) \right) \frac{\mathrm{d}y \,\mathrm{d}x}{|x-y|^n}.
\end{equation*}
By recalling \eqref{eq-apriori} and taking $\delta$ sufficiently small so that $s+(q-1)\delta s < 1$, we obtain $|I_1| \leq C g^\delta(\theta)$.

For $I_2$, we have by means of Young's inequality
\begin{equation*}
|I_2| \leq \varepsilon \iint_{D_2^2 \setminus A_\theta^2} (G(|D^su_j|) + G(|D^su|)) \frac{\mathrm{d}y \,\mathrm{d}x}{|x-y|^n} + C(\varepsilon) \iint_{D_2^2 \setminus A_\theta^2} G(|D^s\varphi|) \frac{\mathrm{d}y \,\mathrm{d}x}{|x-y|^n}
\end{equation*}
for any $\varepsilon>0$. Since $\lim_{j \to \infty}|D_2^2 \setminus A_\theta^2|=0$ due to the pointwise convergence of $u_j$, we have $\limsup_{j \to \infty} |I_2| \leq C\varepsilon$. Therefore, by sending $\varepsilon \to 0$ and then $\theta \to 0$, we conclude
\begin{equation}\label{eq-conv-I12}
\limsup_{j \to \infty} (|I_1|+|I_2|) = 0.
\end{equation}

Let us next investigate $I_3$. We observe that the integrand in $I_3$ is bounded by 
\begin{equation*}
C \sum_{i=1}^2 \left( g\left( \frac{|h_i(x)|}{|x-y|^s} \right) + g\left( \frac{|h_i(y)|}{|x-y|^s} \right) \right) \frac{\varphi(x)}{|x-y|^{n+s}}.
\end{equation*}
Note also that
\begin{equation*}
\int_{D_2} \int_{\mathbb{R}^n \setminus D_2} g\left( \frac{|h_i(x)|}{|x-y|^s} \right) \frac{\varphi(x)}{|x-y|^{n+s}} \,\mathrm{d}y \,\mathrm{d}x \leq C\int_{D_3} g(|h_i|) \,\mathrm{d}x
\end{equation*}
for some $C=C(n, p, q, s, \mathrm{dist}(D_3, \partial D_2))>0$ and
\begin{align*}
\int_{D_2} \int_{\mathbb{R}^n \setminus D_2} g\left( \frac{|h_i(y)|}{|x-y|^s} \right) \frac{\varphi(x)}{|x-y|^{n+s}} \,\mathrm{d}y \,\mathrm{d}x
&\leq C \int_{\mathbb{R}^n} g\left( \frac{|h_i(y)|}{(1+|y|)^s} \right) \frac{\mathrm{d}y}{(1+|y|)^{n+s}}
\end{align*}
for some $C=C(n, p, q, s, \mathrm{dist}(D_3, \partial D_2), \mathrm{diam}(D_3))>0$. Since $h_i \in L^g_s(\mathbb{R}^n)$, we can apply the dominated convergence theorem to conclude that
\begin{equation}\label{eq-conv-I3}
\lim_{j \to \infty} |I_3| = 0.
\end{equation}
It follows from \eqref{eq-conv-I12} and \eqref{eq-conv-I3} that $\mathcal{E}(u_j, \varphi) \to \mathcal{E}(u, \varphi)$ as $j \to \infty$.
\end{proof}

The boundedness assumption in \Cref{thm-conv} can be removed if we assume that the sequence is increasing.

\begin{corollary}\label{cor-conv}
Assume \eqref{eq-q-1}. Let $\{u_j\}$ be an increasing sequence of supersolutions of $\mathcal{L}u=0$ in $\Omega$ such that $u_j$ converges a.e.\ in $\mathbb{R}^n$ to a function $u \in W^{s, G}_{\mathrm{loc}} \cap L_s^g(\mathbb{R}^n)$ as $j \to \infty$. Then $u$ is a supersolution in $\Omega$ as well.
\end{corollary}

\begin{proof}
For each $M>0$, we define $u_k:=u\land k$ and $u_{j, k}:=u_j \land k$. Then $u_{j, k}$ is a supersolution of $\mathcal{L}u_{j, k}=0$ in $\Omega$. To prove this, one can follow the lines of proof of Lemma~7 in Korvenp\"a\"a--Kuusi--Palatucci~\cite{KKP17}. We omit the proof because it is essentially the same. Moreover, since $u_{1, k} \leq u_{j, k} \leq k$ for any $j \in \mathbb{N}$ and $u_{j, k} \to u_k$ a.e.\ in $\mathbb{R}^n$ as $j \to \infty$, \Cref{thm-conv} shows that $u_k$ is a supersolution of $\mathcal{L}u_k=0$ in $\Omega$, namely $\mathcal{E}(u_k, \varphi) \geq 0$ for any nonnegative function $\varphi \in C^\infty_c(\Omega)$. Since
\begin{equation*}
\left| g(|D^su_k|) \frac{D^su_k}{|D^su_k|} D^s\varphi \right| \leq g(|D^su|) |D^s\varphi| \in L^1(\mathbb{R}^{2n}; |x-y|^{-n})
\end{equation*}
for all $k>0$, we employ the dominated convergence theorem to arrive at the conclusion.
\end{proof}

%%%%%%%%%%%%%%%%%%%%%%%%%%%%%%%%%%%%%%%
\section{Properties of superharmonic functions}\label{sec-superharmonic}
%%%%%%%%%%%%%%%%%%%%%%%%%%%%%%%%%%%%%%%

As explained in the introduction, superharmonic functions play a fundamental role in nonlocal nonlinear potential theory. We investigate the relation between supersolutions and superharmoinc functions, and study some properties of superharmonic functions.

%%%%%%%%%%%%%%%%%%%%%%%%%%%%%%%%%%%%%%%
\subsection{Relation between supersolutions and superharmonic functions}
%%%%%%%%%%%%%%%%%%%%%%%%%%%%%%%%%%%%%%%

Supersolutions and superharmonic functions are closely related to each other. We prove that under certain circumstances supersolutions are superharmonic functions and vice versa.

\begin{theorem}\label{thm-relation}
\begin{enumerate}[(i)]
\item
Let $u$ be a supersolution of $\mathcal{L}u=0$ in $\Omega$ that is lower semicontinuous in $\Omega$ and satisfies
\begin{equation}\label{eq-v-lsc}
u(x)=\essliminf_{y \to x}u(y) \quad \text{for every}~ x \in \Omega.
\end{equation}
Then $u$ is superharmonic in $\Omega$.
\item
Let $u$ be a superharmonic function in $\Omega$. If $u$ is locally bounded from above in $\Omega$ or $W^{s, G}_{\mathrm{loc}}(\Omega)$, then $u$ is a supersolution of $\mathcal{L}u=0$ in $\Omega$.
\end{enumerate}
\end{theorem}

\begin{proof}[Proof of \Cref{thm-relation} (i)]
It is enough to check that \Cref{def-superharmonic}~(iii) holds. Let $D\Subset \Omega$ be an open set and $v \in C(\overline{D})$ a solution of $\mathcal{L}v=0$ in $D$ with $v_+ \in L^\infty(\mathbb{R}^n)$ such that $u \geq v$ on $\partial D$ and a.e.\ on $\mathbb{R}^n \setminus D$. For $\varepsilon>0$, we consider a compact set $K_\varepsilon:=\{v-\varepsilon \geq u \} \cap \overline{D}$, which satisfies $K_\varepsilon \cap \partial D = \emptyset$. Let $D_1$ be an open subset of $\mathbb{R}^n$ with Lipschitz boundary such that $K_\varepsilon \subset D_1 \Subset D$. Since $u \geq v-\varepsilon$ a.e.\ in $\mathbb{R}^n \setminus D_1$ and $\partial D_1$ is Lipschitz, we have $(u-v+\varepsilon)_- \in W^{s, G}_0(D_1, \Omega)$, see Remark~2.2~(ii) in Kim--Lee--Lee~\cite{KLL23}. Thus, \Cref{lem-comparison} yields $u \geq v-\varepsilon$ a.e.\ in $\mathbb{R}^n$. The desired result $u \geq v$ in $D$ follows from the assumption \eqref{eq-v-lsc} and continuity of $v$. Indeed, for $x \in D$ there exists $r>0$ such that $B_r(x) \subset D$ and
\begin{equation*}
u(x) \geq \essinf_{B_r(x)}u-\varepsilon \geq \inf_{B_r(x)} v - 2\varepsilon \geq v(x) - 3\varepsilon.
\end{equation*}
Since $\varepsilon>0$ and $x \in D$ were arbitrary, we conclude $u\geq v$ in $D$.
\end{proof}

The following lemma shows that one can approximate a superharmonic function by continuous supersolutions.

\begin{lemma}\label{lem-approx}
Let $u$ be a superharmonic function in $\Omega$ and let $D \Subset \Omega$ be an open set. There exists a sequence $\{u_j\}$ of increasing supersolutions of $\mathcal{L}u_j=0$ in $D$ such that $u_j \in C(\overline{D})$ and $u_j \to u$ pointwise in $\mathbb{R}^n$.
\end{lemma}

\begin{proof}
Let us choose two open sets $D_0$ and $U$ such that $D \Subset D_0 \Subset U \Subset \Omega$ and that $\mathbb{R}^n \setminus D_0$ satisfies the measure density condition. Note that Lemma~8 in Korvenp\"a\"a--Kuusi--Palatucci~\cite{KKP17} shows that there exists an increasing sequence of smooth functions $\{\psi_j\} \subset C^{\infty}(\overline{U})$ which converges to $u$ pointwise in $U$. We point out that this fact is proved for an $(s, p)$-superharmonic function $u$, but its proof only uses the lower semicontinuity and the local boundedness from below, see \Cref{rmk-superharmonic} (ii). We define
	\begin{equation*}
		\vartheta_j(x) :=
		\begin{cases}
			\psi_j(x) &\text{for}~x \in U, \\
			\min\{j, u(x)\} &\text{for}~x \in \mathbb{R}^n \setminus U.
		\end{cases}
	\end{equation*}
It follows from the smoothness of $\psi_j$ and the fact that $u_- \in L_s^g(\mathbb{R}^n)$ that $\vartheta_j \in W^{s, G}(D_0, U) \cap L^g_s(\mathbb{R}^n)$. Note that $\vartheta_j \in \mathcal{K}_{\vartheta_j, \vartheta_j}(D_0, U)$. Thus, by \Cref{thm-obstacle}, there exists a unique solution $u_j \in \mathcal{K}_{\vartheta_j, \vartheta_j}(D_0, U)$ to the obstacle problem and it is a supersolution of $\mathcal{L}u_j=0$ in $D_0$ .An application of \Cref{thm-Holder-bdry} yields that $u_j$ belongs to $C(\overline{D_0})$. Moreover, by \Cref{thm-relation} (i), $u_j$ is superharmonic in $\Omega$.

We claim that $\{u_j\}$ is an increasing sequence. Indeed, for $A_{j, \varepsilon}:=D_0 \cap \{u_j > \vartheta_j+\varepsilon\}$ with $\varepsilon>0$, we see that $u_j$ is a solution of $\mathcal{L}u_j=0$ in $A_{j, \varepsilon}$ by \Cref{cor-obstacle}. Moreover, we have $A_{j, \varepsilon} \Subset D_0$, $u_{j+1}+\varepsilon \geq u_j$ on $\overline{D}_0 \setminus A_{j, \varepsilon}$, and $u_{j+1}+\varepsilon \geq u_j$ a.e.\ in $\mathbb{R}^n \setminus D_0$. Thus, by \Cref{def-superharmonic} (iii) for $u_{j+1}$, we deduce that $u_{j+1}+\varepsilon \geq u_j$ in $D_0$. The claim follows by letting $\varepsilon \to 0$.

Similarly, we have $u_j \leq u$. Since $\vartheta_j$ converges pointwise to $u$, we conclude that $u_j$ also converges pointwise to $u$.
\end{proof}

The following lemma shows the second part of \Cref{thm-relation} under an additional assumption $u \in L^g_s(\mathbb{R}^n)$. The full proof of \Cref{thm-relation} (ii) without this additional assumption follows by combining \Cref{lem-relation} and \Cref{thm-integrability}. 

\begin{lemma}\label{lem-relation}
Assume \eqref{eq-q-1}. Let $u \in L^g_s(\mathbb{R}^n)$ be a superharmonic function in $\Omega$. If $u$ is locally bounded from above in $\Omega$ or $W^{s, G}_{\mathrm{loc}}(\Omega)$, then $u$ is a supersolution of $\mathcal{L}u=0$ in $\Omega$.
\end{lemma}

\begin{proof}
Let $u \in L^g_s(\mathbb{R}^n)$ be a superharmonic function in $\Omega$ which is locally bounded from above. Let $D \Subset \Omega$ be an open set and $\{u_j\}$ be the sequence constructed in \Cref{lem-approx}. Since $u_j$ satisfies $u_1 \leq u_j \leq u$ with $u_1, u \in L_s^g(\mathbb{R}^n)$ and $u$ is bounded from above in $D$, \Cref{thm-conv} shows that $u$ is a supersolution of $\mathcal{L}u=0$ in $D$. Since $D \Subset \Omega$ was arbitrarily chosen, we conclude that $u$ is a supersolution of $\mathcal{L}u=0$ in $\Omega$.

Let us next assume that $u \in W^{s, G}_{\mathrm{loc}}(\Omega) \cap L^g_s(\mathbb{R}^n)$ is superharmonic in $\Omega$. For $k>0$, we define $u_k:=u \land k$. By the previous result, $u_k$ is a supersolution of $\mathcal{L}u=0$ in $\Omega$. Since $\{u_k\}$ is an increasing sequence and $u_k$ converges to $u$ pointwise in $\mathbb{R}^n$, an application of \Cref{cor-conv} finishes the proof.
\end{proof}

%%%%%%%%%%%%%%%%%%%%%%%%%%%%%%%%%%%%%%%
\subsection{Pointwise behaviors}
%%%%%%%%%%%%%%%%%%%%%%%%%%%%%%%%%%%%%%%

$\mathcal{L}$-superharmonic functions share the same pointwise behaviors with $(s, p)$-superharmonic functions. We collect some properties of $\mathcal{L}$-superharmonic functions in this section.

\begin{lemma}\label{lem-zero}
Let $u$ be superharmonic in $\Omega$. If $u=0$ a.e.\ in $\Omega$, then $u=0$ in $\Omega$.
\end{lemma}

\begin{theorem}\label{thm-pointwise}
Let $u$ be superharmonic in $\Omega$, then
\begin{equation*}
u(x)=\liminf_{y \to x}u(y)=\essliminf_{y \to x}u(y) \quad \text{for every}~ x\in \Omega.
\end{equation*}
In particular, $\inf_Du=\essinf_Du$ for any open set $D \Subset \Omega$.
\end{theorem}

We refer the reader to Lemma~10~and~Theorem~13 in Korvenp\"a\"a--Kuusi--Palatucci~\cite{KKP17} for the proofs of \Cref{lem-zero} and \Cref{thm-pointwise}, respectively.

%%%%%%%%%%%%%%%%%%%%%%%%%%%%%%%%%%%%%%%
\subsection{Integrability}
%%%%%%%%%%%%%%%%%%%%%%%%%%%%%%%%%%%%%%%

Superharmonic functions do not belong to $W^{s, G}_{\mathrm{loc}}(\Omega)$ in general even in the standard growth case $G(t)=t^p$ as the fundamental solution shows. Note that the assumptions in \Cref{thm-relation} (ii) exclude the fundamental solution. Nevertheless, superharmonic functions have lower integrability.

\begin{theorem}\label{thm-integrability}
Let
\begin{equation}\label{eq-order}
\delta \in 
\begin{cases}
(0, \frac{n}{n-sp}) &\text{if}~ sp<n, \\
(0, \infty) &\text{if}~ sp \geq n,
\end{cases}
\quad \sigma \in (0, s), \quad\text{and}\quad \alpha \in (0, \min\{\tfrac{n}{n-sp/q}, \tfrac{q}{q-1} \}).
\end{equation}
Then any superharmonic function belongs to $L^{g^\delta}_{\mathrm{loc}}(\Omega) \cap W^{\sigma, g^{\alpha}}_{\mathrm{loc}}(\Omega) \cap L^g_s(\mathbb{R}^n)$.
\end{theorem}

Note that the orders of differentiability and integrability given in \eqref{eq-order} are optimal in the sense that \Cref{thm-integrability} recovers \cite[Theorem~1~(ii)]{KKP17} in the case $G(t)=t^p$.

\begin{lemma}\label{lem-integrability}
Let $\sigma \in (0,s)$ and $0<\alpha<\min\{ \frac{n}{n-sp/q}, \frac{q}{q-1} \}$. Let $u$ be a supersolution of $\mathcal{L}u=0$ in $B_{4R}=B_{4R}(x_0)$ such that $u \geq 0$ in $B_{4R}$. There exists a constant $C=C(n, p, q, s, \sigma, \alpha, \Lambda)>0$ such that
\begin{equation*}
\fint_{B_R}\int_{B_R} g^\alpha(R^{\sigma-s}|D^\sigma u|) \frac{\mathrm{d}y\,\mathrm{d}x}{|x-y|^n} \leq Cg^\alpha\left(\essinf_{B_R}\frac{u}{R^s}\right) + Cg^\alpha\left( \frac{\mathrm{Tail}_g(u_-; x_0, 4R)}{R^s} \right).
\end{equation*}
\end{lemma}

\begin{proof}
For $\varepsilon>0$, we let
\begin{equation*}
d=\essinf_{B_R}u+\mathrm{Tail}_g(u_-; x_0, 4R)+\varepsilon>0.
\end{equation*}
By applying \Cref{lem-Caccio2} with $v=(u+d)/(4R)^s$ and $\gamma=1+\beta \in (0,1)$, and then using \Cref{thm-WHI-int}, we have
\begin{equation}\label{eq-int-caccio}
\fint_{B_R} \int_{B_R} G(|D^su|) A \frac{\mathrm{d}y \,\mathrm{d}x}{|x-y|^{n}}\leq C \fint_{B_{2R}} \bar{g}^{\gamma}(v) \,\mathrm{d}x \leq C \bar{g}^{\gamma}(d/R^s),
\end{equation}
where $A= \fint_{v(y)}^{v(x)} \frac{\bar{g}^{\beta}(t)}{t} \,\mathrm{d}t$.

We claim that
\begin{equation}\label{eq-int-claim}
\frac{g^\alpha(R^{\sigma-s}|D^\sigma u|)}{\bar{g}^\alpha(d/R^s)}
\lesssim \frac{G(|D^su|)A}{\bar{g}^\gamma(d/R^s)} + \left( \frac{\bar{g}(v(x) \lor v(y))}{\bar{g}(d/R^s)} \right)^{\frac{\alpha(1-(p-1)\beta)}{p-\alpha(p-1)}} \left( \frac{|x-y|}{R} \right)^{\frac{\alpha(s-\sigma)p(p-1)}{p-\alpha(p-1)}}.
\end{equation}
Once we prove the claim, then it follows from \eqref{eq-int-caccio} the assumption $0<\alpha<\frac{q}{q-1}\leq \frac{p}{p-1}$ that
\begin{equation*}
\fint_{B_R}\int_{B_R} \frac{g^\alpha(R^{\sigma-s}|D^\sigma u|)}{\bar{g}^\alpha(d/R^s)} \frac{\mathrm{d}y\,\mathrm{d}x}{|x-y|^n} \leq C + C \fint_{B_R} \left( \frac{\bar{g}(v)}{\bar{g}(d/R^s)} \right)^{\frac{\alpha(1-(q-1)\beta)}{q-\alpha(q-1)}} \,\mathrm{d}x.
\end{equation*}
Moreover, since $\alpha < \frac{n}{n-sp/q} \iff \frac{\alpha}{q-\alpha(q-1)} < \frac{n}{n-sp}$, we can choose $\beta \in (-1, 0)$ sufficiently close to 0 so that
\begin{equation*}
\frac{\alpha(1-(q-1)\beta)}{q-\alpha(q-1)} < \frac{n}{n-sp},
\end{equation*}
which allows us to apply \Cref{thm-WHI-int}. This finishes the proof.

Let us now prove the claim. We may assume without loss of generality that $v(x) > v(y)$. We define $H(t) = G^\ast(t^{1/\alpha})$, then the same computation as in the proof of \Cref{lem-H} shows that
\begin{equation}\label{eq-H-ast}
H^\ast\left( \frac{G(t)}{g^\alpha(t)} \right) \sim G(t).
\end{equation}
Moreover, we have $(H^\ast)^{-1}(t) \sim t^{1-\alpha}(G^{-1}(t))^\alpha$, and hence it follows from
\begin{align*}
&\lambda^{\frac{q-\alpha(q-1)}{q}} t^{1-\alpha} (G^{-1}(t))^\alpha \lesssim (\lambda t)^{1-\alpha} (G^{-1}(\lambda t))^\alpha \lesssim \lambda^{\frac{p-\alpha(p-1)}{p}} t^{1-\alpha} (G^{-1}(t))^\alpha, \quad\lambda \geq1, \\
&\lambda^{\frac{p-\alpha(p-1)}{p}} t^{1-\alpha} (G^{-1}(t))^\alpha \lesssim (\lambda t)^{1-\alpha} (G^{-1}(\lambda t))^\alpha \lesssim \lambda^{\frac{q-\alpha(q-1)}{q}} t^{1-\alpha} (G^{-1}(t))^\alpha, \quad\lambda \leq 1,
\end{align*}
that
\begin{align}\label{eq-H-ast-pq}
\begin{split}
&\lambda^{\frac{p}{p-\alpha(p-1)}} H^\ast(t) \lesssim H^\ast(\lambda t) \lesssim \lambda^{\frac{q}{q-\alpha(q-1)}} H^\ast(t), \quad \lambda \geq 1, \\
&\lambda^{\frac{q}{q-\alpha(q-1)}} H^\ast(t) \lesssim H^\ast(\lambda t) \lesssim \lambda^{\frac{p}{p-\alpha(p-1)}} H^\ast(t), \quad \lambda \leq 1.
\end{split}
\end{align}
By using \Cref{lem-G}, Young's inequality, and \eqref{eq-H-ast-pq}, we have
\begin{align*}
\frac{g^\alpha(R^{\sigma-s}|D^\sigma u|)}{\bar{g}^\alpha(d/R^s)}
&\leq \frac{C}{\bar{g}^\alpha(d/R^s)} g^\alpha(|D^su|) \left( \frac{|x-y|}{R}\right)^{\alpha(s-\sigma)(p-1)} \\
&\leq \frac{C}{\bar{g}^\gamma(d/R^s)} A \left( H(g^\alpha(|D^su|)) + H^\ast \left( \frac{(|x-y|/R)^{\alpha(s-\sigma)(p-1)}}{\bar{g}^{\alpha-\gamma}(d/R^s)A} \right) \right) \\
&\leq C \frac{G(|D^su|)A}{\bar{g}^\gamma(d/R^s)} + C \frac{A}{\bar{g}^\gamma(d/R^s)} H^\ast \left( \frac{1}{\bar{g}^{\alpha-\gamma}(d/R^s)A} \right) \left( \frac{|x-y|}{R} \right)^{\frac{\alpha(s-\sigma)p(p-1)}{p-\alpha(p-1)}}.
\end{align*}
We may assume that $\alpha \geq 1$ because the case $\alpha<1$ can be deduced by using H\"older's inequality. Since $A^{-1} \leq \bar{g}^{1-\gamma}(v(x)) v(x)$ and $\alpha \geq 1 > \gamma$, we obtain
\begin{align*}
H^\ast \left( \frac{1}{\bar{g}^{\alpha-\gamma}(d/R^s)A} \right)
&\leq C \left( \frac{1}{A\bar{g}^{1-\gamma}(v(x)) v(x)} \right)^{\frac{p}{p-\alpha(p-1)}} H^\ast \left( \frac{\bar{g}^{\alpha-\gamma}(v(x))}{\bar{g}^{\alpha-\gamma}(d/R^s)} \frac{G(v(x))}{\bar{g}^\alpha(v(x))} \right) \\
&\leq C \left( \frac{1}{A\bar{g}^{1-\gamma}(v(x)) v(x)} \right)^{\frac{p}{p-\alpha(p-1)}} \left( \frac{\bar{g}(v(x))}{\bar{g}(d/R^s)} \right)^{\frac{(\alpha-\gamma)q}{q-\alpha(q-1)}} G(v(x))
\end{align*}
by using \eqref{eq-H-ast-pq} and \eqref{eq-H-ast}. Thus, by using $A^{-1} \leq \bar{g}^{1-\gamma}(v(x))v(x)$ again, we deduce
\begin{equation*}
\frac{A}{\bar{g}^\gamma(d/R^s)} H^\ast \left( \frac{1}{\bar{g}^{\alpha-\gamma}(d/R^s)A} \right) \leq C \frac{\bar{g}^\gamma(v(x))}{\bar{g}^\gamma(d/R^s)} \left( \frac{\bar{g}(v(x))}{\bar{g}(d/R^s)} \right)^{\frac{(\alpha-\gamma)q}{q-\alpha(q-1)}},
\end{equation*}
which leads us to the claim \eqref{eq-int-claim}.
\end{proof}

\begin{proof}[Proof of \Cref{thm-integrability}]
Let $B_{4R}=B_{4R}(x_0) \subset \Omega$. Since $u$ is locally bounded from below, we may assume that $u$ is nonnegative in $B_{4R}$. If we let $u_k:=u \land k$ for $k>0$, then $u_k$ is bounded from above and $u_k \in L^g_s(\mathbb{R}^n)$, and hence by \Cref{lem-relation} $u_k$ is a supersolution of $\mathcal{L}u=0$ in $\Omega$. An application of \Cref{thm-WHI-int} to $u_k$ yields
\begin{equation*}
\fint_{B_R} g^\delta \left( \frac{u_k}{R^s} \right) \,\mathrm{d}x \leq C g^\delta\left( \essinf_{B_R} \frac{u_k}{R^s} \right) + C g^\delta \left( \frac{\mathrm{Tail}_g((u_k)_-; x_0, 4R)}{R^s} \right).
\end{equation*}
By letting $k \to \infty$ and employing Fatou's lemma, we arrive at
	\begin{equation*}
		\fint_{B_R} g^\delta \left( \frac{u}{R^{s}} \right) \,\mathrm{d}x  \leq Cg^\delta\left(\essinf_{B_R} \frac{u}{R^s}\right) + Cg^\delta \left( \frac{\mathrm{Tail}_g(u_-; x_0, 4R)}{R^s} \right) < \infty,
	\end{equation*}
which shows that $u \in L^{g^\delta}_{\mathrm{loc}}(\Omega)$.

In a similar way, by applying \Cref{lem-integrability} for $u_k$, letting $k \to \infty$, and using Fatou's lemma, we obtain
\begin{equation*}
\fint_{B_R}\int_{B_R} g^\alpha(R^{\sigma-s}|D^\sigma u|) \frac{\mathrm{d}y\,\mathrm{d}x}{|x-y|^n} \leq Cg^\alpha\left(\essinf_{B_R}\frac{u}{R^s}\right) + Cg^\alpha\left( \frac{\mathrm{Tail}_g(u_-; x_0, 4R)}{R^s} \right).
\end{equation*}
Thus, $u \in W^{\sigma, g^{\alpha}}_{\mathrm{loc}}(\Omega)$. 

Finally, by applying \Cref{lem-tail} to $u_k$ with $\sigma \in (0,s)$ satisfying $\sigma \geq (sq-1)/(q-1)$, taking $k \to \infty$, and using Fatou's lemma, we also obtain $u \in L^g_s(\mathbb{R}^n)$.
\end{proof}

\begin{appendix}

%%%%%%%%%%%%%%%%%%%%%%%%%%%%%%%%%%%%%%%
\section{Algebraic inequalities}\label{sec-appendix}
%%%%%%%%%%%%%%%%%%%%%%%%%%%%%%%%%%%%%%%

In this section we provide several algebraic inequalities.

\begin{lemma}\label{lem-ineq1}
For any $a, b \in \mathbb{R}$, it holds that
\begin{equation}\label{eq-alg-tail}
g(|a-b|)\frac{a-b}{|a-b|} \leq 2g(a_+ + b_-) - \frac{p}{q}2^{1-q} g(b_+).
\end{equation}
\end{lemma}

\begin{proof}
If $a<b$, then we need to prove
\begin{equation*}
g(b_+) \leq \frac{q}{p}2^{q-1} (g(b-a)+2g(a_++b_-)),
\end{equation*}
which is trivial when $b \leq 0$. The case $b>0$ simply follows from \Cref{lem-G}. If $a \geq b$, then we may assume that $b>0$ since otherwise \eqref{eq-alg-tail} is obvious. We then have $g(a-b) + \frac{p}{q}2^{1-q}g(b) \leq 2g(a_+)$ since $\frac{p}{q}2^{1-q} \leq 1$.
\end{proof}

The following lemma is an Orlicz version of Lemma~A.1 in Kim--Lee--Lee~\cite{KLL23}.

\begin{lemma}\label{lem-ineq2}
Let $\gamma, \beta \in \mathbb{R} \setminus \{0\}$ be such that $\gamma=\beta+1$ and let $d>0$. Let $\bar{u}$ and $\varphi$ be defined as in \eqref{eq-bar-u} and \eqref{eq-varphi}, respectively, and $v=(\bar{u}+d)/R^s$. There exist constants $c, C>0$, depending only on $p$ and $q$, such that
\begin{align}\label{eq-ineq1}
\begin{split}
\frac{1}{\beta} g(|D^s\bar{u}|) \frac{D^s\bar{u}}{|D^s\bar{u}|} D^s\varphi
&\geq \frac{c}{R^s} G(|D^s\bar{u}|) \left( \fint_{v(y)}^{v(x)} \frac{\bar{g}^{\beta}(t)}{t} \,\mathrm{d}t \right) (\eta(x) \land \eta(y))^q \\
&\quad - \frac{C}{R^s} \left( \frac{1}{|\beta|} \lor \frac{1}{|\beta|^q} \right) (\bar{g}^{\gamma}(v(x)) \lor \bar{g}^{\gamma}(v(y))) ((R^s |D^s\eta|)^p \lor (R^s |D^s\eta|)^q)
\end{split}
\end{align}
for all $x, y \in B_R=B_R(x_0)$, where $\bar{g}$ is defined as in \eqref{eq-bar-g}.
\end{lemma}

\begin{proof}
We may assume $\bar{u}(x) > \bar{u}(y)$ without loss of generality. Let $L$ denote the left-hand side of \eqref{eq-ineq1}. Let us first consider the case $\eta(x)/\eta(y) \in [1/2,2]$. Since
\begin{align*}
\frac{1}{\beta} D^s\varphi
&= D^s \left( \frac{\bar{g}^\beta(v)-\bar{g}^\beta(l)}{\beta} \right) \eta^q(x) + \frac{1}{\beta} (\bar{g}^\beta(v(y))-\bar{g}^\beta(l))D^s\eta^q \\
&\geq D^s \left( \frac{\bar{g}^\beta(v)}{\beta} \right) (\eta(x) \land \eta(y))^q - \frac{q}{|\beta|} \frac{\bar{g}^\beta(v(y))}{v(y)} v(x) (\eta(x) \lor \eta(y))^{q-1} |D^s\eta|
\end{align*}
and similarly
\begin{equation*}
\frac{1}{\beta} D^s\varphi \geq D^s \left( \frac{\bar{g}^\beta(v)}{\beta} \right) (\eta(x) \land \eta(y))^q - \frac{q}{|\beta|} \frac{\bar{g}^\beta(v(x))}{v(x)} v(x) (\eta(x) \lor \eta(y))^{q-1} |D^s\eta|,
\end{equation*}
we have
\begin{equation*}
\frac{1}{\beta} D^s\varphi \geq D^s \left( \frac{\bar{g}^\beta(v)}{\beta} \right) (\eta(x) \land \eta(y))^q - \frac{q}{|\beta|} \left( \frac{\bar{g}^\beta(v(x))}{v(x)} \land \frac{\bar{g}^\beta(v(y))}{v(y)} \right) v(x) (\eta(x)\lor \eta(y))^{q-1}|D^s\eta^q|.
\end{equation*}
By using \eqref{eq-bar-g-pq} and $\eta(x)\lor\eta(y) \leq 2(\eta(x)\land\eta(y))$, we obtain
\begin{align*}
L
&\geq \frac{p-1}{R^s} g(D^s\bar{u}) D^s\bar{u} \left( \fint_{v(y)}^{v(x)} \frac{\bar{g}^\beta(t)}{t} \,\mathrm{d}t \right) (\eta(x) \land \eta(y))^q \\
&\quad - \frac{q2^{q-1}}{|\beta|} g(D^s\bar{u}) \frac{|D^s\eta|v(x)}{\eta(x)\land \eta(y)} \left( \frac{\bar{g}^\beta(v(x))}{v(x)} \land \frac{\bar{g}^\beta(v(y))}{v(y)} \right) (\eta(x) \land \eta(y))^q =: I_1 + I_2.
\end{align*}
Let $\varepsilon > 0$, then the inequality \eqref{eq-alg} shows that
\begin{equation*}
g(D^s\bar{u}) \frac{|D^s \eta|v(x)}{\eta(x) \land \eta(y)} \leq \varepsilon g(D^s\bar{u}) D^s\bar{u} + g\left( \frac{1}{\varepsilon} \frac{|D^s\eta| v(x)}{\eta(x) \land \eta(y)} \right) \frac{|D^s\eta|v(x)}{\eta(x) \land \eta(y)}.
\end{equation*}
By taking $\varepsilon=(p-1)|\beta|/(q2^qR^s)$, we deduce
\begin{equation*}
I_2 \geq - \frac{1}{2}I_1 - \frac{q2^{q-1}}{|\beta|} g\left( \frac{1}{\varepsilon} \frac{|D^s\eta| v(x)}{\eta(x)\land\eta(y)} \right) |D^s\eta| \bar{g}^\beta(v(x)) (\eta(x)\land\eta(y))^{q-1}.
\end{equation*}
We have from \Cref{lem-G} and \eqref{eq-bar-g-comp} that
\begin{align*}
&g\left( \frac{1}{\varepsilon} \frac{|D^s\eta|v(x)}{\eta(x) \land \eta(y)} \right) |D^s\eta| (\eta(x)\land\eta(y))^{q-1} \\
&\leq C \left( \frac{1}{|\beta|^p} \lor \frac{1}{|\beta|^q} \right) |D^s\eta| g(R^s |D^s\eta| v(x)) \\
&\leq \frac{C}{R^s} \left( \frac{1}{|\beta|} \lor \frac{1}{|\beta|^q} \right) ((R^s|D^s\eta|)^p \lor (R^s|D^s\eta|)^q) \bar{g}(v(x)).
\end{align*}
Thus, we arrive at
\begin{equation*}
L \geq \frac{1}{2} I_1 - \frac{C}{R^s} \left( \frac{1}{|\beta|} \lor \frac{1}{|\beta|^q} \right) ((R^s|D^s\eta|)^p \lor (R^s|D^s\eta|)^q) \bar{g}^\gamma(v(x))
\end{equation*}
with some constant $C = C(p, q) > 0$, which finishes the proof for the case $\eta(x)/\eta(y) \in [1/2,2]$.

Let us next consider the case $\eta(x)/\eta(y) \notin [1/2, 2]$. In this case, we have
\begin{equation}\label{eq-eta}
\eta(x) \lor \eta(y) \leq 2|\eta(x)-\eta(y)|.
\end{equation}
We write
\begin{equation*}
L = \frac{1}{\beta} g(D^s\bar{u}) \frac{\bar{g}^\beta(v(x))\eta^q(x)}{|x-y|^s} - \frac{1}{\beta} g(D^s\bar{u}) \frac{\bar{g}^\beta(v(y)) \eta^q(y)}{|x-y|^s} = J_1+J_2
\end{equation*}
when $\beta>0$ and
\begin{equation*}
L = \frac{1}{|\beta|} g(D^s\bar{u}) \frac{\bar{g}^\beta(v(y)) \eta^q(y)}{|x-y|^s} - \frac{1}{|\beta|} g(D^s\bar{u}) \frac{\bar{g}^\beta(v(x))\eta^q(x)}{|x-y|^s} = J_1+J_2
\end{equation*}
when $\beta<0$. We show that $J_1$ and $J_2$ are estimated from below by the two terms on the right-hand side of \eqref{eq-ineq1}, respectively. For $J_1$, we use \eqref{eq-bar-g-pq} to have
\begin{equation*}
\int_{v(y)}^{v(x)} \frac{\bar{g}^{\beta}(t)}{t} \,\mathrm{d}t \leq \frac{1}{(p-1)\beta} \int_{v(y)}^{v(x)} (\bar{g}^{\beta}(t))' \,\mathrm{d}t = \frac{\bar{g}^{\beta}(v(x)) - \bar{g}^{\beta}(v(y))}{(p-1)\beta} \leq \frac{\bar{g}^{\beta}(v(x))}{(p-1)\beta}
\end{equation*}
when $\beta>0$ and
\begin{equation*}
\int_{v(y)}^{v(x)} \frac{\bar{g}^{\beta}(t)}{t} \,\mathrm{d}t \leq - \frac{1}{(p-1)|\beta|} \int_{v(y)}^{v(x)} (\bar{g}^{\beta}(t))' \,\mathrm{d}t = \frac{\bar{g}^{\beta}(v(y)) - \bar{g}^{\beta}(v(x))}{(p-1)|\beta|} \leq \frac{\bar{g}^{\beta}(v(y))}{(p-1)|\beta|}
\end{equation*}
when $\beta<0$. Thus, we deduce
\begin{equation*}
J_1 \geq \frac{p(p-1)}{R^s} G(D^s\bar{u}) \left( \fint_{v(y)}^{v(x)} \frac{\bar{g}^{\beta}(t)}{t} \,\mathrm{d}t \right) (\eta(x) \land \eta(y))^q
\end{equation*}
by using \eqref{eq-pq}.

We now estimate $J_2$. Note that we have by using \eqref{eq-eta}
\begin{equation*}
J_2 \geq - \frac{2^q}{|\beta|} g(D^s\bar{u}) \bar{g}^\beta(v(x)) \frac{|\eta(x)-\eta(y)|^q}{|x-y|^s}
\end{equation*}
in both cases $\beta>0$ and $\beta<0$. We utilize \Cref{lem-G} and \eqref{eq-bar-g-comp} to have
\begin{equation*}
g(D^s\bar{u}) \leq g\left( \frac{R^sv(x)}{|x-y|^{s}} \right) \leq C \bar{g}(v(x)) \left( \frac{R^s}{|x-y|^s} \right)^{q-1}
\end{equation*}
for some $C = C(p, q) > 0$, which yields
\begin{equation*}
J_2 \geq - \frac{C}{R^s} \left( \frac{1}{|\beta|} \lor \frac{1}{|\beta|^q} \right) \bar{g}^{\gamma}(v(x)) (R^s |D^s\eta|)^q,
\end{equation*}
finishing the proof.
\end{proof}

\begin{lemma}\label{lem-ineq3}
Let $\bar{u}$ and $v$ be defined as in \eqref{eq-bar-u-v}. There exists a constant $C>0$, depending only on $p$ and $q$, such that
\begin{equation*}
|D^s \log \bar{u}|^p \leq \frac{C}{R^{sp}} \left( G(|D^s\bar{u}|) \fint_{v(y)}^{v(x)} \frac{\mathrm{d}t}{G(t)} + 1 \right)
\end{equation*}
for all $x, y \in B_R=B_R(x_0)$.
\end{lemma}

\begin{proof}
We may assume without loss of generality $\bar{u}(x)>\bar{u}(y)$. By Jensen's inequality, we have
\begin{equation}\label{eq-ineq2-1}
|D^s \log \bar{u}|^p = \left( D^sv \fint_{v(y)}^{v(x)} \frac{\mathrm{d}t}{t} \right)^p \leq \frac{1}{R^{sp}} \fint_{v(y)}^{v(x)} |D^s\bar{u}|^p \frac{\mathrm{d}t}{t^{p}}.
\end{equation}
Let us define $H(t)=G(t^{1/p})$, then by Young's inequality and \Cref{lem-H} we have
\begin{equation}\label{eq-ineq2-2}
|D^s\bar{u}|^p \frac{G(t)}{t^p} \leq H(|D^s\bar{u}|^p) + H^{\ast} \left( \frac{G(t)}{t^{p}} \right) \lesssim G(|D^s\bar{u}|) + G(t).
\end{equation}
Thus, the desired inequality follows by combining \eqref{eq-ineq2-1} and \eqref{eq-ineq2-2}.
\end{proof}

\begin{lemma}\label{lem-ineq4}
Let $\gamma \in \mathbb{R}$ and $d>0$. Let $\bar{u}$ be defined as in \eqref{eq-bar-u} and $v=(\bar{u}+d)/R^s$. There exists a constant $C>0$, depending only on $p$ and $q$, such that
\begin{equation*}
|D^s\bar{g}^{\gamma/p}(v)|^p \leq C \frac{|\gamma|^p}{R^{sp}} \left( G(|D^s\bar{u}|) \fint_{v(y)}^{v(x)} \frac{\bar{g}^{\gamma}(t)}{G(t)} \,\mathrm{d}t + \bar{g}^{\gamma}(v(x)) \lor \bar{g}^{\gamma}(v(y)) \right)
\end{equation*}
for all $x, y \in B_R=B_R(x_0)$, where $\bar{g}$ is defined as in \eqref{eq-bar-g}.
\end{lemma}

\begin{proof}
We may assume without loss of generality $\bar{u}(x) > \bar{u}(y)$. By using Jensen's inequality and \eqref{eq-bar-g-pq}, we have
\begin{equation}\label{eq-ineq3-1}
|D^s\bar{g}^{\gamma/p}(v)|^p = \left| D^sv \fint_{v(y)}^{v(x)} \frac{\gamma}{p} \bar{g}^{\gamma/p-1}(t) \bar{g}'(t) \,\mathrm{d}t \right|^p \leq C\frac{|\gamma|^p}{R^{sp}} \fint_{v(y)}^{v(x)} |D^s\bar{u}|^p \frac{\bar{g}^{\gamma}(t)}{t^p} \,\mathrm{d}t.
\end{equation}
Thus, the desired inequality follows from \eqref{eq-ineq3-1} and \eqref{eq-ineq2-2}.
\end{proof}

\begin{lemma}\label{lem-ineq5}
Assume \eqref{eq-q-1}. Let $a, b \in \mathbb{R}$, and $\delta \in (0, \frac{1}{q-1}]$.  Then
\begin{equation}\label{eq-alg3-L}
L:= \left| g(|a|)\frac{a}{|a|} - g(|b|)\frac{b}{|b|} \right| \leq \frac{1}{\delta} g^\delta(|a-b|) \left( g^{1-\delta}(|a|) + g^{1-\delta}(|b|) \right).
\end{equation}
In particular, for each $\varepsilon > 0$ there exists a constant $C=C(q, \varepsilon) > 0$ such that
\begin{equation*}
L \leq Cg(|a-b|) + \varepsilon (g(|a|) + g(|b|))
\end{equation*}
\end{lemma}

\begin{proof}
If $a \geq 0 \geq b$, then $|a-b|=a+b > \max\lbrace a, |b| \rbrace$, and hence
\begin{equation*}
L \leq g(a)+g(|b|) \leq g^\delta(|a-b|)g^{1-\delta}(a) + g^\delta(|a-b|)g^{1-\delta}(|b|).
\end{equation*}
Thus, we may assume without loss of generality that $a>b>0$. In this case, we have
\begin{equation}\label{eq-alg3-FTC}
L = g(a)-g(b) = \int_{g^\delta(b)}^{g^\delta(a)} \frac{1}{\delta} t^{1/\delta-1} \,\mathrm{d}t \leq \frac{1}{\delta} (g^\delta(a)-g^\delta(b)) g^{1-\delta}(a).
\end{equation}
Note that $t \mapsto g^\delta(t)/t = (g(t)/t^{q-1})^\delta t^{(q-1)\delta-1}$ is non-increasing, and hence
\begin{equation}\label{eq-alg3-tau}
g^\delta(a) = (a-b) \frac{g^\delta(a)}{a} + b \frac{g^\delta(a)}{a} \leq g^\delta(a-b) + g^\delta(b).
\end{equation}
The desired inequality \eqref{eq-alg3-L} follows from \eqref{eq-alg3-FTC} and \eqref{eq-alg3-tau}. The second assertion is a consequence of \eqref{eq-alg3-L} and Young's inequality.
\end{proof}

\begin{lemma}\label{lem-ineq6}
Let $w \geq 0$ in $B_{3R/2}$ and $0 \leq \eta \leq 1$. Then
\begin{align*}
g(|D^sw|)\frac{D^sw}{|D^sw|} D^s (w\eta^q)
&\geq c\,G(|D^sw|) (\eta(x) \land \eta(y))^q \\
&\quad - C G\left( \frac{w(x)\lor w(y)}{R^s} \right) ((R^s |D^s\eta|)^p \lor (R^s |D^s\eta|)^q)
\end{align*}
for all $x, y \in B_{3R/2}$, where $c$ and $C$ are positive constants depending only on $p$ and $q$.
\end{lemma}

Since the proof of \Cref{lem-ineq6} is similar to that of \Cref{lem-ineq2}, we omit the proof.

\end{appendix}

%\bibliographystyle{abbrv}
%\bibliography{references}

\end{document}